\documentclass[12pt,twoside,a4paper]{article}

\newcommand{\figVer}[1]{#1}

\usepackage[OT4]{fontenc}
\usepackage[cp1250]{inputenc}
\usepackage[dvips]{color}
\usepackage{amsfonts}
\usepackage{amsmath}
\usepackage{amsthm}
\usepackage{graphicx}
\usepackage{hyperref}
\usepackage{float}
\usepackage{url}
\usepackage{array}
\usepackage{algorithm,algpseudocode}

\newcommand{\Bem}[1]{}
\newcommand{\Niepotrzebne}[1]{}

\newcommand{\x}{\mathbf{x}}
\newcommand{\z}{\mathbf{z}}

\newcommand{\eref}[1]{(\ref{#1})}
\newcommand{\waga}{\mathfrak{w}}

\newcommand{\figaddr}[1]{#1}

\newcommand{\ImplicationsForSpectralClustering}[1]{} 

\newtheorem{theorem}{Theorem}

\begin{document}
\newcommand{\MovjTytulv}{Spectral Analysis of Laplacians of an Unweighted and Weighted Multidimensional Grid Graph\\ - Combinatorial  versus Normalized and Random Walk Laplacians }
\newcommand{\MaInstytucja}{Institute of Computer Science of the Polish Academy of Sciences\\ul. Jana Kazimierza 5, 01-248 Warszawa
Poland}
\title{
\MovjTytulv  }
\author{
Mieczys{\l}aw A. K{\l}opotek (klopotek\@ipipan.waw.pl)  
\\ \MaInstytucja
}
 
\maketitle

\begin{abstract}
In this paper we generalise the results on eigenvalues and eigenvectors 
of unnormalized (combinatorial) Laplacian of 
 two-dimensional grid presented by \cite{Edwards:2013} 
first to a grid graph of any dimension, and second also to other types of Laplacians, that is unoriented Laplacians, normalized Laplacians, and random walk Laplacians. 

While the closed-form or nearly closed form solutions to the eigenproblem 
of multidimensional grid graphs 
constitute a good test suit for spectral clustering algorithms for the case of no structure in the data, 
the multidimensional weighted grid graphs, presented also in this paper 
can serve as testbeds for these algorithms as graphs with some predefined cluster structure.  The weights permit to simulate node clusters not perfectly separated from each other.

This fact opens new possibilities for exploitation of closed-form or nearly closed form solutions 
eigenvectors and eigenvalues of graphs while testing and/or developing such algorithms and exploring their theoretical properties.   

Besides, the differences between the weighted and unweighted case allow for new insights into the nature of normalized and unnormalized Laplacians

 \end{abstract}
\noindent

\section{Introduction}
%

%

The concept of Graph Laplacians is in use for a long time now. 
An extensive overview of early research can be found in the paper \cite{Merris:1994} by Merris from the year 1994. 
That paper summarizes research on \emph{combinatorial Laplacians} 
$L=D-S$ of graphs $G$, where $S$ is the adjacency matrix of $G$, and $D$ is the (diagonal) degree matrix of $G$,
on \emph{unoriented Laplacians} $K=D+S$,   
on \emph{normalized Laplacians} 
 $\mathfrak{L}=D^{-1/2}L D^{-1/2}= I -D^{-1/2}S D^{-1/2}$ 
(called there correlation matrices) and on 
  \emph{random walk Laplacians} 
 $\mathbb{L}=LD^{-1}  = I -SD^{-1} $. 
A recent survey can be found in the booklet \cite{Gallier:2017} by Gallier, with a particular orientation towards applications in graph clustering\footnote{For another overview of spectral clustering methods, see e.g. Chapter 5 of the book \cite{STWMAKSpringer:2018}. 
}. 

Regular graph structures and their properties are of interest 
for a number of reasons, mostly for derivation of analytical graph properties \cite{Notarstefano:2012}. 
In particular Ramachandran 
and  Berman \cite{Ramachandran:2016} exploit 
  a priori knowledge of Laplacians 
of rectangular grid  in investigations 
of properties of robotic  swarms. 
Stankiewicz \cite{Stankiewicz:2017} discusses relation between
  the orientable
genus of a graph
(the minimum number of handles  to be added to the plane in order to  embed this graph without crossings)
 and the spectrum of its Laplacian. 
Cornelissen et al.  \cite{Cornelissen:2015}
investigate  gonality of   curves using grid Laplacians. 
Merris \cite{Merris:1994} reviews numerous properties of grid graph Laplacians 
from the point of view of  chemical applications. 
Cetkovic et al. \cite{Cvetkovic:1980} write about application in mechanics (membrane vibration). 
Cheung et al. \cite{Cheung:2018} elaborate applications in image processing, with a particular interest in grid structures. 


Grid graphs may have further applications. 
When developing graph clustering methods, especially those based on spectral analysis, it is good to have a well investigated class of graphs for which the exact form of eigenvalues and eigenvectors is known. 
Grid graphs can be considered as graphs without a definite structure. 
They can be therefore used as a kind of negative testbed for graph clustering methods. 
 In particular,   compressive  spectral clustering CSC \cite{Tremblay:2016} exploits the assumption of uniform distribution of eigenvalues of normalized Laplacian. This assumption can be tested extensively using the grid graphs. 

Generally, complex network research (covering areas of social network analysis, transportation network properties and many other) takes advantage of computation of various Laplacian  types for a grid structure.

This fact motivated the current research. 
In particular we developed closed-form or nearly closed-form formulas for eigenvalues and eigenvectors for the four aforementioned types of Laplacians (combinatorial, unoriented, normalized and random walk) for multidimensional grid graphs. 

A starting point for this research was the paper \cite{Edwards:2013}  by 
Edwards who 
elaborated an explicit analytical solution to the problem of 
eigenvalues and eigenvectors of 
a Laplacian of a two-dimensional rectangular grid. 
There exists a body of earlier research on the topic of closed-form solutions to the eigen-problem of Laplacians, at which we will point in Section 
\ref{sec:prework}.

In Section \ref{sec:notation} we introduce our notation.
In Section \ref{sec:CLtheorems} we present theorems describing our generalisation for unweighted combinatorial Laplacians to higher dimensional grids.
In Section \ref{secW:CLtheorems} we extend these results to weighted grid graphs.
Sections \ref{sec:UOLgeneralization} and \ref{secW:UOLgeneralization} describe  our generalisation to resp. unweighted and weighted unoriented Laplacians. 
Sections \ref{sec:NLtheorems} and \ref{secW:NLtheorems} handle  the generalization for normalized Laplacians of resp. unweighted and weighted graphs. 
Sections \ref{sec:RWLtheorems} and \ref{secW:RWLtheorems} explain  briefly the generalization to random walk Laplacians in resp. unweighted and weighted case. 
Sections \ref{sec:COLproperties} and \ref{sec:NOLproperties} are devoted to some discussions of the elaborated closed-form and nearly-closed-form solutions to the Laplacian eigen-problems in unweighted graphs. 
Section \ref{sec:conclusions} contains some final remarks.

\subsection{Our Contribution}
\begin{itemize}
\item A new proof of formulas for eigenvalues and eigenvectors of combinatorial Laplacian of multidimensional grid graph,
based on specific features of these eigenvectors, like zero sum of components of eigenvector, and the unlimited space of indexes of eigenvalues and eigenvectors. 
\item Closed form solution of the eigen-problem for  unoriented Laplacian of multidimensional grid graph.
\item Nearly closed form solution of the eigen-problem for  normalized Laplacian of multidimensional grid graph.
\item Nearly closed form solution of the eigen-problem for  random walk Laplacian of multidimensional grid graph.
\item Showing that the eigenvalue distribution of normalized Laplacians  of multidimensional grid graph is not uniform.
\item Showing similarities between eigenvalues of   combinatorial and normalized Laplacians  of multidimensional grid graph and  pointing at differences between corresponding eigenvectors. 
\item Handling both unweighted and weighted grid graphs.
\end{itemize}

\section{Previous Work}\label{sec:prework}

Grid graphs have been subject of investigation from the point of view of eigen-problem of their Laplacians for a considerable amount of time. 

Burden and Hedstrom \cite{Burden:1972} were interested in 
the eigenvalue spectrum of combinatorial Laplacians of grid graphs 
and derived them from the continuous Laplacian equations.

Fiedler \cite{Fiedler:1973} established bounds for the second lowest eigenvalue of the combinatorial Laplacian (currently called Fiedler eigenvalue), while mentioning the formula of the Fiedler eigenvalue for the path graph. 
He also provided with a theorem allowing to combine product graph eigenvalues from component graphs. 
Based on that paper, 
Anderson and Morey 
\cite{Anderson:1985}
derived  explicit formulas for combinatorial Laplacian eigenvalues of grid graphs, without referring to the continuous analogue.

The book  \cite{Cvetkovic:1980}  by  Cetkovic et al. presents explicit solutions to the combinatorial Laplacian eigen-problem (eigenvalues and eigenvectors) of the path-graph and as a consequence by the virtue of the construction of the two-dimensional grid graph as a product of path graphs also a solution to the rectangular grid graph combinatorial Laplacian. 
No explicit solution is provided there to normalized Laplacian eigen-problem for grid graphs. 

Merris \cite{Merris:1994} recalls a number of previous results relevant to grid graphs, including e.g. his Theorem 2.21 (due to Fiedler \cite{Fiedler:1973}) on combinatorial Laplacian eigenvalue composition for graph products (a grid graph being a product of path graphs), and also for other special graphs, like tree graphs.  
No explicit solution is provided there to normalized Laplacian eigen-problem for grid graphs. The author states only that the eigenvalues are real.

Spielman \cite{Spielman:2004} proves explicit formulas for eigenvalues and eigenvectors for path graphs and grid graphs, without, however, caring about eigenvalues with multiplicity.

Fan et al. \cite{Fan:2008} tackle the issue of unoriented Laplacians for bicyclic graphs. 
Though not directly connected to the problem of multidimensional grid graphs, the paper nonetheless points at the way how the eigen-problem for unoriented Laplacians may be decomposed. We took advantage of this idea. 

Edwards \cite{Edwards:2013} investigated two-dimensional grid graphs 
and developed  an explicit analytical solution to the problem of 
eigenvalues and eigenvectors without referring to the continuous Laplacian.  He has demonstrated that his formulas identify an orthogonal basis also in case of ties in the set of eigenvalues. 
In this paper we generalise his formulas in two directions: first for grids of higher dimensionality and second for other types of Laplacians, that is unoriented, normalized and random walk Laplacians.

In the current paper we are interested in both the unweighted and the a weighted version of the grid graphs, 
that is also in graphs with different edge weights along various dimensions of the grid (but with the same weight in a given direction). 

While these new types of grid graphs are still very restrictive in structure, they nevertheless allow for providing graphs with some kind of predefined clustering 
so that they are more suitable for consideration when testing clustering algorithms. 

Note that weighted graphs are considered as an important research topic in the area of spectral clustering. 
For example, 
Ng et al. \cite{Ng:2002} consider cases when the low weight  edges can be neglected. 
Belkin et al. \cite{Belkin:2001} consider graphs with weights driven by so-called "heat kernel". 
Jordan et al. \cite{Jordan:2003} 
discuss the issue of interpreting $k$-means  cost function as 
a weighted distortion measure for $k$-means spectral relaxation.  
Dhillon et al. \cite{Dhillon:2005, Dhillon:2007} consider graph weighting as a way to formulate generalized transitions between 
application of kernel $k$-means, spectral clustering and graph cuts algorithms. 

Therefore, the weighted grid graphs may serve as a valuable support in testing the assumptions and the performance of such algorithms and similar ones.


\section{Notation}\label{sec:notation}

A neighbourhood  matrix $S$ of any graph 
shall be defined as a matrix with entries 
$s_{jk}>0$ if there is a link between nodes $j,k$, 
and otherwise it is equal $0$. 
We assume that always $s_{jj}=0$. 
$s_{jk}$ is considered as a weight of the link (edge) between nodes $j,k$, being deemed as a kind of similarity between the nodes. 
However, by setting $s_{jj}=0$, this is not strictly a similarity measure.
If either  $s_{jk}=1$ or $s_{jk}=0$, we will talk about unweighted graph, otherwise about a weighted one.

An unnormalised (combinatorial) Laplacian $L$ of the same graph is defined as 
$$L=D-S$$ 
where $D$ is the diagonal matrix with $d_{jj}=\sum_{k=1}^ns_{jk}$ for each $j=1,\dots n$. 
An unoriented Laplacian $K$ of a graph is defined as: 
$$K=D+S$$

A normalized Laplacian $\mathfrak{L}$ of a graph  is defined as 
$$\mathfrak{L}=D^{-1/2}L D^{-1/2}= I -D^{-1/2}S D^{-1/2}$$

A random walk Laplacian $\mathbb{L}$  of a graph  is defined as 
$$\mathbb{L}=LD^{-1}  = I -SD^{-1} $$

Note that in general eigenvalues of $\mathbb{L}$ and $\mathfrak{L}$ are identical, while they differ from those of  $L$. 
On the other hand,  the eigenvectors differ in each case. 
However, eigenvectors of random walk Laplacian can be easily derived from those of normalized Laplacian. Let $\mathbf{v}$ be the eigenvector of $\mathfrak{L}$ with eigenvalue $\lambda$. 
$$\lambda \mathbf{v}= \mathfrak{L}  \mathbf{v}$$
$$\lambda \mathbf{v}= D^{-1/2}L D^{-1/2}  \mathbf{v}$$
$$\lambda \mathbf{v}= D^{-1/2}L D^{-1} D^{1/2}  \mathbf{v}$$
$$\lambda  D^{1/2} \mathbf{v}=L D^{-1} D^{1/2}  \mathbf{v}$$
$$\lambda  D^{1/2}   \mathbf{v}=\mathbb{L} D^{1/2}  \mathbf{v}$$
Hence   $ D^{1/2}\mathbf{v}$ is the eigenvector of $\mathbb{L}$ for the eigenvalue $\lambda$. 
Therefore we will not consider them separately.   

The eigenvalues of $L$ and $K$ will also differ unless we have to do with a bipartite graph which is the case with a grid graph. We will exploit this fact also.  

 A two-dimensional grid graph \cite{Wolfram:2017},
(called also a square grid graph, or rectangular grid graph, or $m\times n$ grid)
is  an $m\times  n$ lattice graph $G_{(m,n)}$,
meaning  the graph Cartesian product $P_m \times   P_n$
 of path graphs on $m$ and $n$ vertices resp. 
A generalized unweighted grid graph can also be defined as 
$G_{(n_1)}$ being a path graph of $n_1$ vertices,
and the $d$ dimensional grid graph
$G_{(n_1,\dots,n_{d})}$ being 
the graph Cartesian product 
$G_{(n_1,\dots,n_{d-1})}  \times   P_{n_d}$

So a $d$-dimensional unweighted grid graph is uniquely defined by a \emph{grid graph identity vector} 
$[n_1,...,n_d]$ where $n_j$ is the number of layers in the $j$th dimension. 

In this paper we go beyond the concept of unweighted grid graphs. 
Let us define a weighted  
  generalized grid graph   as 
$G_{(n_1)(\waga_1)}$ being a weighted path graph of $n_1$ vertices with weight $\waga_1$ for any link in this graph,
and the $d$ dimensional weighted grid graph
$G_{(n_1,\dots,n_{d})(\waga_1,\dots,\waga_{d})}$ being 
the weighted graph Cartesian product 
$G_{(n_1,\dots,n_{d-1})(\waga_1,\dots,\waga_{d-1})}  \times   G_{(n_d)(\waga_d)}$

Thus a $d$-dimensional weighted grid graph is uniquely defined by a \emph{weighted grid graph identity vector pair} 
$[n_1,...,n_d][\waga_1,...,\waga_d]$ where $n_j$ is the number of layers in the $j$th dimension and   $\waga_j$ is the weight of linkis between layers in the $j$th dimension.

Following  \cite{Edwards:2013}, 
let us introduce a special way of assigning (integer) identities to unweighted
grid graph $G_{(n_1,\dots,n_{d})}$ nodes
or weighted grid graph $G_{(n_1,\dots,n_{d})(\waga_1,\dots,\waga_{d})}$ nodes. 
The \emph{node identity numbers} run consecutively from 1 to 
$\prod_{j=1}^d n_j$.
Each node identity number $i$ is uniquely associated with 
a \emph{node identity vector} $\x=[x_1,\dots,x_d]$ 
via the (invertible) formula:
$$i=1+\sum_{j=1}^d (x_j-1) \cdot \prod_{k=j+1}^d n_k$$
Let $\mathbf{i}(i)$ be a function turning the node identity number $i$ to the corresponding \emph{node identity vector} $\x$.

A node with identity vector $[x_1,\dots,x_d]$ is connected for each $j$ 
with the node $[x_1,\dots,x_j-1,x_d]$ if $x_j>1$ 
and with node $[x_1,\dots,x_j+1,x_d]$ if $x_j<n_j$ and there are no other connections in the graph. 

We will index the eigenvalues and the corresponding eigenvectors with an \emph{eigen identity vector} of $d$ integers $\z=[z_1,\dots,z_d]$. 
You will easily see, however, that in all cases increasing/decreasing a $z_j$ by $2n$ will leave any eigenvalue and eigenvector unchanged. 
Also replacing $z_j$ with $-z_j$ (occasionally together with replacing the corresponding shift $\delta$ with $-\delta$ to be explained later) will leave eigenvalue and eigenvector unchanged. So the value range of $z_j$  can be easily reduced to the range $[0,n_j]$. 
For some technical reasons,  to be visible later,
we are subsequently interested only in the range $[-n_j+1,2n_j-1]$ for $z_j$.

Consider  the similarity matrix $S$ of the unweighted grid graph 
$G_{(n_1,\dots,n_{d})}$. It is a  
$(\prod_{j=1}^d n_j)\times(\prod_{j=1}^d n_j)$ matrix 
with $s_{il}=1$ if nodes with identities $i,l$ are connected and $s_{il}=0$ otherwise. 
The similarity matrix $S$ of the weighted grid graph 
$G_{(n_1,\dots,n_{d})(\waga_1,\dots,\waga_{d})}$ differs from this as follows:  It is a  
$(\prod_{j=1}^d n_j)\times(\prod_{j=1}^d n_j)$ matrix 
with $s_{il}=\waga_j $ if nodes with identities $i,l$ are connected and their connection is in dimension $j$ and $s_{il}=0$ otherwise.

Let $n=\prod_{j=1}^d n_j$ for simplicity.

\section{Combinatorial Laplacians of Unweighted Grid Graphs}\label{sec:CLtheorems}
In this section we will demonstrate that, in case of Combinatorial Laplacians of a $d$-dimensional unweighted grid graph, the eigenvalues are of the form described by formula \eref{eq:lambdaC} and the corresponding eigenvecvtors have the form \eref{eq:coLvector}, that is that there exists a closed-form solution for the eigen-problem of combinatorial Laplacian.
We will proceed as follows: 
With Theorem \ref{th:reductiontoeigenvectors}, we show that the components of the eigen identity vector can be reduced to the range of $[0,n_j-1]$, because outside of this range the vectors described by formula  \eref{eq:coLvector} are identical up to the sign to the vectors within this range so that they cannot constitute valid alternative eigenvectors. 
With Theorem \ref{th:eigenpairs}, we demonstrate that indeed the numbers described by formula \eref{eq:lambdaC} are eigenvalues and the vectors of the form \eref{eq:coLvector} are the corresponding eigenvectors.The proof of this theorem is based on the idea of grid graph adjacency matrix decomposition into (additive) parts related to individual directions and the auxiliary Theorem \ref{th:componenteigenpairs} is used to prove eigenvalue and eigenvectior properties for these parts. 
Finally, we need to demonstrate that we have identified all the eigenvalues and eigenvectors. As you can easily deduce, the number of eigenvalues and eigenvectors in the desired ranges of $z_j\in [0,n_j-1]$ is identical with the number of nodes in the grid graph. However, several eigenvalues can turn out to be identical for distinct eigen identity vectors. So we need to prove that these vectors are orthogonal to each other.  
We prove therefore the auxiliary Theorem \ref{th:zerosum} before proving the proper orthogonality with Theorem \ref{th:orthogonality}. We broadly exploit the trigonometric properties of the sine and cosine functions.  

The presentation below may appear as a straight forward extension of known results. However, the proofs in the presented form are worth studying because one can easily then extend them to unoriented Laplacians for which such results were not explicitly published to my knowledge. 

Let us define 
\begin{equation}\label{eq:lambdaC}
\lambda_{[z_1,\dots,z_d]}=\sum_{j=1}^d 
\left(2 \sin\left(\frac{\pi  z_j}{2 n_j}\right)\right)^2
\end{equation}
\noindent
where for each $j=1,\dots,d$ $z_j$ is an integer such that $0\le z_j\le n_j-1$.
Define 
 $\lambda_{(j,z_j)}=  
\left(2 \sin\left(\frac{\pi  z_j}{2 n_j}\right)\right)^2
$. Then 
$\lambda_{[z_1,\dots,z_d]}=\sum_{j=1}^d 
\lambda_{(j,z_j)}
$. 
Define furthermore 
\begin{equation}\label{eq:eigenvectorcomponentC}
\nu_{[z_1,\dots,z_d],[x_1,\dots,x_d]}= 
\prod_{j=1}^d  \cos\left(\frac{\pi z_j}{n_j} \left(x_j-0.5\right)\right) 
\end{equation}
\noindent
where for each $j=1,\dots,d$ $x_j$ is an integer such that $1\le x_j\le n_j$.

And finally define the $n$ dimensional vector 
 $\mathbf{v}_{[z_1,\dots,z_d]}$
such that 
 \begin{equation}\label{eq:coLvector}
 \mathbf{v}_{[z_1,\dots,z_d],i}=\nu_{[z_1,\dots,z_d],[x_1,\dots,x_d] }
\end{equation}

\begin{theorem}\label{th:reductiontoeigenvectors}
\begin{itemize}
\item
If $z_j\in  [-n_j+1,-1]$, then 
$$\mathbf{v}_{[z_1,\dots,z_j,\dots,z_d]}
= \mathbf{v}_{[z_1,\dots,z'_j,\dots,z_d]}$$
where $z'_j\in  [0,n_j-1]$, and $z'_j=-z_j$.
\item
If $z_j=   n_j$, then 
$$\mathbf{v}_{[z_1,\dots,z_j,\dots,z_d]}
= \mathbf{0}$$
\item
If $z_j\in  [n_j+1,2n_j-1]$, then 
$$\mathbf{v}_{[z_1,\dots,z_j,\dots,z_d]}
= -\mathbf{v}_{[z_1,\dots,z'_j,\dots,z_d]}$$
where $z'_j\in  [0,n_j-1]$, and $z'_j=2n_j-z_j$.
\end{itemize} 
\end{theorem}

\begin{proof}
If $z_j\in  [-n_j+1,-1]$, then obviously the transformation 
$z'_j=-z_j$ will bring $\mathbf{v}_{[z_1,\dots,z_d]}$ to the required range of interest with indexes  $[0,n_j-1]$, as 
$$\cos\left(\frac{\pi z_j}{n_j} \left(x_j-0.5\right)\right) 
= \cos\left(\frac{\pi -z_j}{n_j} \left(x_j-0.5\right)\right) 
= \cos\left(\frac{\pi z'_j}{n_j} \left(x_j-0.5\right)\right) $$

If $z_j=  n_j $, then all the entries of  $\mathbf{v}_{[z_1,\dots,z_d],[x_1,\dots,x_d]}=0$,
because
$$\cos\left(\frac{\pi z_j}{n_j} \left(x_j-0.5\right)\right) 
= \cos\left(\pi \left(x_j-0.5\right)\right) 
= 0 $$

If $z_j\in  [n_j+1,2n_j-1]$, then  the transformation 
$z'_j=2n_j-z_j$ will do the job of bringing the indexes of $\mathbf{v}_{[z_1,\dots,z_d]}$  into the desired range $[0,n_j-1]$ as 

$$\cos\left(\frac{\pi z_j}{n_j} \left(x_j-0.5\right)\right) 
= \cos\left(\frac{\pi (n_j+(z_j-n_j))}{n_j} \left(x_j-0.5\right)\right) 
$$ $$
= \cos\left(\frac{\pi  (z_j-n_j)}{n_j} \left(x_j-0.5\right)+\pi \left(x_j-0.5\right) \right) 
$$ $$
= \cos\left(\frac{\pi  (z_j-n_j)}{n_j} \left(x_j-0.5\right)\right)\cdot \cos\left(\pi \left(x_j-0.5\right) \right) 
$$ $$
- \sin\left(\frac{\pi  (z_j-n_j)}{n_j} \left(x_j-0.5\right)\right)\cdot \sin\left(\pi \left(x_j-0.5\right) \right) 
$$ $$
= - \sin\left(\frac{\pi  (z_j-n_j)}{n_j} \left(x_j-0.5\right)\right)\cdot \sin\left(\pi \left(x_j-0.5\right) \right) 
$$ $$
=  \sin\left(\frac{\pi  (z_j-n_j)}{n_j} \left(x_j-0.5\right)\right)\cdot  \left(-1\right)^{x} 
$$ $$
=  \sin\left(\frac{\pi (n_j- (2n_j-z_j))}{n_j} \left(x_j-0.5\right)\right)\cdot  \left(-1\right)^{x} 
$$ $$
=  \sin\left(\frac{\pi (- (2n_j-z_j))}{n_j} \left(x_j-0.5\right)+\pi \left(x_j-0.5\right) \right)\cdot  \left(-1\right)^{x} 
$$ $$
=  \sin\left(\frac{\pi (- (2n_j-z_j))}{n_j} \left(x_j-0.5\right)\right)\cdot\cos\left(\pi \left(x_j-0.5\right) \right)\cdot  \left(-1\right)^{x} 
$$ $$
+  \cos\left(\frac{\pi (- (2n_j-z_j))}{n_j} \left(x_j-0.5\right)\right)\cdot\sin\left(\pi \left(x_j-0.5\right) \right)\cdot  \left(-1\right)^{x} 
$$ $$
=   \cos\left(\frac{\pi (- (2n_j-z_j))}{n_j} \left(x_j-0.5\right)\right)\cdot\sin\left(\pi \left(x_j-0.5\right) \right)\cdot  \left(-1\right)^{x} 
$$ $$
=   \cos\left(\frac{\pi (- (2n_j-z_j))}{n_j} \left(x_j-0.5\right)\right)\cdot \left(-1\right)^{x+1}  \cdot  \left(-1\right)^{x} 
$$ $$
=   -\cos\left(\frac{\pi (2n_j-z_j)}{n_j} \left(x_j-0.5\right)\right) 
 $$
\end{proof}

The above  we will need later.

We claim the following
\begin{theorem}\label{th:eigenpairs}
For the Laplacian $L$
of the grid graph $G_{(n_1,\dots,n_{d})}$
for 
each 
vector of integers 
$[z_1,\dots,z_d]$ such that for each $j=1,\dots,d$ 
$0\le z_j\le n_j-1$, 
the $\lambda_{[z_1,\dots,z_d]}$ is an eigenvalue of $L$
and $\mathbf{v}_{[z_1,\dots,z_d]}$ is a corresponding eigenvector. 
\end{theorem}
\begin{proof}
Note that the similarity matrix 
$S$ 
can be expressed as the sum of similarity matrices 
$$S=\sum_{j=1}^d S_j$$
where $S_j$ is a connectivity matrix of a graph in which 
a node with identity vector $[x_1,\dots,x_d]$ is connected  
with the node $[x_1,\dots,x_j-1,x_d]$ if $x_j>1$ 
and with node $[x_1,\dots,x_j+1,x_d]$ if $x_j<n_j$ and there are no other connections in the graph. 
Let $L_j$ be the Laplacian corresponding to the similarity matrix $S_j$. 
Then clearly
$$L=\sum_{j=1}^d L_j$$
According to the subsequent theorem  \ref{th:componenteigenpairs}
$$L \mathbf{v}_{[z_1,\dots,z_d]}=(\sum_{j=1}^d L_j)\mathbf{v}_{[z_1,\dots,z_d]}
=  \sum_{j=1}^d (L_j\mathbf{v}_{[z_1,\dots,z_d]})
$$ $$
= \sum_{j=1}^d (\lambda_{(j,z_j)} \mathbf{v}_{[z_1,\dots,z_d]})
= \left(\sum_{j=1}^d \lambda_{(j,z_j)}\right) \mathbf{v}_{[z_1,\dots,z_d]} 
=  \lambda_{[z_1,\dots,z_d]} \mathbf{v}_{[z_1,\dots,z_d]} 
$$

\end{proof}
\begin{theorem}\label{th:componenteigenpairs}
For the Laplacian $L_j$, as defined above 
of the grid graph $G_{(n_1,\dots,n_{d})}$
for 
each 
vector of integers 
$[z_1,\dots,z_d]$ such that for each $k=1,\dots,d$ 
$0\le z_k\le n_k-1$, 
the $\lambda_{(j,z_j)}$ is an eigenvalue of $L_j$
and $\mathbf{v}_{[z_1,\dots,z_d]}$ is a corresponding eigenvector. 
\end{theorem}
\begin{proof}
First note that for $x_j=1$ we have 
$$-\cos\left(\frac{\pi z_j}{n_j} \left(x_j-1-0.5\right)\right) 
+\cos\left(\frac{\pi z_j}{n_j} \left(x_j-0.5\right)\right) 
$$ $$= 
-\cos\left(\frac{\pi z_j}{n_j} \left(1-1-0.5\right)\right) 
+\cos\left(\frac{\pi z_j}{n_j} \left(1-0.5\right)\right) 
$$ $$= 
-\cos\left(\frac{\pi z_j}{n_j} \left(-0.5\right)\right) 
+\cos\left(\frac{\pi z_j}{n_j} \left(+0.5\right)\right) 
$$ $$= 
-\cos\left(\frac{\pi z_j}{n_j} \left(+0.5\right)\right) 
+\cos\left(\frac{\pi z_j}{n_j} \left(+0.5\right)\right) 
=
0
$$
For $x_j=n_j$ we have 
$$+\cos\left(\frac{\pi z_j}{n_j} \left(x_j-0.5\right)\right) 
-\cos\left(\frac{\pi z_j}{n_j} \left(x_j+1-0.5\right)\right) 
$$ $$= 
+\cos\left(\frac{\pi z_j}{n_j} \left(n_j-0.5\right)\right) 
-\cos\left(\frac{\pi z_j}{n_j} \left(n_j+1-0.5\right)\right) 
$$ $$= 
+\cos\left( \pi z_j  -0.5 \pi z_j\right) 
-\cos\left( \pi z_j  +0.5 \pi z_j\right) 
$$ $$= 
+\cos\left(  -0.5 \pi z_j\right) 
-\cos\left(  +0.5 \pi z_j\right) 
$$ $$= 
+\cos\left(  +0.5 \pi z_j\right) 
-\cos\left(  +0.5 \pi z_j\right) 
=0
$$
For a given node $p$ 
with vector identity $[x'_1,\dots,x'_j ,\dots x'_d]$
consider now all the nodes that have identical 
identity vectors at all positions except for the $jth$ one. 
Consider the product 
$L_j\mathbf{v}_{[z_1,\dots,z_d]}$ at a position with $j$th coordinate equal 
$x_j$ for $x_j=1,\dots,n_j$.
It can be expressed as 
$$
-\cos\left(\frac{\pi z_j}{n_j} \left(x_j-1-0.5\right)\right)
\cdot \prod_{k=1,\dots,d, k\ne j}   \cos\left(\frac{\pi z_k}{n_k} \left(x'_k-0.5\right)\right) 
$$ $$+2\cos\left(\frac{\pi z_j}{n_j} \left(x_j-0.5\right)\right)
\cdot \prod_{k=1,\dots,d, k\ne j}   \cos\left(\frac{\pi z_k}{n_k} \left(x'_k-0.5\right)\right) 
$$ $$-\cos\left(\frac{\pi z_j}{n_j} \left(x_j+1-0.5\right)\right)
\cdot \prod_{k=1,\dots,d, k\ne j}   \cos\left(\frac{\pi z_k}{n_k} \left(x'_k-0.5\right)\right) 
$$ $$=
\left(
-\cos\left(\frac{\pi z_j}{n_j} \left(x_j-1-0.5\right)\right)
+2\cos\left(\frac{\pi z_j}{n_j} \left(x_j-0.5\right)\right)
\right.$$ $$\left.
-\cos\left(\frac{\pi z_j}{n_j} \left(x_j+1-0.5\right)\right)
\right)
\cdot \prod_{k=1,\dots,d, k\ne j}   \cos\left(\frac{\pi z_k}{n_k} \left(x'_k-0.5\right)\right) 
$$

Let us consider subsequently only the expression
$$\left(
-\cos\left(\frac{\pi z_j}{n_j} \left(x_j-1-0.5\right)\right)
+2\cos\left(\frac{\pi z_j}{n_j} \left(x_j-0.5\right)\right)
\right.$$ $$\left.
-\cos\left(\frac{\pi z_j}{n_j} \left(x_j+1-0.5\right)\right)
\right)
$$ $$=
-\cos\left(\frac{\pi z_j}{n_j} \left(x_j-0.5\right)-\frac{\pi z_j}{n_j}\right)
+2\cos\left(\frac{\pi z_j}{n_j} \left(x_j-0.5\right)\right)
$$ $$
-\cos\left(\frac{\pi z_j}{n_j} \left(x_j-0.5\right)+\frac{\pi z_j}{n_j}\right)
$$ $$=
-\cos\left(\frac{\pi z_j}{n_j} \left(x_j-0.5\right)\right)\cos\left(\frac{\pi z_j}{n_j}\right)
-\sin\left(\frac{\pi z_j}{n_j} \left(x_j-0.5\right)\right)\sin\left(\frac{\pi z_j}{n_j}\right)
$$ $$
+2\cos\left(\frac{\pi z_j}{n_j} \left(x_j-0.5\right)\right)
$$ $$
-\cos\left(\frac{\pi z_j}{n_j} \left(x_j-0.5\right)\right)\cos\left(\frac{\pi z_j}{n_j}\right)
+\sin\left(\frac{\pi z_j}{n_j} \left(x_j-0.5\right)\right)\sin\left(\frac{\pi z_j}{n_j}\right)
$$ $$=
-\cos\left(\frac{\pi z_j}{n_j} \left(x_j-0.5\right)\right)\cos\left(\frac{\pi z_j}{n_j}\right)
+2\cos\left(\frac{\pi z_j}{n_j} \left(x_j-0.5\right)\right)
$$ $$
-\cos\left(\frac{\pi z_j}{n_j} \left(x_j-0.5\right)\right)\cos\left(\frac{\pi z_j}{n_j}\right)
$$ $$=
2\cos\left(\frac{\pi z_j}{n_j} \left(x_j-0.5\right)\right)
\left(1-\cos\left(\frac{\pi z_j}{n_j}\right)\right)
$$ $$=
2\cos\left(\frac{\pi z_j}{n_j} \left(x_j-0.5\right)\right)
\left(2\sin^2\left(\frac{\pi z_j}{2n_j}\right)\right)
$$ $$=
\left(2\sin\left(\frac{\pi z_j}{2n_j}\right)\right)^2
\cos\left(\frac{\pi z_j}{n_j} \left(x_j-0.5\right)\right)
$$ $$=
\lambda_{(j,z_j)} 
\cos\left(\frac{\pi z_j}{n_j} \left(x_j-0.5\right)\right)
$$
This means that 
$$
-\cos\left(\frac{\pi z_j}{n_j} \left(x_j-1-0.5\right)\right)
\cdot \prod_{k=1,\dots,d, k\ne j}   \cos\left(\frac{\pi z_k}{n_k} \left(x'_k-0.5\right)\right) 
$$ $$+2\cos\left(\frac{\pi z_j}{n_j} \left(x_j-0.5\right)\right)
\cdot \prod_{k=1,\dots,d, k\ne j}   \cos\left(\frac{\pi z_k}{n_k} \left(x'_k-0.5\right)\right) 
$$ $$-\cos\left(\frac{\pi z_j}{n_j} \left(x_j+1-0.5\right)\right)
\cdot \prod_{k=1,\dots,d, k\ne j}   \cos\left(\frac{\pi z_k}{n_k} \left(x'_k-0.5\right)\right) 
$$ $$=
\lambda_{(j,z_j)} 
\cos\left(\frac{\pi z_j}{n_j} \left(x_j-0.5\right)\right)
\cdot \prod_{k=1,\dots,d, k\ne j}   \cos\left(\frac{\pi z_k}{n_k} \left(x'_k-0.5\right)\right) 
$$
that is 
$\lambda_{(j,z_j)} $ times its position in the $\mathbf{v}$ vector.
And it happens so for any node with any index.
So the claim of the theorem is demonstrated. 
\end{proof}

Let us now establish that all eigenvectors
\footnote{
It is well known that eigenvectors
 associated with different eigenvalues are orthogonal, but in the multidimensional grid not all eigenvalues need to be different} are orthogonal to one another. 

But first 
note  that
\begin{theorem}\label{th:zerosum}
The sum of elements of each of the above-mentioned  eigenvectors is zero
except for $\mathbf{v}_{[0,\dots,0]}$. 
\end{theorem}
\begin{proof}
Just consider the product $L_j\mathbf{v}_{[z_1,\dots,z_d]}$. 
For a given node $p$ 
with vector identity $[x'_1,\dots,x'_j ,\dots x'_d]$
consider now all the nodes that have identical 
identity vectors at all positions except for the $jth$ one 
and compute their sum.
The transformation $L_j$ transforms this sum by a factor $\lambda_{(j,z_j)} $ 
and lets each node in each  pair  of consecutive nodes  occur once with positive sign (when its own transformation is computed) and once with negative sign (when the other  node transformation is computed). 
This means in practice that the contributions cancel out one another so that the sum is equal 0. 
\end{proof}

\begin{theorem}\label{th:orthogonality}
Any two eigenvectors 
$\mathbf{v}_{[z1_1,\dots,z1_d]},  \mathbf{v}_{[z2_1,\dots,z2_d]}$
 such that the index 
$[z1_1,\dots,z1_d]$ is not identical with 
$[z2_1,\dots,z2_d]$
are orthogonal. 
\end{theorem}
\begin{proof}
Consider the product of two eigenvectors 
$\mathbf{v}_{[z1_1,\dots,z1_d]},  \mathbf{v}_{[z2_1,\dots,z2_d]}$
corresponding to distinct eigenvalues. 
$\lambda_{[z1_1,\dots,z1_d]},  \lambda_{[z2_1,\dots,z2_d]}$.
If $[z1_1,\dots,z1_d]$ equals $[0,\dots,0]$, then the corresponding eigenvector is constant 
so that the dot product of with eigenvectors is equal to the second one times a constant.  
As the sum of elements of a vector is zero, so is this dot product. 
Otherwise 
let us have a look at a node with identity vector 
$[x_1,\dots,x_d]$. The dot product at this node will have the contribution to the overall dot product equal
$$
\prod_{j=1}^d 
 \cos\left(\frac{\pi z1_j}{n_j} \left(x_j-0.5\right)\right) 
 \cos\left(\frac{\pi z2_j}{n_j} \left(x_j-0.5\right)\right) 
$$
$$
=0.5^d \prod_{j=1}^d 
\left( \cos\left(\frac{\pi z1_j}{n_j} \left(x_j-0.5\right)+\frac{\pi z2_j}{n_j} \left(x_j-0.5\right)\right) 
\right.$$ $$\left.
+ \cos\left(\frac{\pi z1_j}{n_j} \left(x_j-0.5\right)-\frac{\pi z2_j}{n_j} \left(x_j-0.5\right)\right) 
\right)
$$
$$
=0.5^d \prod_{j=1}^d 
\left(  \cos\left(\frac{\pi (z1_j+z2_j)}{n_j} \left(x_j-0.5\right)\right) 
+ \cos\left(\frac{\pi (z1_j-z2_j)}{n_j} \left(x_j-0.5\right) \right) 
\right)
$$
After multiplying the sums out 
we get a sum  of components 
of $\mathbf{v}$ vectors with indexes ranging from $-n_j+1$ to $2n_j-2$,
which according to the theorem \ref{th:reductiontoeigenvectors} can be transformed to eigenvectors of $L$ or are identical with $\mathbf{0}$.  
As both vector identities are different, none of the eigenvector indices never will have the form
$[0,\dots,0]$, hence they sum up to 0. 
This finishes the proof.
\end{proof}

As all eigenvectors computed by our formulas are orthogonal, 
and the index vectors exhaust the number of nodes,
then the list of eigenvectors and eigenvalues is complete.

\section{Some Properties of Combinatorial Laplacian for a multidimensional unweighted grid graph}\label{sec:COLproperties}

The formula \eref{eq:lambdaC} implies that the combinatorial Laplacian 
ranges from 0 to $2^d$, where $d$ is the dimensionality of the grid graph. 
The upper bound is approached with increase of the lowest number of layers in any dimension ($\min_j n_j \rightarrow \infty$).

Another interesting aspect of the grid graph eigenvalues is whether or not they are uniformly distributed. Though, as mentioned in the Introduction, Compressive Spectral Cluster Analysis is rather interested in uniformity in case of normalized Laplacians, let us nonetheless consider this property for combinatorial Laplacians. 
The formula for computing the eigenvalue denies at an inspection this property. 
\figVer{
But let us investigate this visually. 

\begin{figure}
\centering
 (a)\includegraphics[width=0.4\textwidth]{\figaddr{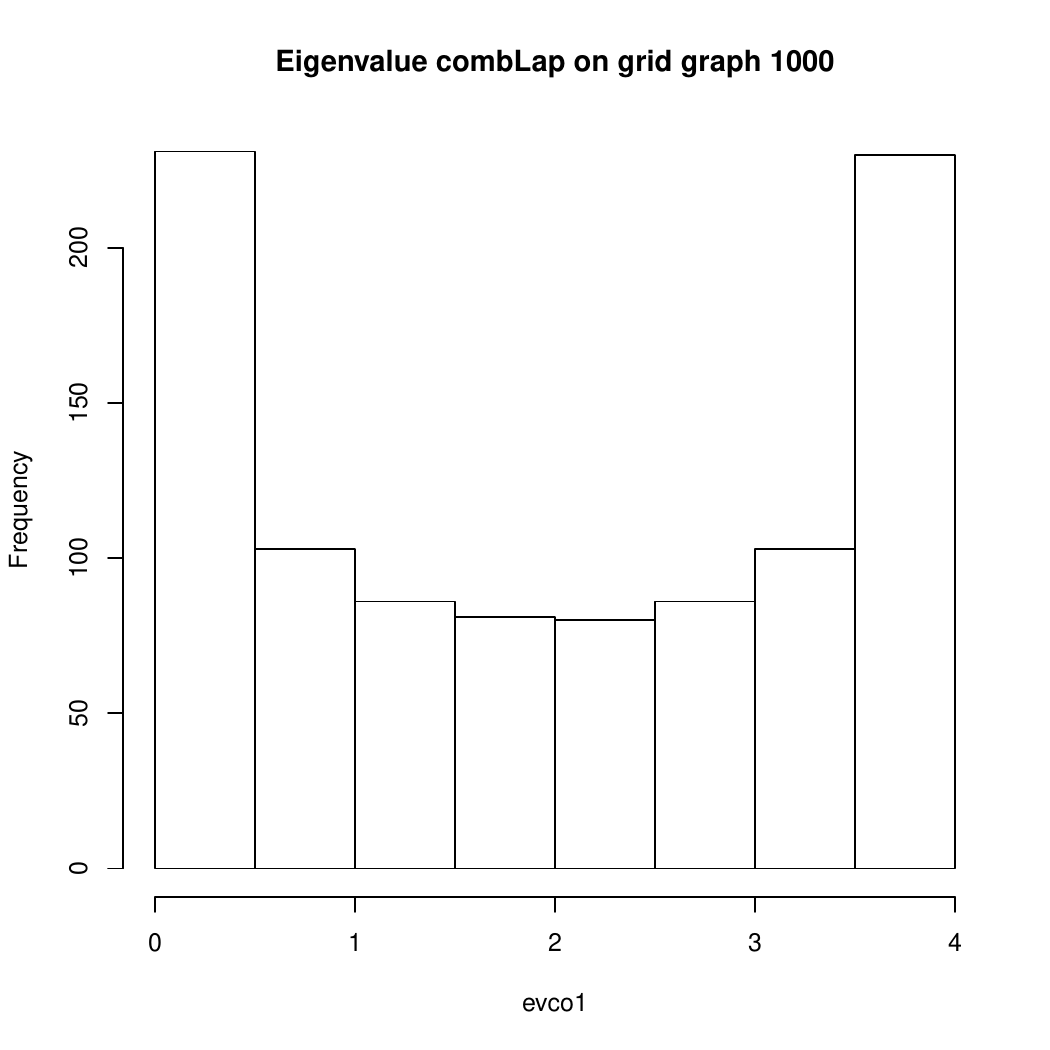}}  %
 (b)\includegraphics[width=0.4\textwidth]{\figaddr{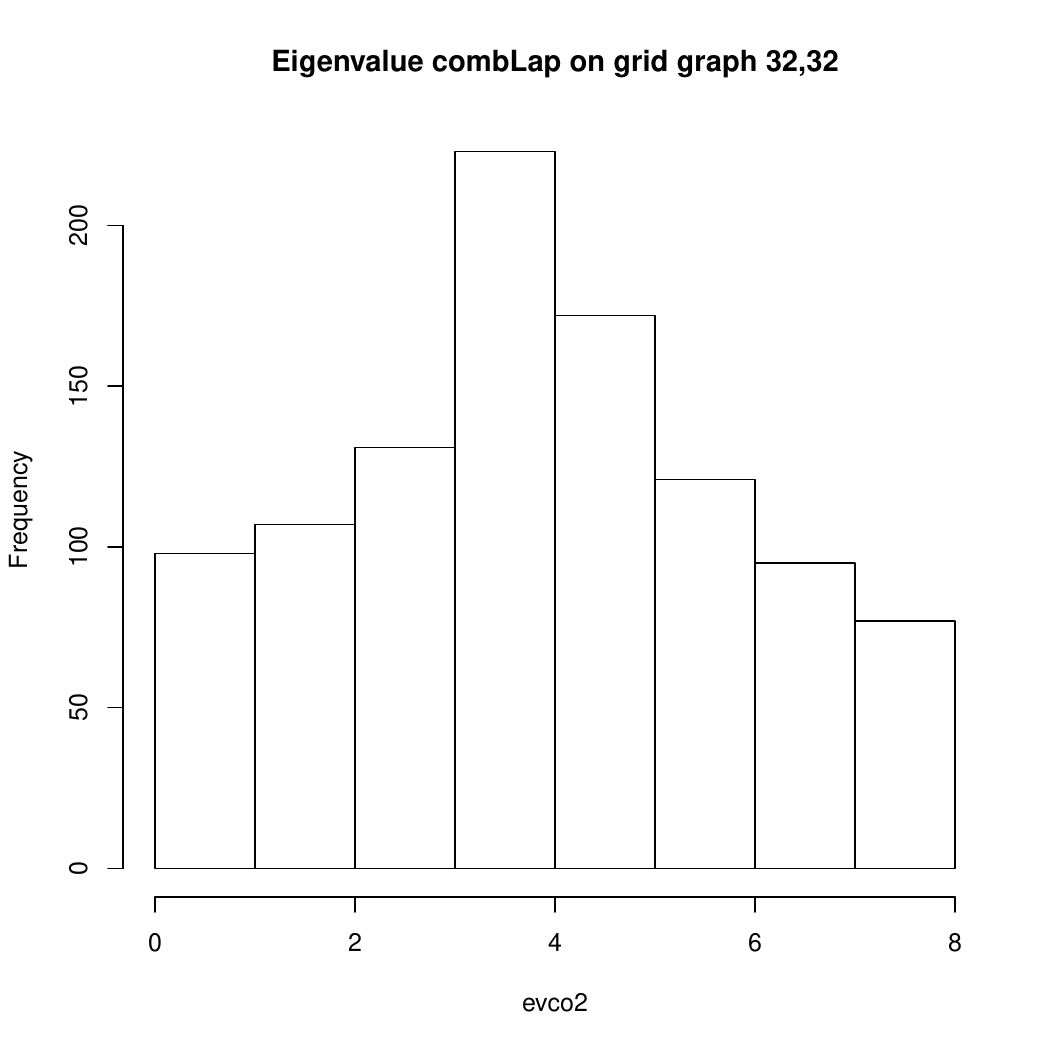}}  %
\\ (c)\includegraphics[width=0.4\textwidth]{\figaddr{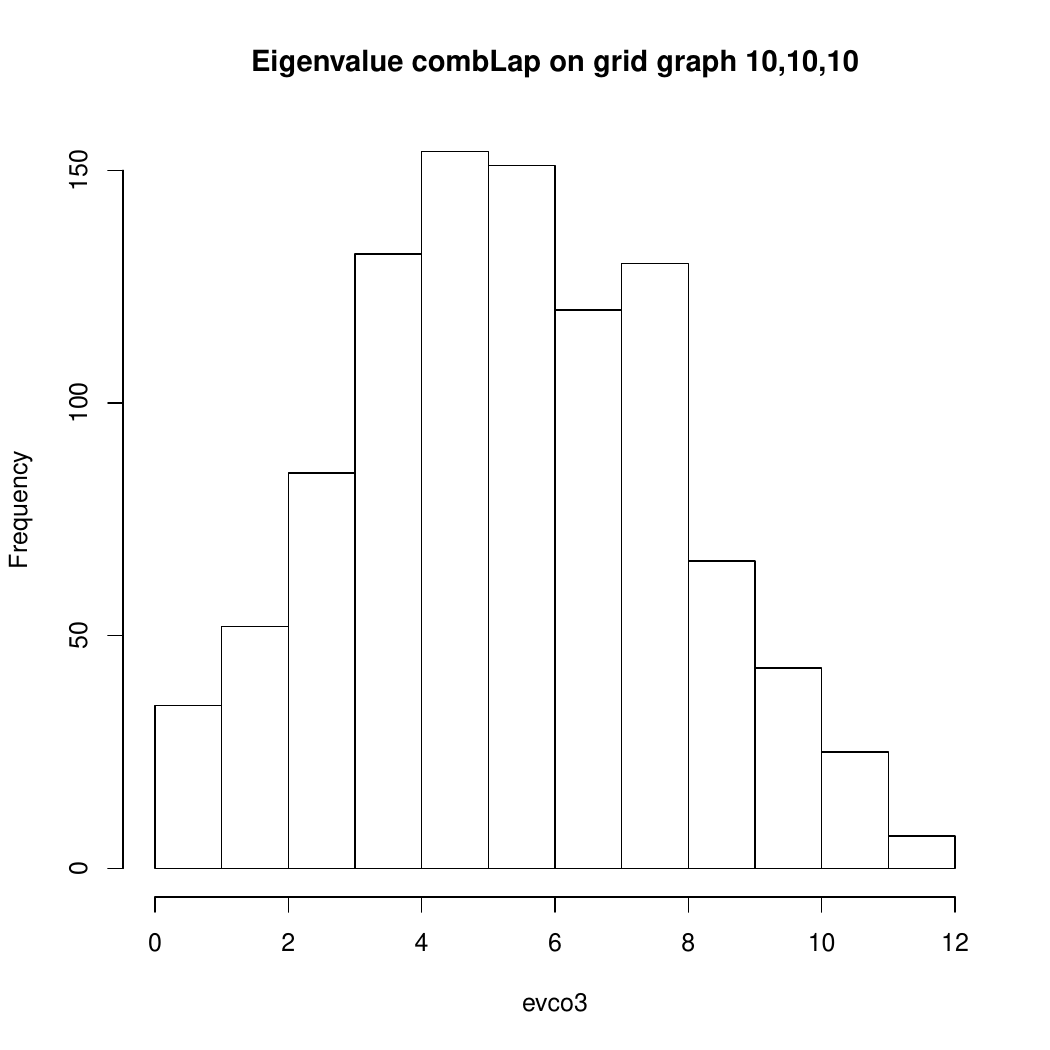}}  %
 (d)\includegraphics[width=0.4\textwidth]{\figaddr{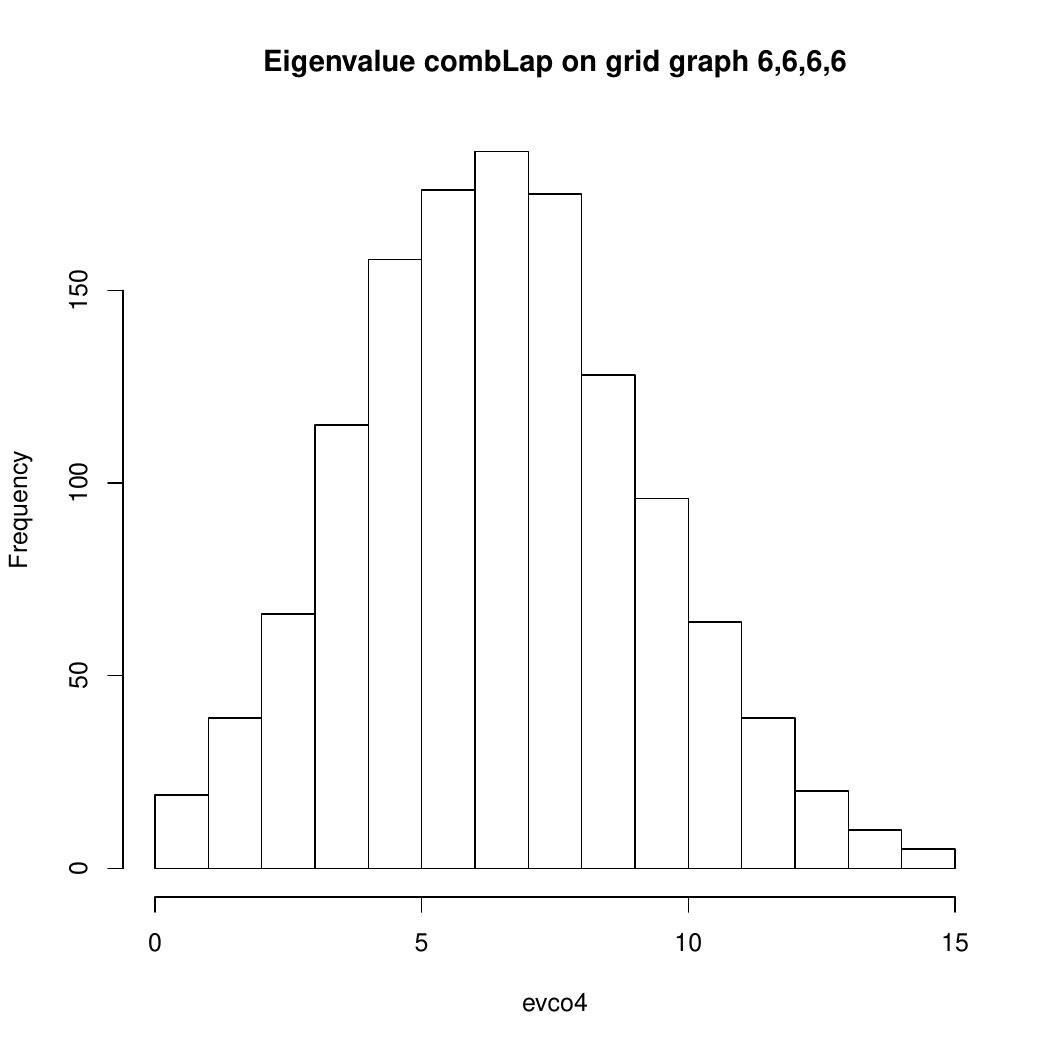}}  %
\caption{The histograms of  eigenvalues of combinatorial Laplacians of grid graphs of approximately 1,000 nodes.
(a) 1-dimensional grid graph, 
(b) 2-dimensional grid graph, 
(c) 3-dimensional grid graph, 
(d) 4-dimensional grid graph. 
}\label{fig:evco1000nodes}
\end{figure}

In Figure \ref{fig:evco1000nodes} you see the histograms of eigenvalue for grid graphs of approximately 1,000 nodes with dimensionality ranging between 1 and 4. 
 Figure \ref{fig:evco10000nodes} depicts analogous   histograms for 10,000 node graphs.  
Obviously, the shapes of histograms for 1,000 nodes and 10,000 nodes are similar and they are in no way uniform, at least for 1 to 4-dimensional grid graphs. 

To look deeper into these issues, a cumulative distribution function of eigenvalues of combinatorial Laplacian of 1,2,3,4,7,8-dimensional grid graphs with the number of nodes "in the limit" have been computed and depicted in Figure  \ref{fig:coLdist}.
A blue line was added in each diagram to indicate how a uniform distribution would have looked like. 
The multidimensional grid graph exhibits no similarity to uniform eigenvalue distribution though it is structureless. 

\begin{figure}
\centering
 (a)\includegraphics[width=0.4\textwidth]{\figaddr{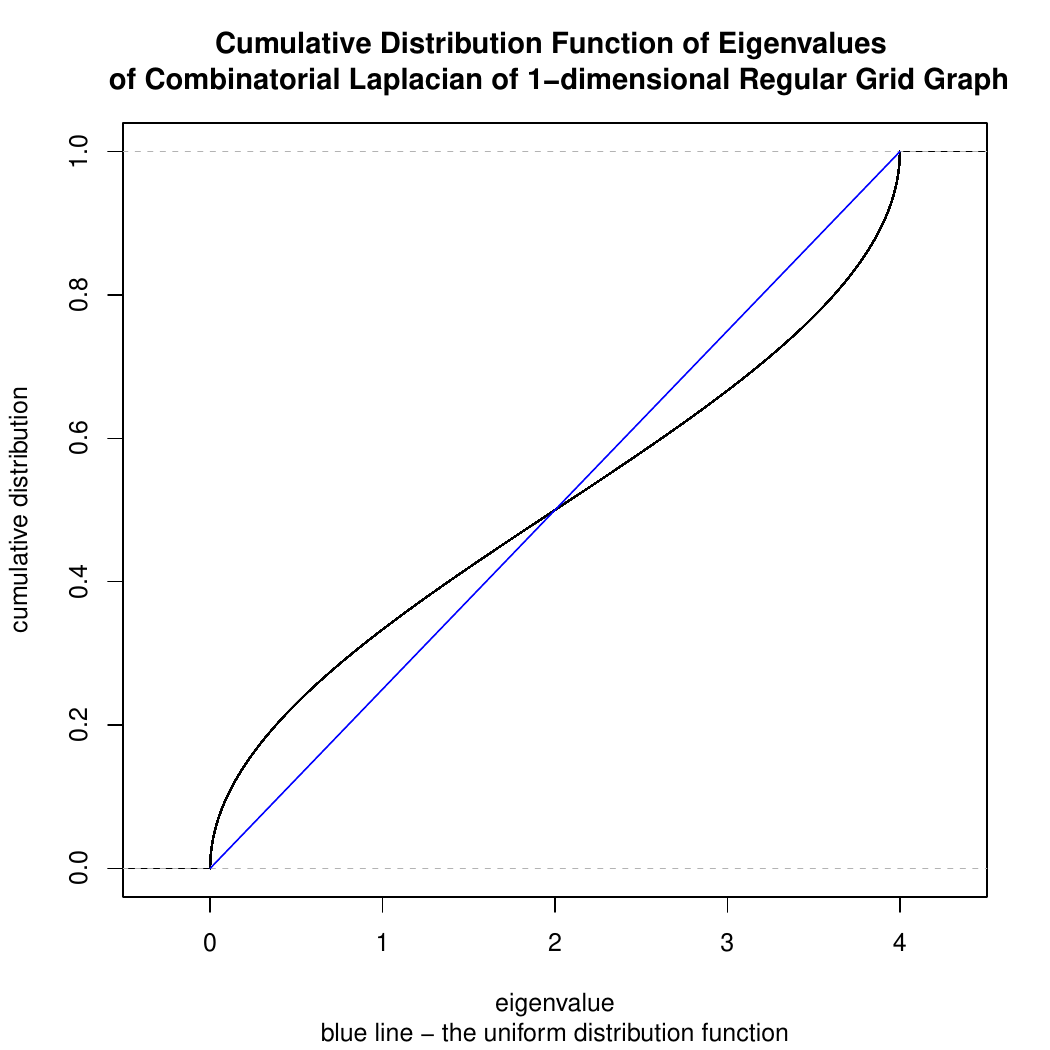}}  %
 (b)\includegraphics[width=0.4\textwidth]{\figaddr{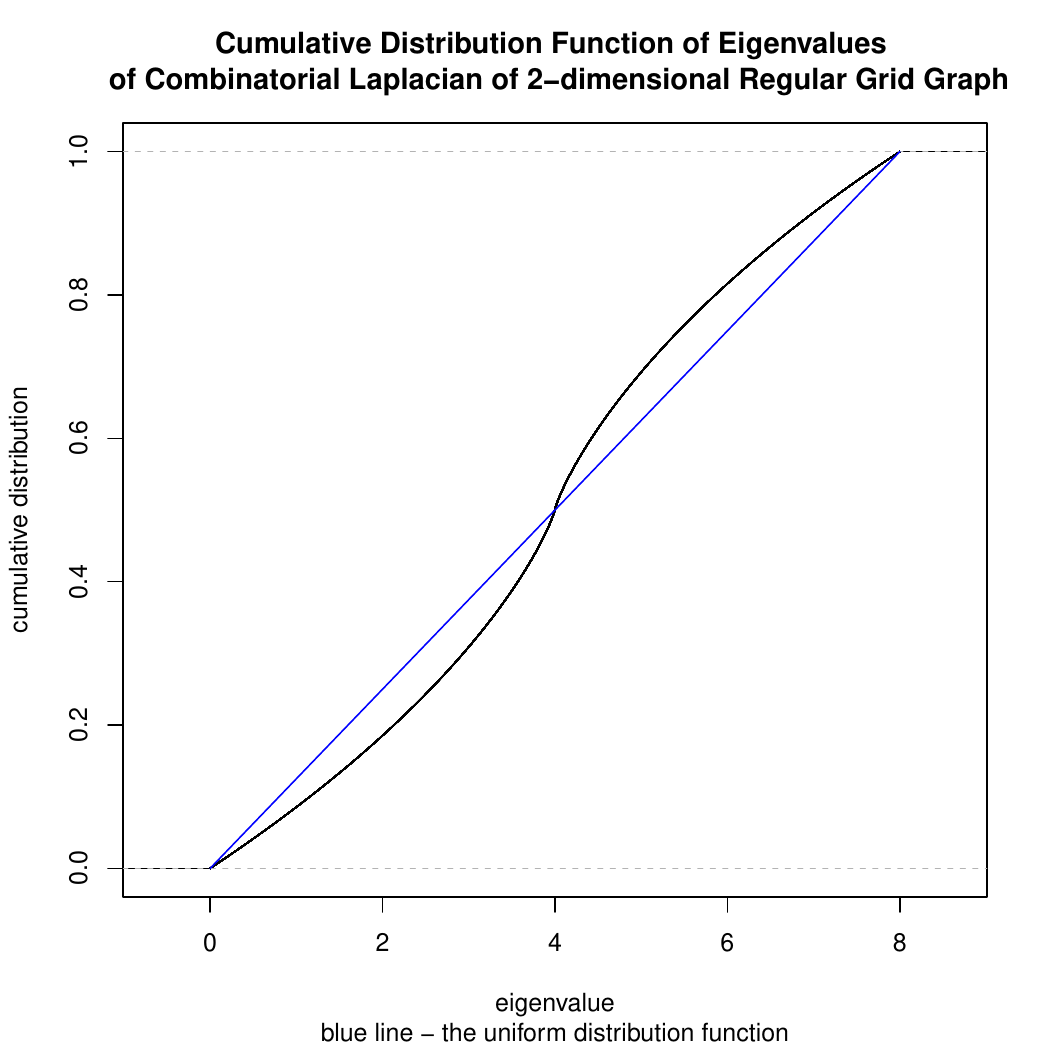}}  %
\\ (c)\includegraphics[width=0.4\textwidth]{\figaddr{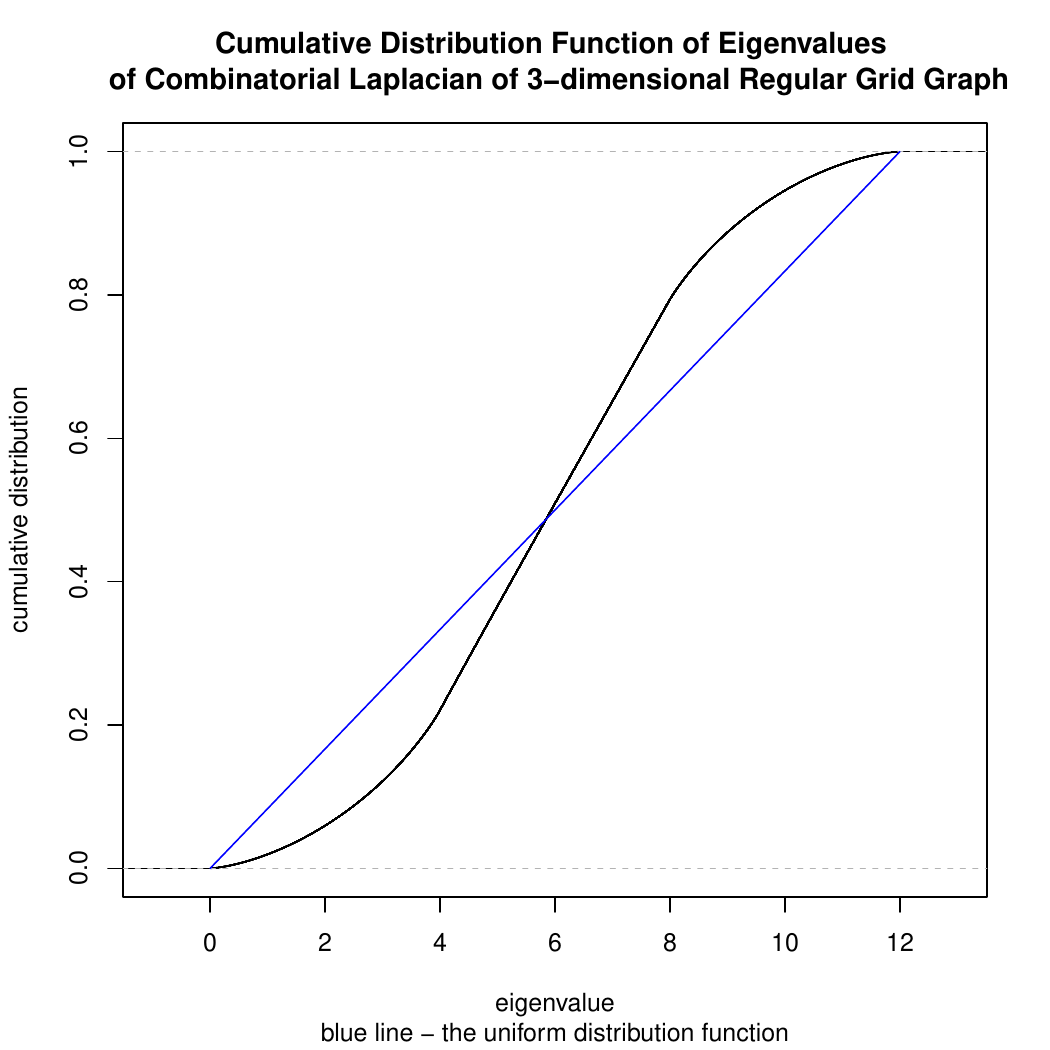}}  %
 (d)\includegraphics[width=0.4\textwidth]{\figaddr{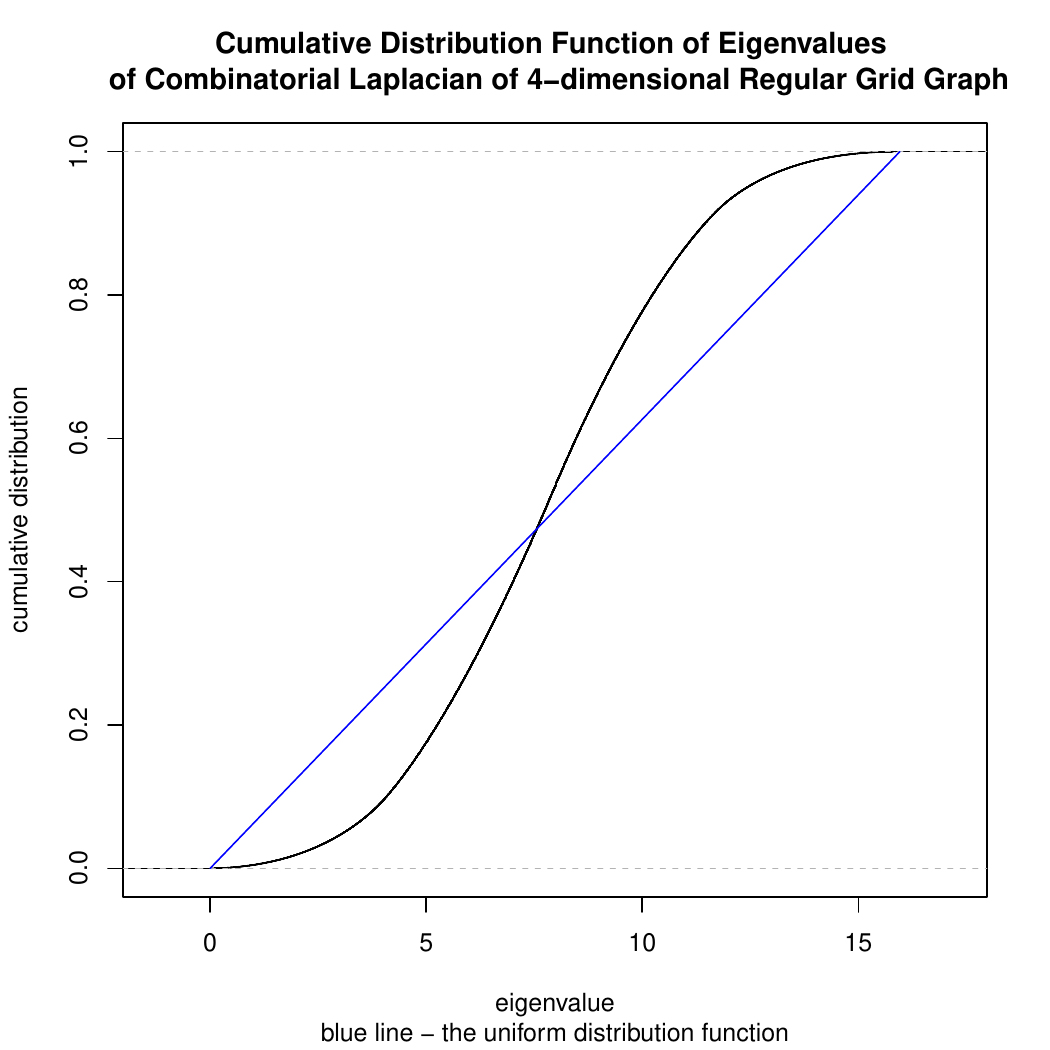}}  %
\\ (e)\includegraphics[width=0.4\textwidth]{\figaddr{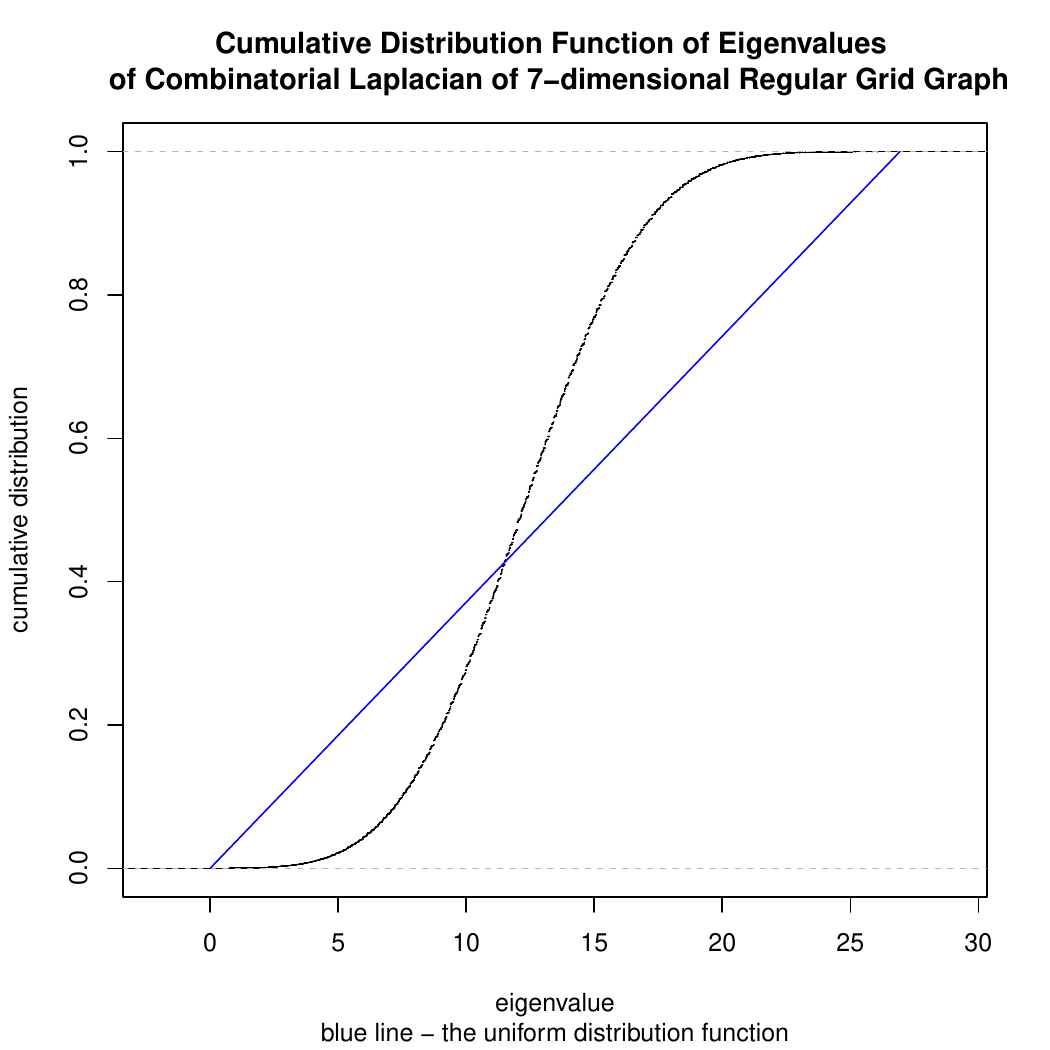}}  %
 (f)\includegraphics[width=0.4\textwidth]{\figaddr{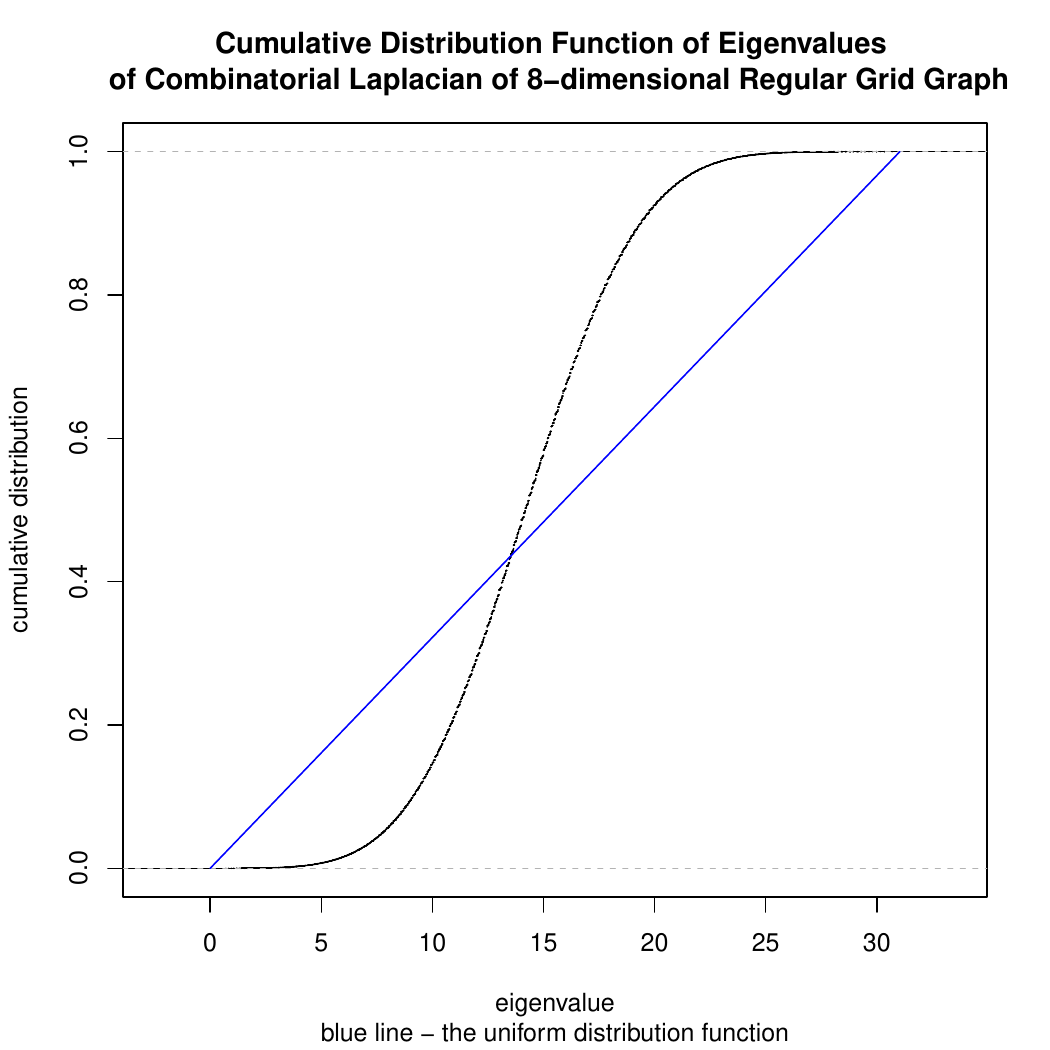}}  %
\caption{The in the limit cumulative distribution function  of  eigenvalues of combinatorial Laplacians of grid graphs.
(a) 1-dimensional grid graph, 
(b) 2-dimensional grid graph, 
(c) 3-dimensional grid graph, 
(d) 4-dimensional grid graph, 
(c) 7-dimensional grid graph, 
(d) 8-dimensional grid graph. 
}\label{fig:coLdist}
\end{figure}

\begin{figure}
\centering
 (a)\includegraphics[width=0.4\textwidth]{\figaddr{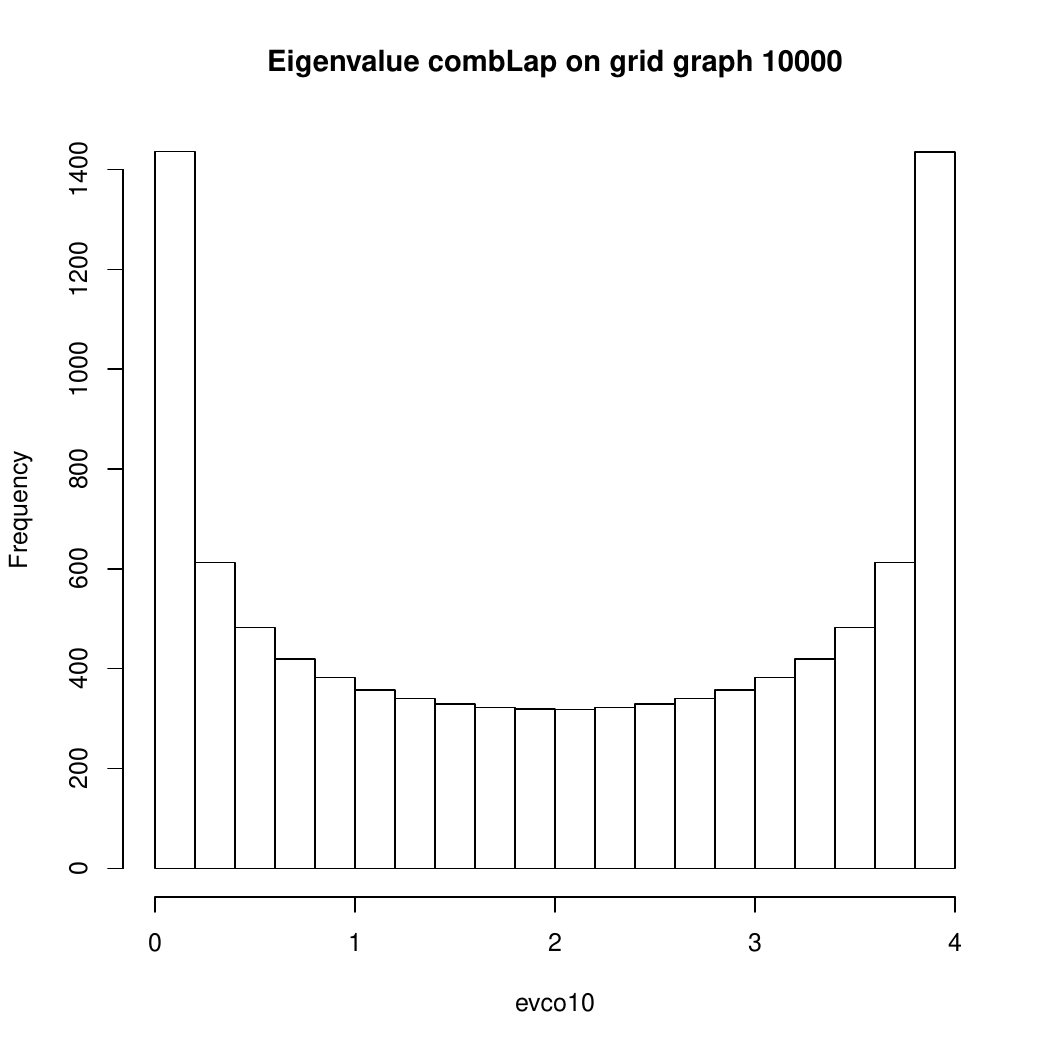}}  %
 (b)\includegraphics[width=0.4\textwidth]{\figaddr{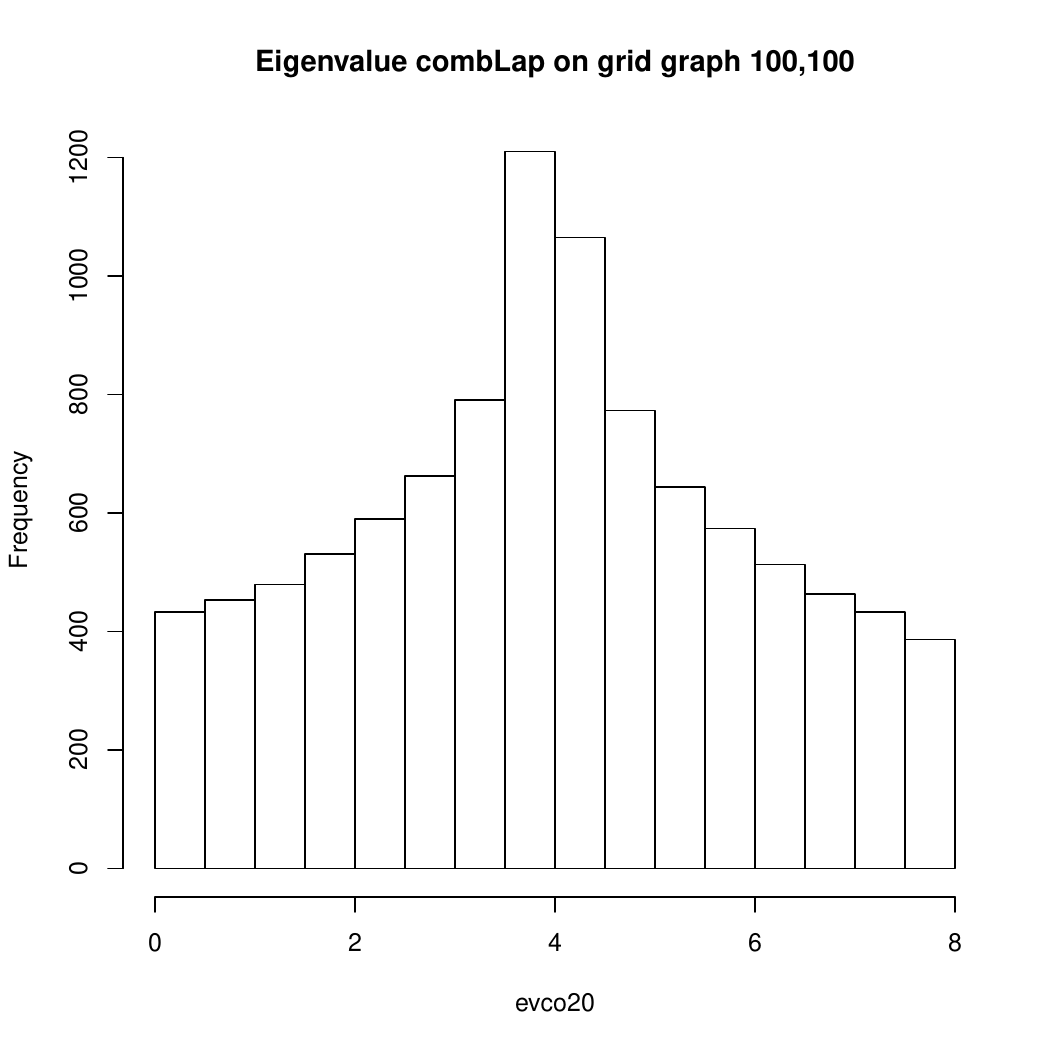}}  %
\\ (c)\includegraphics[width=0.4\textwidth]{\figaddr{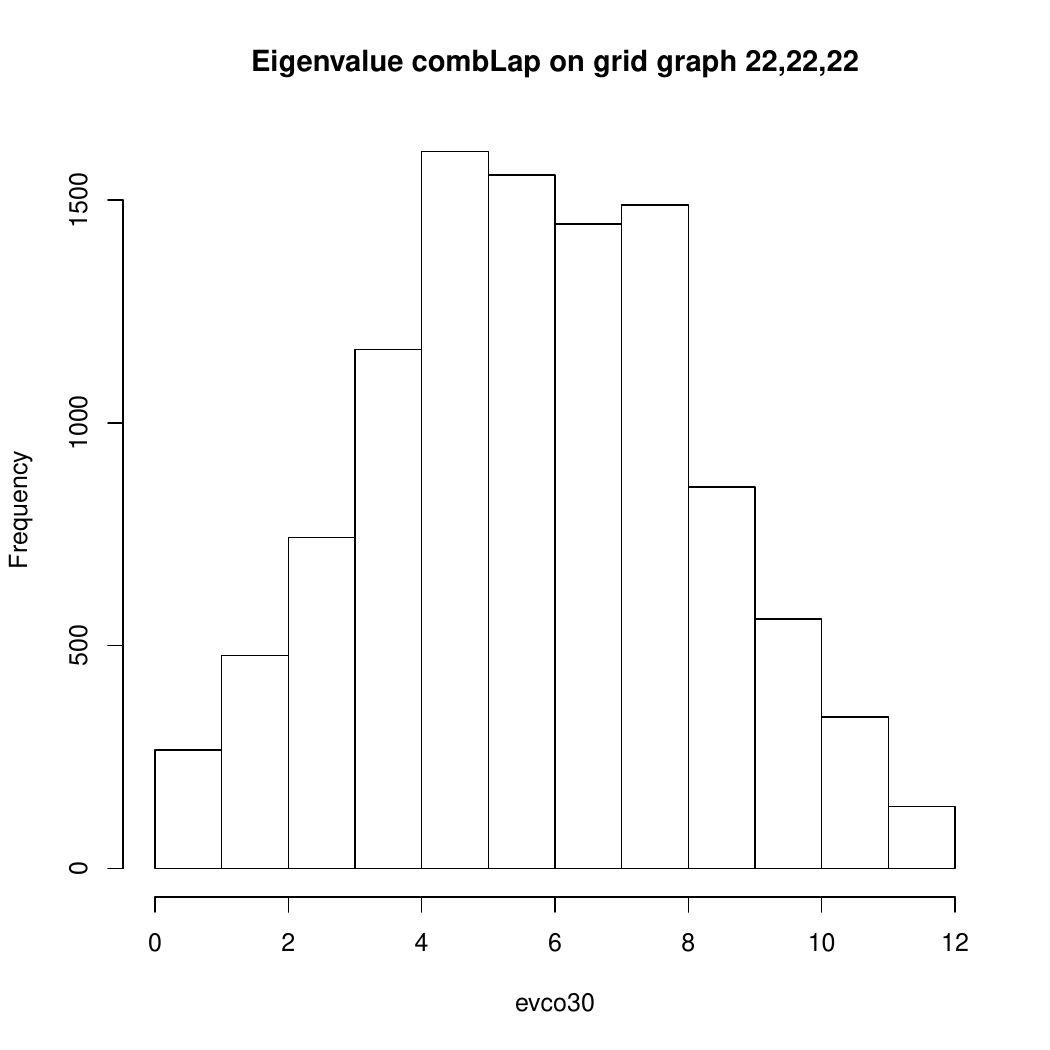}}  %
 (d)\includegraphics[width=0.4\textwidth]{\figaddr{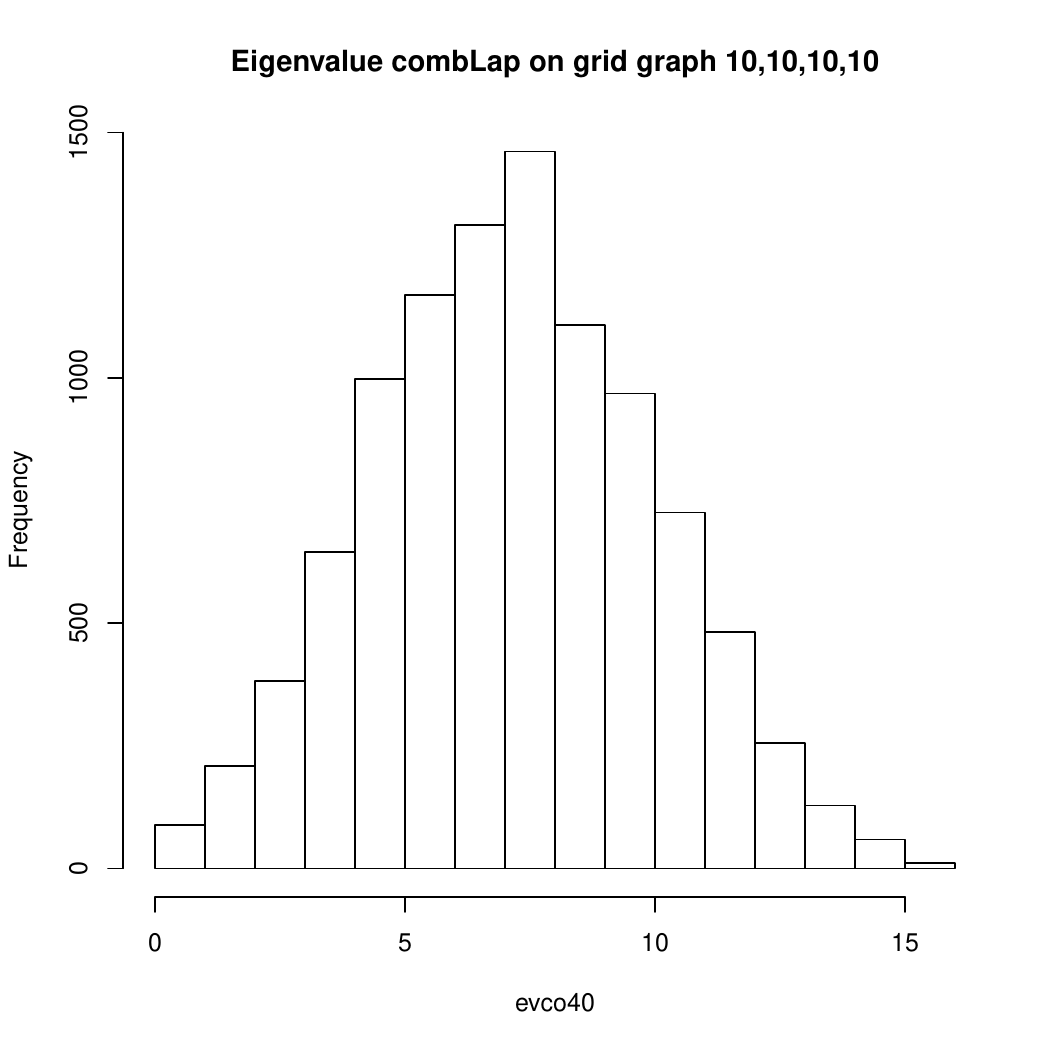}}  %
\caption{The histograms of eigenvalues of combinatorial Laplacians of grid graphs of approximately 10,000 nodes.
(a) 1-dimensional grid graph, 
(b) 2-dimensional grid graph, 
(c) 3-dimensional grid graph, 
(d) 4-dimensional grid graph. 
}\label{fig:evco10000nodes}
\end{figure}

\begin{figure}
\centering
 (a)\includegraphics[width=0.4\textwidth]{\figaddr{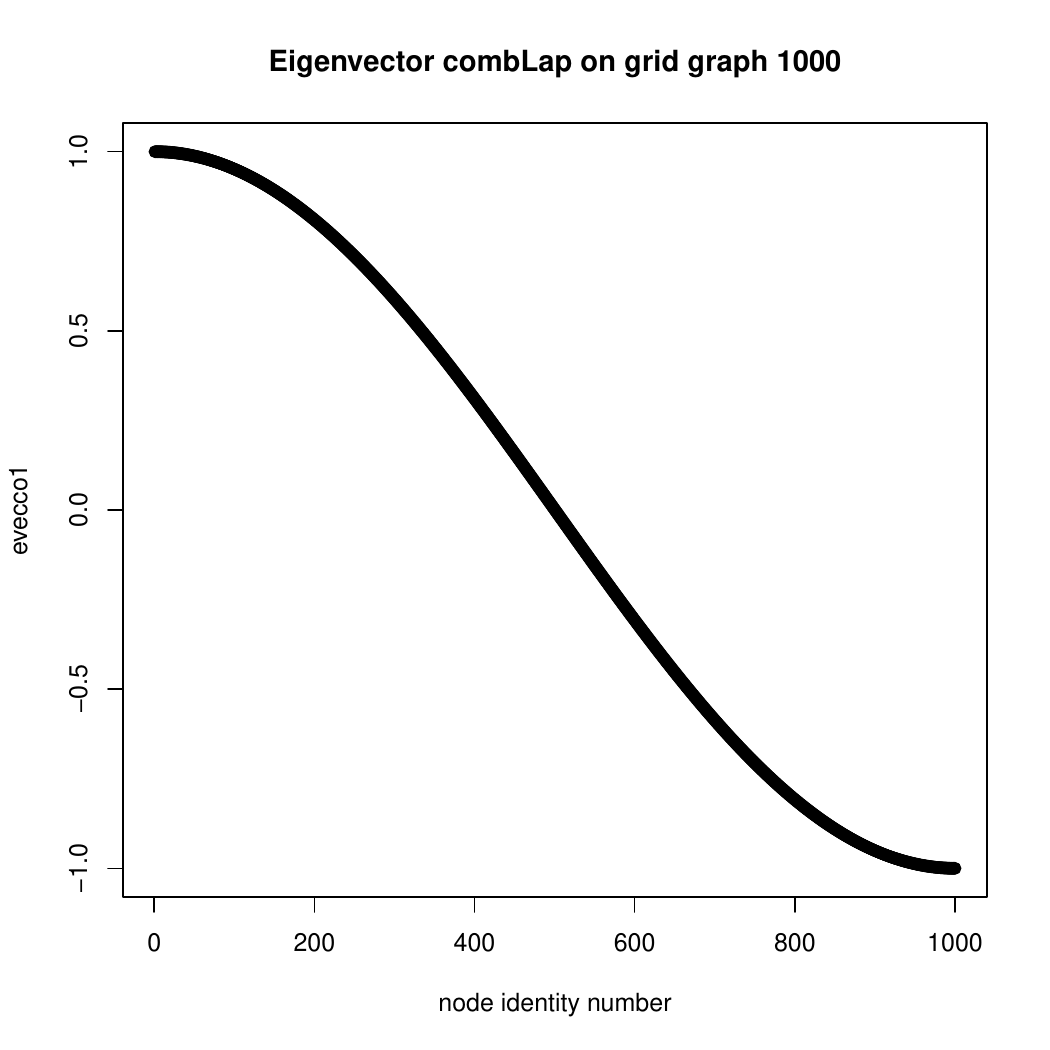}}  %
 (b)\includegraphics[width=0.4\textwidth]{\figaddr{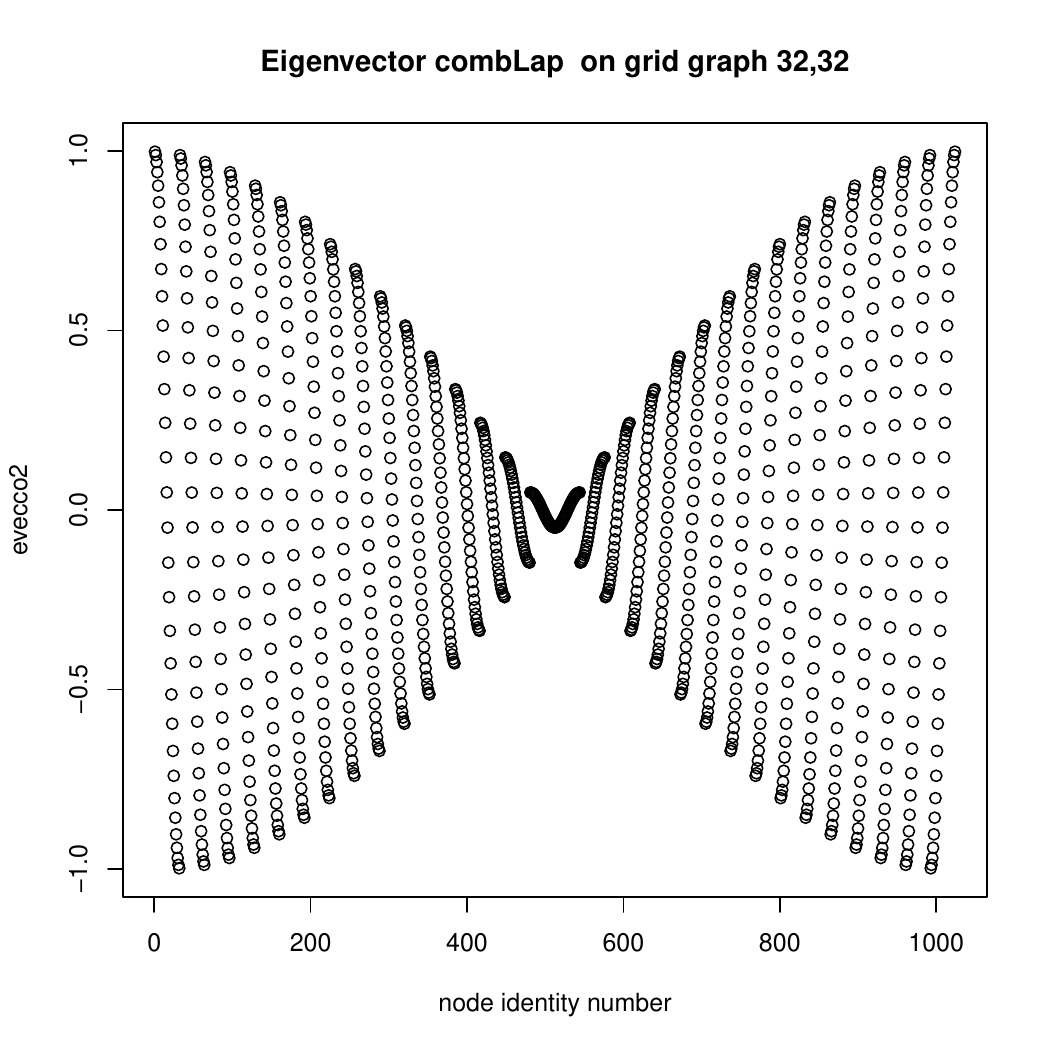}}  %
\\ (c)\includegraphics[width=0.4\textwidth]{\figaddr{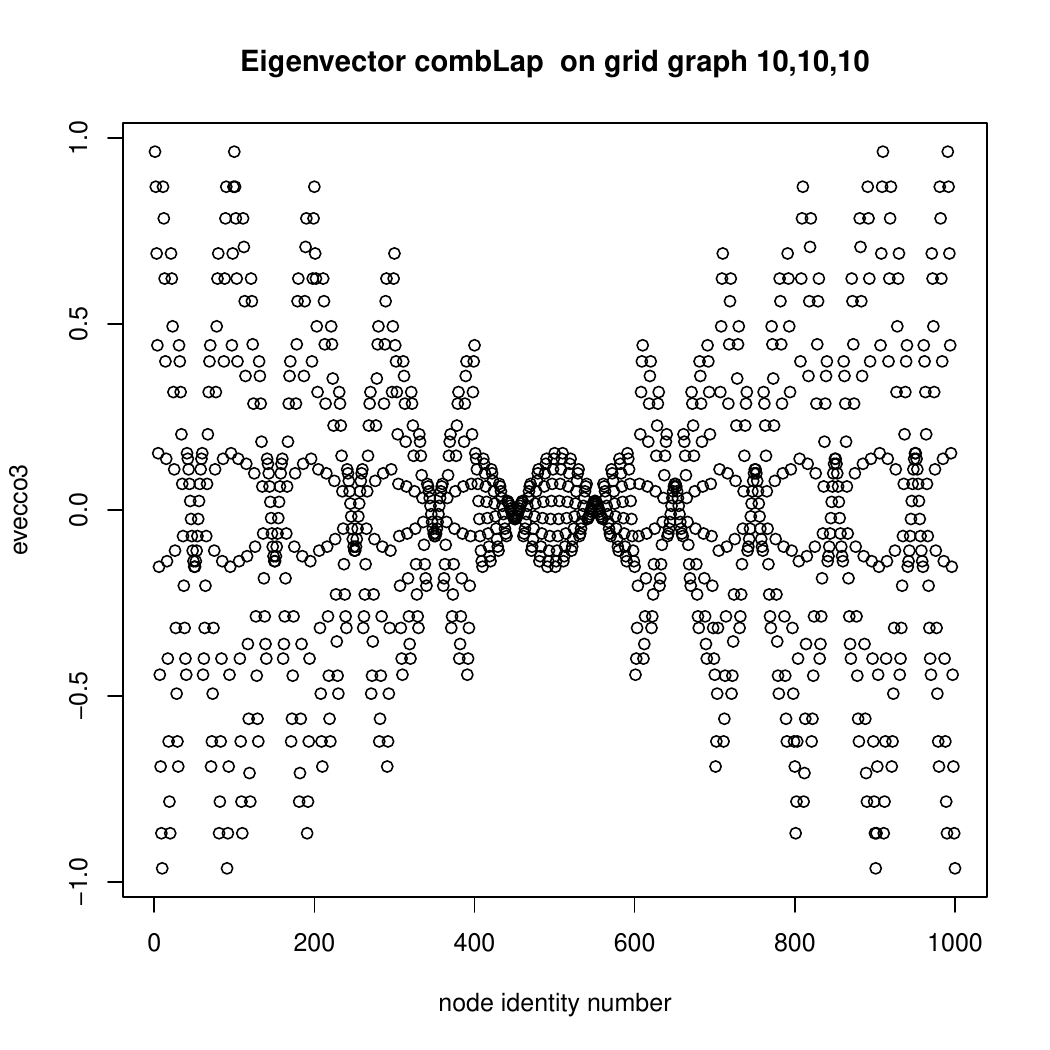}}  %
 (d)\includegraphics[width=0.4\textwidth]{\figaddr{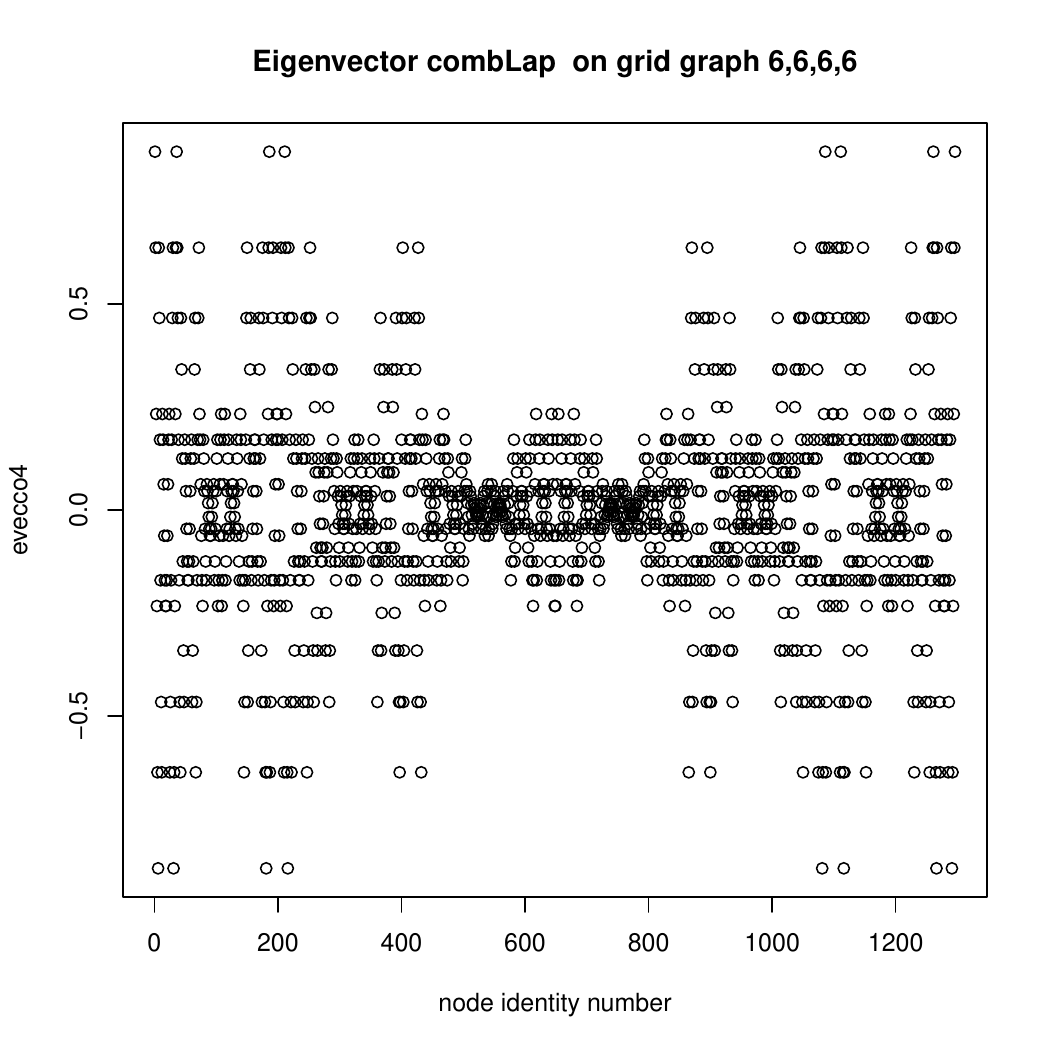}}  %
\caption{The plots of sample  eigenvectors of combinatorial Laplacians of grid graphs of approximately 1,000 nodes.
(a) 1-dimensional grid graph, $\z=[1]$, 
(b) 2-dimensional grid graph, $\z=[1,1]$, 
(c) 3-dimensional grid graph, $\z=[1,1,1]$, 
(d) 4-dimensional grid graph, $\z=[1,1,1,1]$. 
}\label{fig:evecco1000nodes}
\end{figure}

In Figure \ref{fig:evecco1000nodes} you see sample eigenvectors of the afore-mentioned grids. 
For 1-dimensional graph, a clear cosine shape is visible, in two dimensions the cosine product can be seen, in higher dimensions the patterns are not so easily classified by eye inspection.
}
 
\section{Remarks on unoriented Laplacian of  unweighted grid graph}\label{sec:UOLgeneralization}

Interestingly, there exists an elegant solution to the eigen-problem of the unoriented Laplacian. 
The unoriented Laplacian is defined as 
$$K=D+S$$

\begin{theorem}
The unoriented Laplacian eigenvalues for a grid graph are of the same form as for the unnormalised Laplacian that is
\begin{equation}\label{eq:lambdaUO}
\lambda_{[z_1,\dots,z_d]}=\sum_{j=1}^d 
\left(2 \sin\left(\frac{\pi  z_j}{2 n_j}\right)\right)^2
\end{equation}

The corresponding eigenvectors differ slightly. Their components are of the form 

\begin{equation}\label{eq:eigenvectorcomponentUO}
\nu_{[z_1,\dots,z_d],[x_1,\dots,x_d]}= 
\prod_{j=1}^d (-1)^{x_j} \cos\left(\frac{\pi z_j}{n_j} \left(x_j-0.5\right)\right) 
\end{equation}
\end{theorem}

The proofs of these properties follow the same pattern as above with slight 
variations:
the sums of elements in these vectors are not equal zero any more in general
(an analogue of Theorem \ref{th:zerosum} is not there). 
However, as we multiply always pairs of values 
associated with the same $[x_1,\dots,x_d]$ vector, 
the factors $(-1)^{x_j}$ cancel out and the proofs of analogous other four theorems are essentially the same - we can proceed as if the eigenvectors were those of combinatorial Laplacians.

\section{Normalized Laplacians of  Unweighted Grid Graphs}\label{sec:NLtheorems}

Please keep in mind that the normalised Laplacian of a graph is defined as
$$\mathfrak{L}=D^{-1/2}LD^{-1/2}=D^{-1/2}(D-S)D^{-1/2}=I-D^{-1/2} S D^{-1/2}$$

The approach to the eigen-problem of normalised Laplacian would be very similar in spirit, but there exist  technicalities that make out the complexity of the generalization.
It has to be noted also that the solution is not completely closed-form. An iterative component is needed when identifying an eienvalue. Once the eigenvalue is identified, so-called \emph{shifts} or $\delta$s are also identified and then the eigenvalue and eigenvectors are in closed form with respect to these shifts $\delta$.
The problem of only a partial closed-from is strongly related to the fact that the eigen-problem for the normalised Laplacian cannot be decomposed in a way that could be done for the combinatorial Laplacians.

A completely closed-form is possible only in special cases, that are discussed in Subsections \ref{subsec:onDimGrid} 
(on one-dimensional grid) and \ref{subsec:regularGrid} (selected solutions to a regular grid).

This section is essentially devoted to the proof of the Theorem \ref{th:normalizedLap} on the form of eigenvalues and eigenvectors of a normalised Laplacian of a grid graph. 
The proof will be split into two cases of two types of grid graph.
We shall divide the nodes of the grid into two categories: the inner and the border ones. 
The inner ones are those that have two neighbours in the grid in each dimension. 
The border ones are the remaining ones. 
The two types of grid graphs are ones that have inner nodes, and they are handled in Subsection \ref{subsec:generalInner}, while the graphs without inner nodes are treated in Subsection \ref{subsec:generalNoInner}. 

\subsection{The General Case - with inner nodes}\label{subsec:generalInner}

In this subsection we prove the validity of our suggested forms of eigenvalues and eigenvectors of normalized Laplacians of grid graphs, as formulated in the   Theorem \ref{th:normalizedLap}.

The proof will be divided into subsubsections in order not to get lost in the multitude of formulasd. 
So the Subsubsection\ref{subsub:deltas} is devoted to finding a simple equation system allowing to find the values of  shifts $\delta$  occurring in the formulas for eigenvalue and eigenvector based on selected nodes. 
The Subsubsection\ref{subsub:howcompute} contains practical hints on simple solving of the equation system for $\delta$s.
The Subsubsection\ref{subsub:othernodes} is devoted to demonstrating, that once the above equation system is solved, the shifts $\delta$  fit also other nodes, not considered in Subsubsection \ref{subsub:deltas}.
The Subsubsection\ref{subsub:noLorthogonality} demonstrates that all the eigenvectors are orthogonal to each other so that it is assured that all the eigenvectors have been found.

As in the previous sections, we shall index the eigenvalues and eigenvectors with the vector $\z=[z_1,\dots,z_d]$ such that 
$0\le z_j<n_j$ for $j=1,\dots,d$.   

Note that if $\mathbf{v}$ is the eigenvector of $\mathfrak{L}$ for some eigenvector $\lambda$, then 
$\lambda \mathbf{v}=D^{-1/2}LD^{-1/2}  \mathbf{v}$,
$\lambda \mathbf{v}=D^{-1/2}LD^{-1/2}  \mathbf{v}$,
$\lambda  (D^{-1/2} \mathbf{v})=D^{-1}L (D^{-1/2}  \mathbf{v})$,
$\lambda  D (D^{-1/2} \mathbf{v})= L (D^{-1/2}  \mathbf{v})$.
Denote $\mathbf{w}=(D^{-1/2}  \mathbf{v}) $. 
So we seek $\lambda D \mathbf{w} = L \mathbf{w}$,  
$\lambda D \mathbf{w} = (D-S) \mathbf{w}$, 
$(1-\lambda) D \mathbf{w} =  S  \mathbf{w}$, 
$((1-\lambda) D-S) \mathbf{w} = 0$.

We will subsequently show that
\begin{theorem} \label{th:normalizedLap}
For a $d$-dimensional gridc graph with at least one inner node, its normalized Laplacian  $\mathfrak{L}$
has  the 
  eigenvalues  of the   form 
\begin{equation}\label{eq:lambdaN}
\lambda_{\mathbf{z}}=1+\frac{1}{d} \sum_{j=1}^d 
\cos\left(\frac{1}{n_j-1}\left(z_j \pi -2\delta_j\right) \right)
\end{equation}
with the ${\boldsymbol\delta}^\mathbf{z}$ vector defined 
as a solution of the equation system consisting of 
  the subsequent equation \eref{eq:lambdaN2d} and the 
 equations \eref{eq:lambdaN2dm1} for each $l=1,\dots,d$. 
The corresponding eigenvectors $\mathbf{v}_{\mathbf{z}}$ have components of the form  
\begin{equation}\label{eq:eigenvectorcomponentN}
\nu_{\mathbf{z},[x_1,\dots,x_d]}= 
D^{1/2}_{[x_1,\dots,x_d],[x_1,\dots,x_d]}
\prod_{j=1}^d (-1)^{x_j} 
\cos\left(\frac{x_j-1}{n_j-1}\left(z_j \pi -2\delta^\mathbf{z}_j\right)+\delta^\mathbf{z}_j\right)
\end{equation}
\end{theorem}

\subsubsection{Defining equations for $\delta$s}\label{subsub:deltas}
Let us now derive the defining equations for the  $\boldsymbol\delta^\mathbf{z}$ vector.
However, instead of the vector $\mathbf{v}$, consider the vector $\mathbf{w}$ 
with the components
$$\omega_{\mathbf{z},[x_1,\dots,x_d]}= 
D ^{-1/2}_{[x_1,\dots,x_d],[x_1,\dots,x_d]} 
\nu_{\mathbf{z},[x_1,\dots,x_d]}$$
$$=\prod_{j=1}^d (-1)^{x_j} 
\cos\left(\frac{x_j-1}{n_j-1}\left(z_j \pi -2\delta^\mathbf{z}_j\right)+\delta^\mathbf{z}_j\right)
$$

Consider an inner node $[x_1,\dots,x_d]$. 
In order for the $\mathbf{v}$ to be a valid  eigenvector, the following must hold:
\begin{align*}
\lambda_{\mathbf{z}} &
D _{[x_1,\dots,x_d],[x_1,\dots,x_d] } 
\omega_{\mathbf{z},[x_1,\dots,x_d]}
\\
=
&\sum_{j=1}^{d}
\left( 
	\left(\omega_{\mathbf{z},[x_1,\dots,x_d]}
-\omega_{\mathbf{z},[x_1,\dots,x_j-1,\dots,x_d]}\right)
\right.
\\
&\left. +	\left(\omega_{\mathbf{z},[x_1,\dots,x_d]}
-\omega_{\mathbf{z},[x_1,\dots,x_j+1,\dots,x_d]}\right)
\right)
\end{align*}
As for any inner node 
$D _{[x_1,\dots,x_d],[x_1,\dots,x_d] } =2d$,  we obtain
\begin{align*}
2d& \lambda_{\mathbf{z}}
\prod_{j=1}^d  
\cos\left(\frac{x_j-1}{n_j-1}\left(z_j \pi -2\delta^\mathbf{z}_j\right)+\delta^\mathbf{z}_j\right)
\\
=
&\sum_{j=1}^{d}
\left( 
\cos\left(\frac{x_j-1}{n_j-1}\left(z_j \pi -2\delta^\mathbf{z}_j\right)+\delta^\mathbf{z}_j\right)
\prod_{i=1,i\ne j}^d  
\cos\left(\frac{x_i-1}{n_i-1}\left(z_i \pi -2\delta^\mathbf{z}_i\right)+\delta^\mathbf{z}_i\right)
\right. \\&+ \left. 
\cos\left(\frac{x_j-1-1}{n_j-1}\left(z_j \pi -2\delta^\mathbf{z}_j\right)+\delta^\mathbf{z}_j\right)
\prod_{i=1,i\ne j}^d  
\cos\left(\frac{x_i-1}{n_i-1}\left(z_i \pi -2\delta^\mathbf{z}_i\right)+\delta^\mathbf{z}_i\right)
\right. \\&+ \left. 
\cos\left(\frac{x_j-1}{n_j-1}\left(z_j \pi -2\delta^\mathbf{z}_j\right)+\delta^\mathbf{z}_j\right)
\prod_{i=1,i\ne j}^d  
\cos\left(\frac{x_i-1}{n_i-1}\left(z_i \pi -2\delta^\mathbf{z}_i\right)+\delta^\mathbf{z}_i\right)
\right. \\&+ \left. 
\cos\left(\frac{x_j-1+1}{n_j-1}\left(z_j \pi -2\delta^\mathbf{z}_j\right)+\delta^\mathbf{z}_j\right)
\prod_{i=1,i\ne j}^d  
\cos\left(\frac{x_i-1}{n_i-1}\left(z_i \pi -2\delta^\mathbf{z}_i\right)+\delta^\mathbf{z}_i\right)
\right)
\end{align*}

As 
\begin{align*}
\cos&\left(\frac{x_j-1-1}{n_j-1}\left(z_j \pi -2\delta^\mathbf{z}_j\right)+\delta^\mathbf{z}_j\right)
\\=& \cos\left(\frac{x_j -1}{n_j-1}\left(z_j \pi -2\delta^\mathbf{z}_j\right)+\delta^\mathbf{z}_j\right)
\cos\left(\frac{1}{n_j-1}\left(z_j \pi -2\delta^\mathbf{z}_j\right) \right)
\\&+ \sin\left(\frac{x_j -1}{n_j-1}\left(z_j \pi -2\delta^\mathbf{z}_j\right)+\delta^\mathbf{z}_j\right)
\sin\left(\frac{1}{n_j-1}\left(z_j \pi -2\delta^\mathbf{z}_j\right) \right)
\end{align*}
and
\begin{align*}
\cos&\left(\frac{x_j-1+1}{n_j-1}\left(z_j \pi -2\delta^\mathbf{z}_j\right)+\delta^\mathbf{z}_j\right)
\\=& \cos\left(\frac{x_j -1}{n_j-1}\left(z_j \pi -2\delta^\mathbf{z}_j\right)+\delta^\mathbf{z}_j\right)
\cos\left(\frac{1}{n_j-1}\left(z_j \pi -2\delta^\mathbf{z}_j\right) \right)
\\&- \sin\left(\frac{x_j -1}{n_j-1}\left(z_j \pi -2\delta^\mathbf{z}_j\right)+\delta^\mathbf{z}_j\right)
\sin\left(\frac{1}{n_j-1}\left(z_j \pi -2\delta^\mathbf{z}_j\right) \right)
\end{align*}
we obtain
\begin{align*}
2d &\lambda_{\mathbf{z}}
\prod_{j=1}^d  
\cos\left(\frac{x_j-1}{n_j-1}\left(z_j \pi -2\delta^\mathbf{z}_j\right)+\delta^\mathbf{z}_j\right)
\\&=
\sum_{j=1}^{d}
\left( 
2\cos\left(\frac{x_j-1}{n_j-1}\left(z_j \pi -2\delta^\mathbf{z}_j\right)+\delta^\mathbf{z}_j\right)
\prod_{i=1,i\ne j}^d  
\cos\left(\frac{x_i-1}{n_i-1}\left(z_i \pi -2\delta^\mathbf{z}_i\right)+\delta^\mathbf{z}_i\right)
\right.\\&+\left.
2\cos\left(\frac{x_j-1}{n_j-1}\left(z_j \pi -2\delta^\mathbf{z}_j\right)+\delta^\mathbf{z}_j\right)
\cos\left(\frac{1}{n_j-1}\left(z_j \pi -2\delta^\mathbf{z}_j\right) \right)
\right.\\&\cdot\left.
\prod_{i=1,i\ne j}^d  
\cos\left(\frac{x_i-1}{n_i-1}\left(z_i \pi -2\delta^\mathbf{z}_i\right)+\delta^\mathbf{z}_i\right)
\right)
\end{align*}
which, upon division by $\prod_{j=1}^d  
\cos\left(\frac{x_j-1}{n_j-1}\left(z_j \pi -2\delta^\mathbf{z}_j\right)+\delta^\mathbf{z}_j\right)
$,  reduces to:
\begin{equation}\label{eq:lambdaN2d}
 2d \lambda_{\mathbf{z}}
 =\sum_{j=1}^{d} \left(2+2 \cos\left(\frac{1}{n_j-1}\left(z_j \pi -2\delta^\mathbf{z}_j\right) \right)
\right)
\end{equation}
which, after dividing by $2d$  
reduces to the formula 
\eref{eq:lambdaN}. 
So for inner nodes the formula \eref{eq:lambdaN} is a valid description of the eigenvalues, without any assumptions on the 
$\boldsymbol{\delta}^\mathbf{z}$. 

Now let us turn 
to the border nodes. 
Consider the ones that have one neighbour less than the inner nodes (one neighbour missing), say along the dimension $l$, $x_l=1$.
The following must hold:
\begin{align*}
\lambda_{\mathbf{z}}&
(2d-1) 
\omega_{\mathbf{z},[x_1,\dots,x_d]}
\\=&
	\left(\omega_{\mathbf{z},[x_1,\dots,x_d]}
-\omega_{\mathbf{z},[x_1,\dots,x_l+1,\dots,x_d]}\right)
\\&+\sum_{j=1,j\ne l}^{d}
\left( 
	\left(\omega_{\mathbf{z},[x_1,\dots,x_d]}
-\omega_{\mathbf{z},[x_1,\dots,x_j-1,\dots,x_d]}\right)
\right.\\&+\left.	\left(\omega_{\mathbf{z},[x_1,\dots,x_d]}
-\omega_{\mathbf{z},[x_1,\dots,x_j+1,\dots,x_d]}\right)
\right)
\end{align*}

Hence
\begin{align*}
(2d-1)& \lambda_{\mathbf{z}}
\prod_{j=1}^d  
\cos\left(\frac{x_j-1}{n_j-1}\left(z_j \pi -2\delta^\mathbf{z}_j\right)+\delta^\mathbf{z}_j\right)
\\
=
&
\left( 
\cos\left(\frac{x_l-1}{n_l-1}\left(z_l \pi -2\delta^\mathbf{z}_l\right)+\delta^\mathbf{z}_l\right)
\prod_{i=1,i\ne l}^d  
\cos\left(\frac{x_i-1}{n_i-1}\left(z_i \pi -2\delta^\mathbf{z}_i\right)+\delta^\mathbf{z}_i\right)
\right. \\&+ \left. 
\cos\left(\frac{x_l-1+1}{n_l-1}\left(z_l \pi -2\delta^\mathbf{z}_l\right)+\delta^\mathbf{z}_l\right)
\prod_{i=1,i\ne l}^d  
\cos\left(\frac{x_i-1}{n_i-1}\left(z_i \pi -2\delta^\mathbf{z}_i\right)+\delta^\mathbf{z}_i\right)
\right)
 \\&+\sum_{j=1,j\ne l }^{d}
\left( 
\cos\left(\frac{x_j-1}{n_j-1}\left(z_j \pi -2\delta^\mathbf{z}_j\right)+\delta^\mathbf{z}_j\right)
\prod_{i=1,i\ne j}^d  
\cos\left(\frac{x_i-1}{n_i-1}\left(z_i \pi -2\delta^\mathbf{z}_i\right)+\delta^\mathbf{z}_i\right)
\right. \\&+ \left. 
\cos\left(\frac{x_j-1-1}{n_j-1}\left(z_j \pi -2\delta^\mathbf{z}_j\right)+\delta^\mathbf{z}_j\right)
\prod_{i=1,i\ne j}^d  
\cos\left(\frac{x_i-1}{n_i-1}\left(z_i \pi -2\delta^\mathbf{z}_i\right)+\delta^\mathbf{z}_i\right)
\right. \\&+ \left. 
\cos\left(\frac{x_j-1}{n_j-1}\left(z_j \pi -2\delta^\mathbf{z}_j\right)+\delta^\mathbf{z}_j\right)
\prod_{i=1,i\ne j}^d  
\cos\left(\frac{x_i-1}{n_i-1}\left(z_i \pi -2\delta^\mathbf{z}_i\right)+\delta^\mathbf{z}_i\right)
\right. \\&+ \left. 
\cos\left(\frac{x_j-1+1}{n_j-1}\left(z_j \pi -2\delta^\mathbf{z}_j\right)+\delta^\mathbf{z}_j\right)
\prod_{i=1,i\ne j}^d  
\cos\left(\frac{x_i-1}{n_i-1}\left(z_i \pi -2\delta^\mathbf{z}_i\right)+\delta^\mathbf{z}_i\right)
\right)
\end{align*}

Hence
\begin{align*}
(2d-1)& \lambda_{\mathbf{z}}
\prod_{j=1}^d  
\cos\left(\frac{x_j-1}{n_j-1}\left(z_j \pi -2\delta^\mathbf{z}_j\right)+\delta^\mathbf{z}_j\right)
\\
=
&
\left( 
\cos\left(
         \frac{x_l-1}{n_l-1}
          \left(z_l \pi -2\delta^\mathbf{z}_l\right)
   +\delta^\mathbf{z}_l
    \right) 
+\cos\left(\frac{x_l-1+1}{n_l-1}\left(z_l \pi -2\delta^\mathbf{z}_l\right)
       +\delta^\mathbf{z}_l
     \right) 
\right)
\\& \cdot 
\prod_{i=1,i\ne l}^d  
\cos\left(\frac{x_i-1}{n_i-1}\left(z_i \pi -2\delta^\mathbf{z}_i\right)+\delta^\mathbf{z}_i\right)
 \\&+\sum_{j=1,j\ne l }^{d}
\left( 
\cos\left(\frac{x_j-1}{n_j-1}\left(z_j \pi -2\delta^\mathbf{z}_j\right)+\delta^\mathbf{z}_j\right)
\prod_{i=1,i\ne j}^d  
\cos\left(\frac{x_i-1}{n_i-1}\left(z_i \pi -2\delta^\mathbf{z}_i\right)+\delta^\mathbf{z}_i\right)
\right. \\&+ \left. 
\cos\left(\frac{x_j-1-1}{n_j-1}\left(z_j \pi -2\delta^\mathbf{z}_j\right)+\delta^\mathbf{z}_j\right)
\prod_{i=1,i\ne j}^d  
\cos\left(\frac{x_i-1}{n_i-1}\left(z_i \pi -2\delta^\mathbf{z}_i\right)+\delta^\mathbf{z}_i\right)
\right. \\&+ \left. 
\cos\left(\frac{x_j-1}{n_j-1}\left(z_j \pi -2\delta^\mathbf{z}_j\right)+\delta^\mathbf{z}_j\right)
\prod_{i=1,i\ne j}^d  
\cos\left(\frac{x_i-1}{n_i-1}\left(z_i \pi -2\delta^\mathbf{z}_i\right)+\delta^\mathbf{z}_i\right)
\right. \\&+ \left. 
\cos\left(\frac{x_j-1+1}{n_j-1}\left(z_j \pi -2\delta^\mathbf{z}_j\right)+\delta^\mathbf{z}_j\right)
\prod_{i=1,i\ne j}^d  
\cos\left(\frac{x_i-1}{n_i-1}\left(z_i \pi -2\delta^\mathbf{z}_i\right)+\delta^\mathbf{z}_i\right)
\right)
\end{align*}

Hence
\begin{align*}
(2d-1)& \lambda_{\mathbf{z}}
\prod_{j=1}^d  
\cos\left(\frac{x_j-1}{n_j-1}\left(z_j \pi -2\delta^\mathbf{z}_j\right)+\delta^\mathbf{z}_j\right)
\\
=
&
\left( 
\cos\left(
         \frac{x_l-1}{n_l-1}
          \left(z_l \pi -2\delta^\mathbf{z}_l\right)
   +\delta^\mathbf{z}_l
    \right) 
\right.\\ &\left.+ \cos\left(\frac{x_l -1}{n_l-1}\left(z_l \pi -2\delta^\mathbf{z}_l\right)+\delta^\mathbf{z}_l\right)
\cos\left(\frac{1}{n_l-1}\left(z_l \pi -2\delta^\mathbf{z}_l\right) \right)
\right.\\ &\left.- \sin\left(\frac{x_l -1}{n_l-1}\left(z_l \pi -2\delta^\mathbf{z}_l\right)+\delta^\mathbf{z}_j\right)
\sin\left(\frac{1}{n_l-1}\left(z_l \pi -2\delta^\mathbf{z}_l\right) \right)
\right)
\\& \cdot 
\prod_{i=1,i\ne l}^d  
\cos\left(\frac{x_i-1}{n_i-1}\left(z_i \pi -2\delta^\mathbf{z}_i\right)+\delta^\mathbf{z}_i\right)
 \\&+\sum_{j=1,j\ne l }^{d}
\left( 
2\cos\left(\frac{x_j-1}{n_j-1}\left(z_j \pi -2\delta^\mathbf{z}_j\right)+\delta^\mathbf{z}_j\right)
\prod_{i=1,i\ne j}^d  
\cos\left(\frac{x_i-1}{n_i-1}\left(z_i \pi -2\delta^\mathbf{z}_i\right)+\delta^\mathbf{z}_i\right)
\right.\\&+\left.
2\cos\left(\frac{x_j-1}{n_j-1}\left(z_j \pi -2\delta^\mathbf{z}_j\right)+\delta^\mathbf{z}_j\right)
\cos\left(\frac{1}{n_j-1}\left(z_j \pi -2\delta^\mathbf{z}_j\right) \right)
\right.\\&\cdot\left.
\prod_{i=1,i\ne j}^d  
\cos\left(\frac{x_i-1}{n_i-1}\left(z_i \pi -2\delta^\mathbf{z}_i\right)+\delta^\mathbf{z}_i\right)
\right)
\end{align*}

 By dividing as previously we get 
\begin{align*}
(2d-1)& \lambda_{\mathbf{z}}
\\
=
&
\left( 
1 
\right.\\ &\left.+ 
\cos\left(\frac{1}{n_l-1}\left(z_l \pi -2\delta^\mathbf{z}_l\right) \right)
\right.\\ &\left.- 
\tan\left(\frac{x_l -1}{n_l-1}\left(z_l \pi -2\delta^\mathbf{z}_l\right)+\delta^\mathbf{z}_j\right)
\sin\left(\frac{1}{n_l-1}\left(z_l \pi -2\delta^\mathbf{z}_l\right) \right)
\right)
 \\&+\sum_{j=1,j\ne l }^{d}
\left( 
2 
\right.\\&+\left.
2 
\cos\left(\frac{1}{n_j-1}\left(z_j \pi -2\delta^\mathbf{z}_j\right) \right)
\right)
\end{align*}

Considerations  above lead to the conclusion (as $x_l -1=0$)  
\begin{align*}
(2d-1) \lambda_{\mathbf{z}}
 =&
1+\cos\left(\frac{1}{n_l-1}\left(z_l \pi -2\delta^\mathbf{z}_l\right)\right)
-\tan(\delta^\mathbf{z}_l)\sin\left(\frac{1}{n_l-1}\left(z_l \pi -2\delta^\mathbf{z}_l\right)\right)
\\&+
\sum_{j=1,j\ne l}^{d} \left(2+2 \cos\left(\frac{1}{n_j-1}\left(z_j \pi -2\delta^\mathbf{z}_j\right) \right)
\right)
\end{align*}

A subtraction of the preceding formula from the expression  \eref{eq:lambdaN2d}   leads to
\begin{equation}\label{eq:lambdaN2dm1}
 \lambda_{\mathbf{z}}
 =
1+\cos\left(\frac{1}{n_l-1}\left(z_l \pi -2\delta^\mathbf{z}_l\right)\right)
+\tan(\delta^\mathbf{z}_l)\sin\left(\frac{1}{n_l-1}\left(z_l \pi -2\delta^\mathbf{z}_l\right)\right)
\end{equation}
By combining the equation \eref{eq:lambdaN2d} 
with equations \eref{eq:lambdaN2dm1} for each $l$ 
we get an equation system of $d+1$ equations  from which 
$\lambda$ and $\delta$s can be determined. 
\subsubsection{Practical considerations for computing $\lambda$ and $\delta$s}\label{subsub:howcompute}
Note that for practical reasons the equation \eref{eq:lambdaN2dm1} can be transformed to:
$$ (\lambda_{\mathbf{z}}-1) \cos(\delta^\mathbf{z}_l)
 =
 +\cos(\delta^\mathbf{z}_l)\cos\left(\frac{1}{n_l-1}\left(z_l \pi -2\delta^\mathbf{z}_l\right)\right)
+\sin(\delta^\mathbf{z}_l)\sin\left(\frac{1}{n_l-1}\left(z_l \pi -2\delta^\mathbf{z}_l\right)\right)
$$
 
\begin{equation}\label{eq:deltafromlambda}
 (\lambda_{\mathbf{z}}-1) \cos(\delta^\mathbf{z}_l)
 =
 \cos(\delta^\mathbf{z}_l-\frac{1}{n_l-1}\left(z_l \pi -2\delta^\mathbf{z}_l\right) 
\end{equation}
which is simpler to solve for $\delta$ knowing $\lambda$. 
So the solution can be obtained using the bisectional method 
on $\lambda$ using the above formula to obtain $\delta$s, and the using \eref{eq:lambdaN} to get the value of $\lambda'$ and then 
reducing bisectionally the difference between $\lambda$ and $\lambda'$ down to zero. 

\figVer{
\begin{figure}
\centering
 (a)\includegraphics[width=0.4\textwidth]{\figaddr{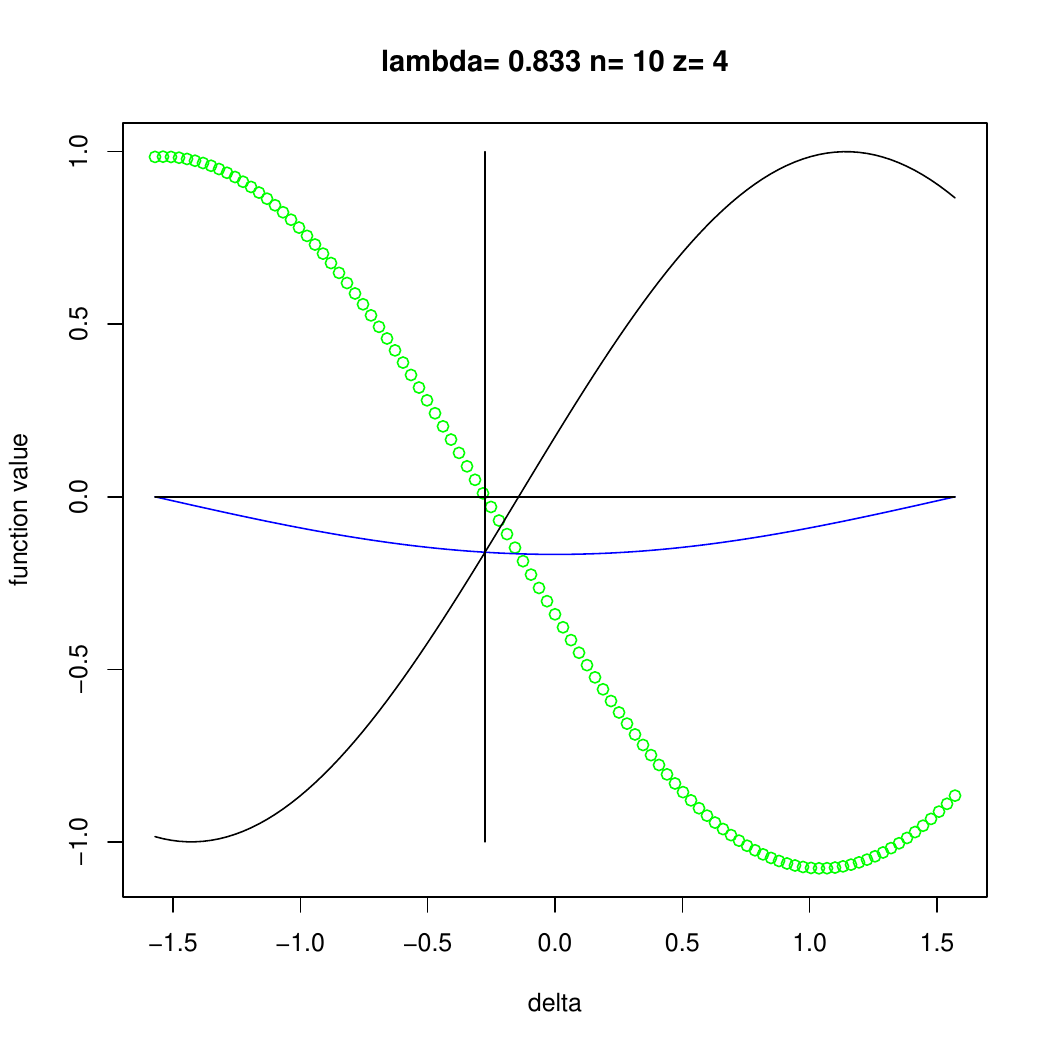}}  %
 (b)\includegraphics[width=0.4\textwidth]{\figaddr{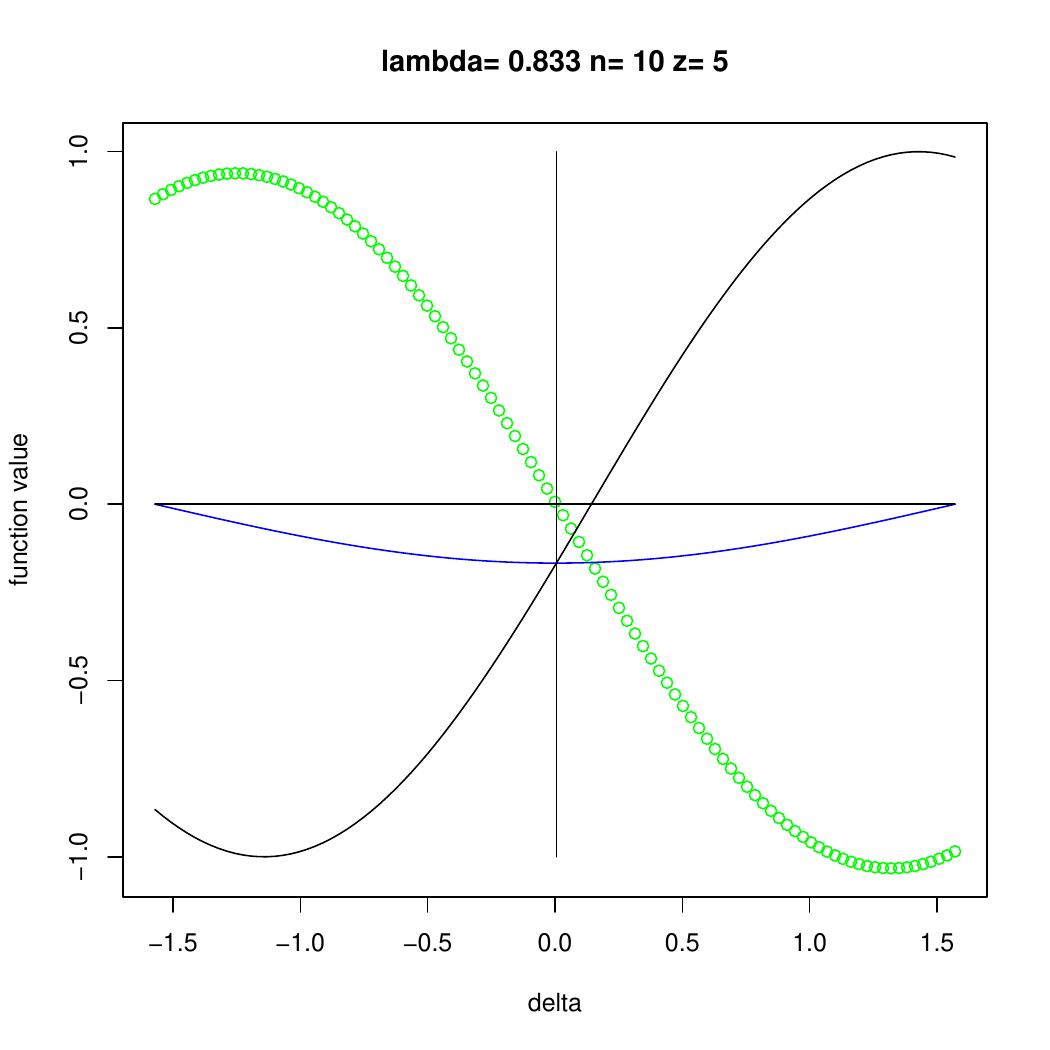}}  %
\\ (c)\includegraphics[width=0.4\textwidth]{\figaddr{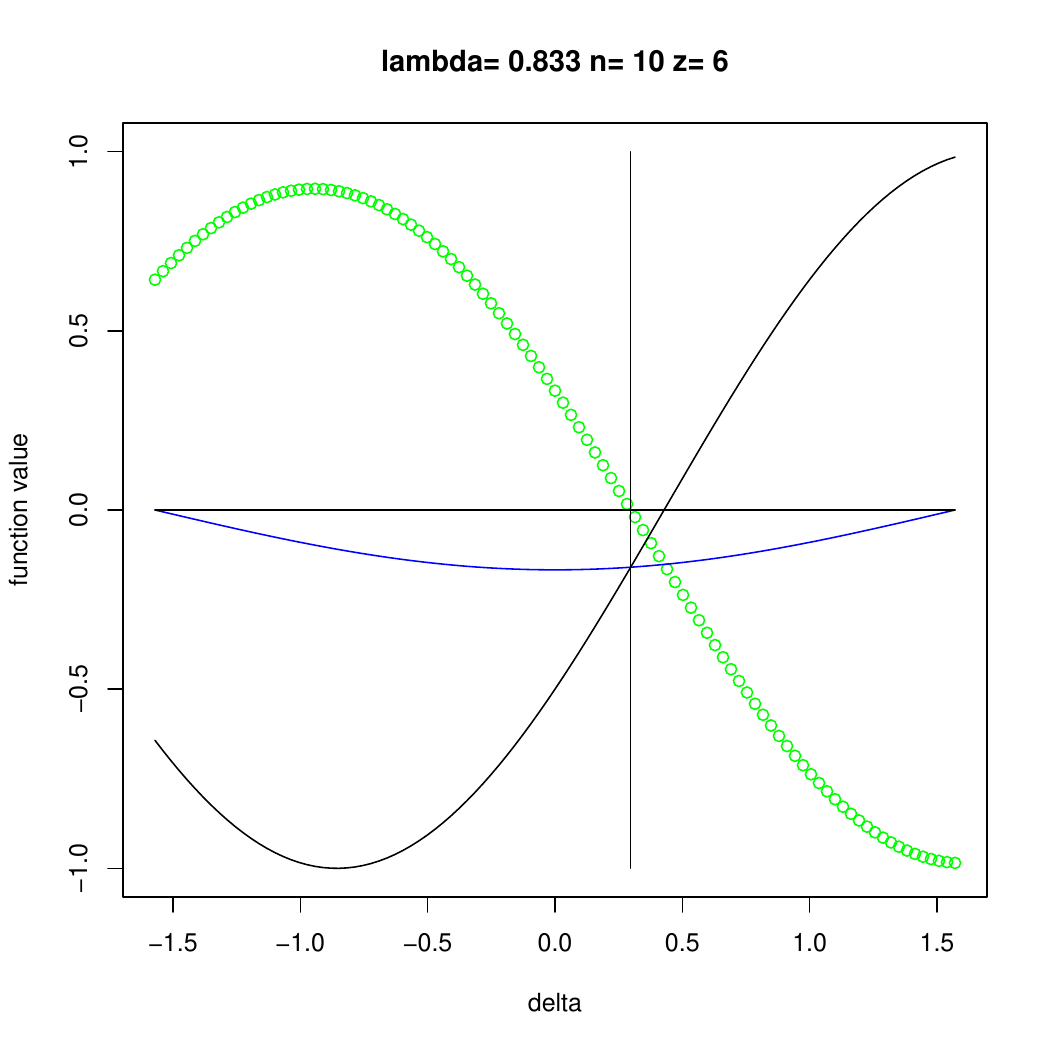}}  %
 (d)\includegraphics[width=0.4\textwidth]{\figaddr{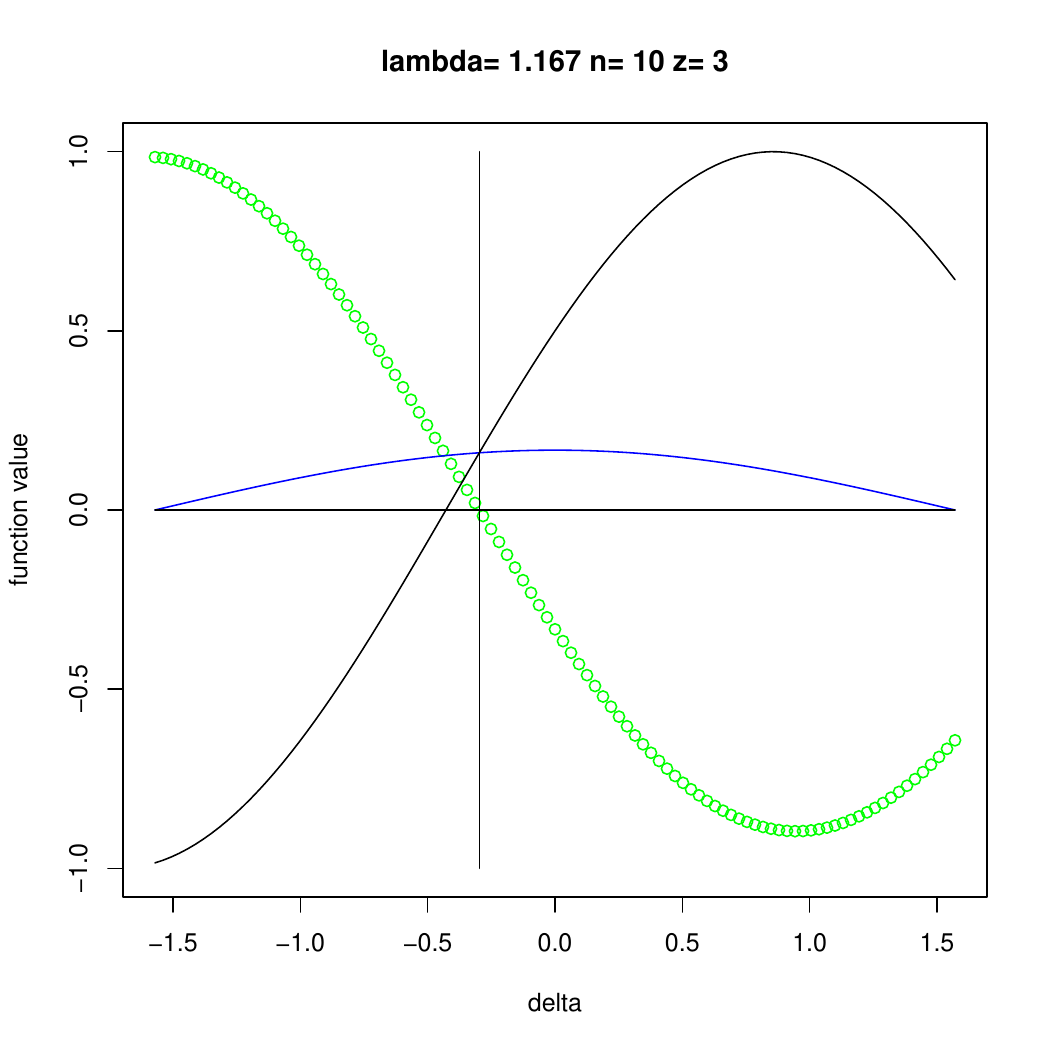}}  %
\\ (e)\includegraphics[width=0.4\textwidth]{\figaddr{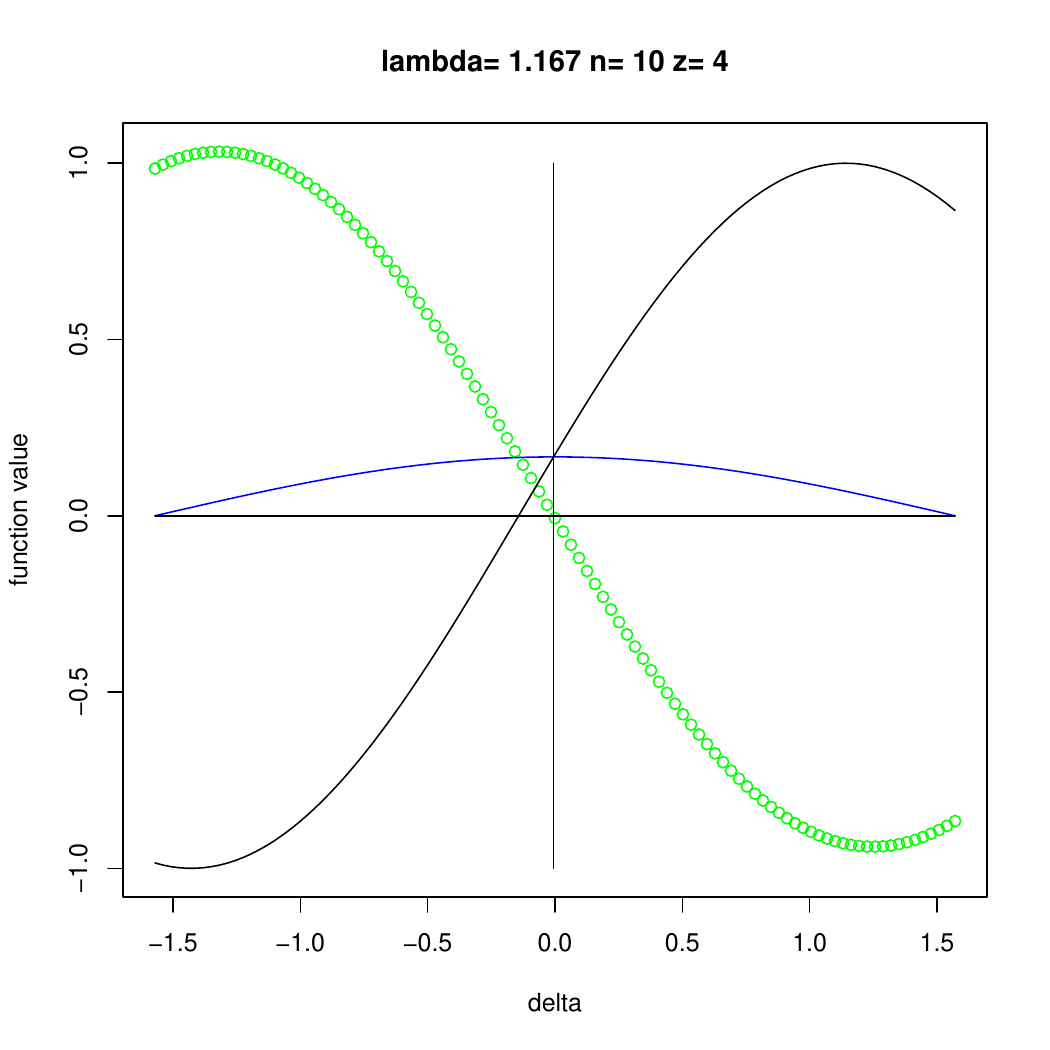}}  %
 (f)\includegraphics[width=0.4\textwidth]{\figaddr{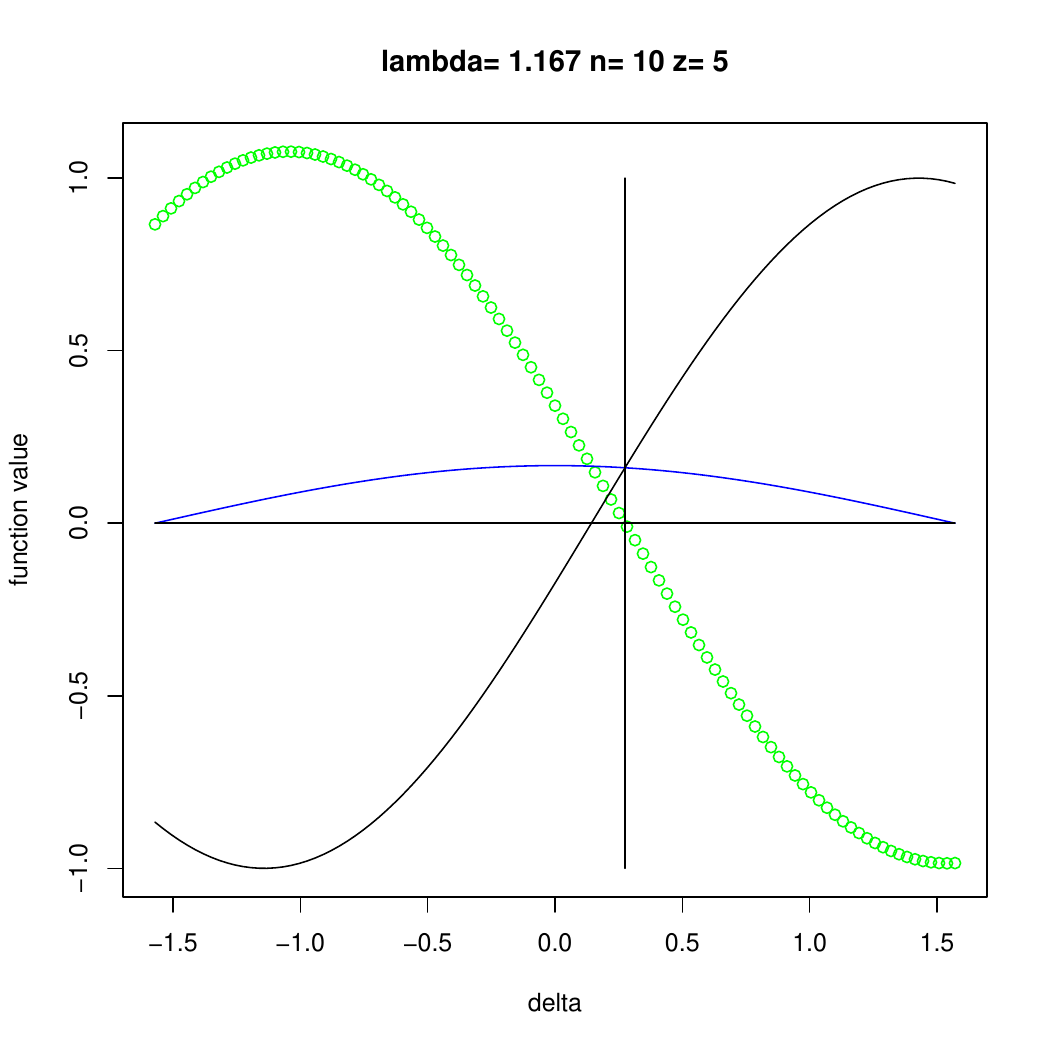}}  %
\caption{Illustration for the formula 
\eref{eq:deltafromlambda} for a 3d grid graph.
The blue curve is $ (\lambda_{\mathbf{z}}-1) \cos(\delta^\mathbf{z}_l)$,
the black curve  illustrates
$ \cos(\delta^\mathbf{z}_l-\frac{1}{n_l-1}\left(z_l \pi -2\delta^\mathbf{z}_l\right) 
$, 
the green one is their difference. 
The figures (a), (b), (c) refer to eigenvalues below 1, 
The figures (d), (e), (f) refer to eigenvalues above 1, 
}\label{fig:drawDeltaByLambda}
\end{figure}

Note that the equation \eref{eq:deltafromlambda} has a special form. 
For $\lambda_{\mathbf{z}}>1$/$\lambda_{\mathbf{z}}<1$, if we restrict ourselves to non-negative/non-positive   valued parts of these functions in the domain $[-\pi,\pi]$  we need to consider the intersection of two convex upward (cases (d-f) in Fig.\ref{fig:drawDeltaByLambda}) or downward (cases (a-c) in Fig.\ref{fig:drawDeltaByLambda})  functions, as illustrated in Figure \ref{fig:drawDeltaByLambda} which cross the $X$ axis at interleaving points, so that they must intersect above/below $X$-axis at exactly one point.
}

\subsubsection{Validity of the derived $\lambda$ and $\delta$s for other nodes}\label{subsub:othernodes}

The question is now if the solution would fit all the other border nodes. 

Consider first 
 the ones that have one neighbour less than the 
inner nodes (one neighbour missing), 
say along the dimension $l$ where $x_l=n_l$.
The following must hold:
\begin{align*}
\lambda_{\mathbf{z}}&
(2d-1) 
\omega_{\mathbf{z},[x_1,\dots,x_d]}
\\=&
	\left(\omega_{\mathbf{z},[x_1,\dots,x_d]}
-\omega_{\mathbf{z},[x_1,\dots,x_l-1,\dots,x_d]}\right)
\\&
+\sum_{j=1,j\ne l}^{d}
\left( 
	\left(\omega_{\mathbf{z},[x_1,\dots,x_d]}
-\omega_{\mathbf{z},[x_1,\dots,x_j-1,\dots,x_d]}\right)
\right.\\&+\left.	\left(\omega_{\mathbf{z},[x_1,\dots,x_d]}
-\omega_{\mathbf{z},[x_1,\dots,x_j+1,\dots,x_d]}\right)
\right)
\end{align*}
Considerations analogous to the above lead to the conclusion 
\begin{align*}
(2d-1)& \lambda_{\mathbf{z}}
\\ = &
1+\cos\left(\frac{1}{n_l-1}\left(z_l \pi -2\delta^\mathbf{z}_l\right)\right)
\\&+\tan(
\frac{n_l-1}{n_l-1}\left(z_l \pi -2\delta^\mathbf{z}_l\right)+
\delta^\mathbf{z}_l)\sin\left(\frac{1}{n_l-1}\left(z_l \pi -2\delta^\mathbf{z}_l\right)\right)
\\&+
\sum_{j=1,j\ne l}^{d} \left(2+2 \cos\left(\frac{1}{n_j-1}\left(z_j \pi -2\delta^\mathbf{z}_j\right) \right)
\right)
\end{align*}
A subtraction of the preceding formula from the expression  \eref{eq:lambdaN2d}   leads to

\begin{equation} 
 \lambda_{\mathbf{z}}
 =
1-\cos\left(\frac{1}{n_l-1}\left(z_l \pi -2\delta^\mathbf{z}_l\right)\right)
-\tan(-\delta^\mathbf{z}_l)\sin\left(\frac{1}{n_l-1}\left(z_l \pi -2\delta^\mathbf{z}_l\right)\right)
\end{equation}
which is the same as equation
\eref{eq:lambdaN2dm1}. 

Now look at other border nodes. 
Consider the ones that have one neighbour less than the inner nodes (one neighbour missing), 
along multiple dimensions, 
say along the dimensions $l^+_{1},l^+_{2},\dots,l^+_{m^+}$, $x_{l^+_k}=1$ 
and along the dimensions $l^-_{1},l^-_{2},\dots,l^-_{m^-}$, 
  $x_{l^-_k}==n_{l^-_k}$, with $1<m^++m^-\le d$.
The following must hold:
\begin{align*}
\lambda_{\mathbf{z}}& 
(2d-m^+-m^-) 
\omega_{\mathbf{z},[x_1,\dots,x_d]}
\\=&
\sum_{k=1}^{m^+}
	\left(\omega_{\mathbf{z},[x_1,\dots,x_d]}
-\omega_{\mathbf{z},[x_1,\dots,x_{l^+_k} + 1,\dots,x_d]}\right)
\\& +
 \sum_{k=1}^{m^-}
	\left(\omega_{\mathbf{z},[x_1,\dots,x_d]}
-\omega_{\mathbf{z},[x_1,\dots,x_{l^-_k} - 1,\dots,x_d]}\right)
\\& +
\sum_{j=1,j\not\in \{ l^+_1,\dots,l^+_{m^+},l^-_1,\dots,l^-_{m^-}\}}^{d}
\left( 
	\left(\omega_{\mathbf{z},[x_1,\dots,x_d]}
-\omega_{\mathbf{z},[x_1,\dots,x_j-1,\dots,x_d]}\right)
\right. \\& + \left.
	\left(\omega_{\mathbf{z},[x_1,\dots,x_d]}
-\omega_{\mathbf{z},[x_1,\dots,x_j+1,\dots,x_d]}\right)
\right)
\end{align*}
This will lead to 
(after subtraction from the expression  \eref{eq:lambdaN2d})  
\begin{align*} 
(m^++m^-)& \lambda_{\mathbf{z}}
\\ = &
\sum_{k=1}^{m^+} \left(1-\cos\left(\frac{1}{n_{l^+_k}-1}\left(z_{l^+_k} \pi -2\delta^\mathbf{z}_{l^+_k}\right)\right)
\right. \\  & + \left.
 \tan(\delta^\mathbf{z}_{l^+_k})\sin\left(\frac{1}{n_{l^+_k}-1}\left(z_{l^+_k} \pi -2\delta^\mathbf{z}_{l^+_k}\right)\right)
\right)
\\  & + 
\sum_{k=1}^{m^-} \left(1-\cos\left(\frac{1}{n_{l^-_k}-1}\left(z_{l^-_k} \pi -2\delta^\mathbf{z}_{l^-_k}\right)\right)
\right. \\  & + \left.
\tan(\delta^\mathbf{z}_{l^-_k})\sin\left(\frac{1}{n_{l^-_k}-1}\left(z_{l^-_k} \pi -2\delta^\mathbf{z}_{l^-_k}\right)\right)
\right)
\end{align*}
This equation results from adding equations \eref{eq:lambdaN2dm1} for respective dimensions
$l  \in \{ l^+_1,\dots,l^+_{m^+},l^-_1,\dots,l^-_{m^-}\}$. 
Hence, once the equation system was solved for $\boldsymbol{\delta}^\mathbf{z}$ for the above-mentioned set of nodes, 
all the other nodes fit. 
So the correctness of the formula for eigenvalues $\lambda$  was proven along with the correctness of eigenvector formulas. 

\subsubsection{The validity of the eigenvectors for identical $\lambda$}\label{subsub:noLorthogonality}

For completeness, as some eigenvalues may be identical because of symmetries, it has to be shown that the eigenvectors proposed are 
orthogonal for different $\mathbf{z}$. 

\Bem{
Consider
\begin{align*} 
\nu_{\mathbf{z},[x_1,\dots,n_l+1-x_l,\dots,x_d]}
\\&= 
D^{-1}_{[x_1,\dots,n_l+1-x_l,\dots,x_d],[x_1,\dots,n_l+1-x_l,\dots,x_d]}
  (-1)^{n_l+1-x_l} 
\cos\left(\frac{n_l+1-x_l-1}{n_l-1}\left(z_l \pi -2\delta^\mathbf{z}_l\right)+\delta^\mathbf{z}_l\right)
\prod_{j=1,j\ne l}^d (-1)^{n_j+1-x_j} 
\cos\left(\frac{ x_j-1}{n_j-1}\left(z_j \pi -2\delta^\mathbf{z}_j\right)+\delta^\mathbf{z}_j\right)
\\&= 
D^{-1}_{[x_1,\dots,x_d],[x_1,\dots,x_d]}
\cos\left(\frac{n_l+1-x_l-1}{n_l-1}\left(z_l \pi -2\delta^\mathbf{z}_l\right)+\delta^\mathbf{z}_l\right)
\prod_{j=1,j\ne l}^d (-1)^{n_j+1-x_j} 
\cos\left(\frac{ x_j-1}{n_j-1}\left(z_j \pi -2\delta^\mathbf{z}_j\right)+\delta^\mathbf{z}_j\right)
\end{align*}
}

As known from the theory, eigenvectors of a symmetric matrix related to different eigenvalues  are orthogonal.

Therefore, to substantiate our claim that we identified all the eigenvectors, we must show that the distinct eigenvectors related to the same eigenvalue are also orthogonal,  that is $ \lambda_{\mathbf{z'}}= \lambda_{\mathbf{z"}}$ 
for $\mathbf{z'} \ne  \mathbf{z"}$. 

\Bem{
This will surely happen when there exists a mapping $m:\{1,\dots,d\}\rightarrow \{1,\dots,d\}$ such that for each $i\in \{1,\dots,d\}$ the following holds: 
$n_i=n_{m(i)}, z'_i=z"_{m(i)}$, and there exists a $k$ such that 
$m(k)\ne k$. 
This implies also that 
$\delta^\mathbf{z'}_j=\delta^\mathbf{z"}_{m(j)}$. 
}

Consider now the orthogonality of 
$\mathbf{v}^\mathbf{z'}$ and  $\mathbf{v}^  \mathbf{z"}$.
The following has to hold:
 $$0 = \sum_\mathbf{x} D_{[\x],[\x]}
\prod_{j=1}^d   
\cos\left(\frac{x_j-1}{n_j-1}\left(z'_j \pi -2\delta^\mathbf{z'}_j\right)+\delta^\mathbf{z'}_j\right)
\cdot
\cos\left(\frac{x_j-1}{n_j-1}\left(z"_j \pi -2\delta^\mathbf{z"}_j\right)+\delta^\mathbf{z"}_j\right)
$$ 
which is equivalent to:
\begin{align}
  0=&  
  \sum_{x_2=1}^{n_2}\dots   \sum_{x_d=1}^{n_d} 
  \prod_{j=2}^d   
\cos\left(\frac{x_j-1}{n_j-1}\left(z'_j \pi -2\delta^\mathbf{z'}_j\right)+\delta^\mathbf{z'}_j\right)
\cdot
\cos\left(\frac{x_j-1}{n_j-1}\left(z"_j \pi -2\delta^\mathbf{z"}_j\right)+\delta^\mathbf{z"}_j\right)
\nonumber \\ &
 \sum_{x_1=1}^{n_1}
 D_{[\x],[\x]}
 \cos\left(\frac{x_1-1}{n_1-1}\left(z'_1 \pi -2\delta^\mathbf{z'}_1\right)+\delta^\mathbf{z'}_1\right)
\cdot
\cos\left(\frac{x_1-1}{n_1-1}\left(z"_1 \pi -2\delta^\mathbf{z"}_1\right)+\delta^\mathbf{z"}_1\right) \label{eq:mainN}
\end{align}
Let $d^{[M]}_{x_2,\dots,x_d}=\max_{x_1 \in \{1,\dots,n_1\}} D_{[x_1,x_2,\dots,x_d],[x_1,x_2,\dots,x_d]}$. (In the special case of $n_1=2$ $d^{[M]}_{x_2,\dots,x_d}$ should be increased by 1.)
Then we claim that 
the sum 
$$ \sum_{x_1=1}^{n_1}
 D_{[\x],[\x]}
 \cos\left(\frac{x_1-1}{n_1-1}\left(z'_1 \pi -2\delta^\mathbf{z'}_1\right)+\delta^\mathbf{z'}_1\right)
\cdot
\cos\left(\frac{x_1-1}{n_1-1}\left(z"_1 \pi -2\delta^\mathbf{z"}_1\right)+\delta^\mathbf{z"}_1\right)
$$
$$=
 -2\cos(\delta^\mathbf{z'}_1)\cos(\delta^\mathbf{z"}_1) ((z'_1+z"_1+1) \mod 2)
$$ $$ +d^{[M]}_{x_2,\dots,\x_d}
\sum_{x_1=1}^{n_1}
 \cos\left(\frac{x_1-1}{n_1-1}\left(z'_1 \pi -2\delta^\mathbf{z'}_1\right)+\delta^\mathbf{z'}_1\right)
\cdot
\cos\left(\frac{x_1-1}{n_1-1}\left(z"_1 \pi -2\delta^\mathbf{z"}_1\right)+\delta^\mathbf{z"}_1\right)
$$
  \begin{equation}\label{eq:res1}=
 -2\cos(\delta^\mathbf{z'}_1)\cos(\delta^\mathbf{z"}_1) ((z'_1+z"_1+1) \mod 2)
\end{equation}
This shall be shown in two steps.
First the validity of the first equal sign will be shown, then we will demonstrate that   
$\sum_{x_j=1}^{n_j}
 \cos\left(\frac{x_j-1}{n_j-1}\left(z'_j \pi -2\delta^\mathbf{z'}_j\right)+\delta^\mathbf{z'}_j\right)
\cdot
\cos\left(\frac{x_j-1}{n_j-1}\left(z"_j \pi -2\delta^\mathbf{z"}_j\right)+\delta^\mathbf{z"}_j\right)=0$

So to the first part. 
By fixing $x_2,\dots,x_d$ we select a path in the graph 
of nodes with identities 
$[1,x_2,\dots,x_d],\dots [n_1,x_2,\dots,x_d]$ in which each node is connected
to the same number of other nodes outside of this path 
and on the path the first and the last node 
have degree 1 on this path and the other have the degree 2.
So in all, the other nodes have the degree  
 $d^{[M]}_{x_2,\dots,x_d}$ and the endpoints have a lower degree 
$d^{[M]}_{x_2,\dots,x_d}-1$. 
So 
$ \sum_{x_1=1}^{n_1}
 D_{[\x],[\x]}
 \cos\left(\frac{x_1-1}{n_1-1}\left(z'_1 \pi -2\delta^\mathbf{z'}_1\right)+\delta^\mathbf{z'}_1\right)
\cdot
\cos\left(\frac{x_1-1}{n_1-1}\left(z"_1 \pi -2\delta^\mathbf{z"}_1\right)+\delta^\mathbf{z"}_1\right)
$
$=(d^{[M]}_{x_2,\dots,x_d}-1)
 \cos\left(\frac{ 1-1}{n_1-1}\left(z'_1 \pi -2\delta^\mathbf{z'}_1\right)+\delta^\mathbf{z'}_1\right)
\cdot
\cos\left(\frac{ 1-1}{n_1-1}\left(z"_1 \pi -2\delta^\mathbf{z"}_1\right)+\delta^\mathbf{z"}_1\right)
+(d^{[M]}_{x_2,\dots,x_d}-1)
 \cos\left(\frac{n_1-1}{n_1-1}\left(z'_1 \pi -2\delta^\mathbf{z'}_1\right)+\delta^\mathbf{z'}_1\right)
\cdot
\cos\left(\frac{n_1-1}{n_1-1}\left(z"_1 \pi -2\delta^\mathbf{z"}_1\right)+\delta^\mathbf{z"}_1\right)
+ 
  \sum_{x_1=2}^{n_1-1}
d^{[M]}_{x_2,\dots,x_d}
 \cos\left(\frac{x_1-1}{n_1-1}\left(z'_1 \pi -2\delta^\mathbf{z'}_1\right)+\delta^\mathbf{z'}_1\right)
\cdot
\cos\left(\frac{x_1-1}{n_1-1}\left(z"_1 \pi -2\delta^\mathbf{z"}_1\right)+\delta^\mathbf{z"}_1\right)
$ \\
$= -  
 \cos\left(\frac{ 1-1}{n_1-1}\left(z'_1 \pi -2\delta^\mathbf{z'}_1\right)+\delta^\mathbf{z'}_1\right)
\cdot
\cos\left(\frac{ 1-1}{n_1-1}\left(z"_1 \pi -2\delta^\mathbf{z"}_1\right)+\delta^\mathbf{z"}_1\right)
- 
 \cos\left(\frac{n_1-1}{n_1-1}\left(z'_1 \pi -2\delta^\mathbf{z'}_1\right)+\delta^\mathbf{z'}_1\right)
\cdot
\cos\left(\frac{n_1-1}{n_1-1}\left(z"_1 \pi -2\delta^\mathbf{z"}_1\right)+\delta^\mathbf{z"}_1\right)
+ 
  \sum_{x_1=1}^{n_1 }
d^{[M]}_{x_2,\dots,x_d}
 \cos\left(\frac{x_1-1}{n_1-1}\left(z'_1 \pi -2\delta^\mathbf{z'}_1\right)+\delta^\mathbf{z'}_1\right)
\cdot
\cos\left(\frac{x_1-1}{n_1-1}\left(z"_1 \pi -2\delta^\mathbf{z"}_1\right)+\delta^\mathbf{z"}_1\right)
$
$= -  2 
 \cos\left( \delta^\mathbf{z'}_1\right)
\cdot
\cos\left( \delta^\mathbf{z"}_1\right) ((z'_1+z"_1+1) \mod 2)
+ 
  \sum_{x_1=1}^{n_1 }
d^{[M]}_{x_2,\dots,x_d}
 \cos\left(\frac{x_1-1}{n_1-1}\left(z'_1 \pi -2\delta^\mathbf{z'}_1\right)+\delta^\mathbf{z'}_1\right)
\cdot
\cos\left(\frac{x_1-1}{n_1-1}\left(z"_1 \pi -2\delta^\mathbf{z"}_1\right)+\delta^\mathbf{z"}_1\right)
$. \\
The factor $((z'_1+z"_1+1) \mod 2)$ occurs because 
when $z'_1+z"_1$ is odd, then 
$ \cos\left(\frac{n_1-1}{n_1-1}\left(z'_1 \pi -2\delta^\mathbf{z'}_1\right)+\delta^\mathbf{z'}_1\right)
\cdot
\cos\left(\frac{n_1-1}{n_1-1}\left(z"_1 \pi -2\delta^\mathbf{z"}_1\right)+\delta^\mathbf{z"}_1\right)$
$= \cos\left(z'_1 \pi -\delta^\mathbf{z'}_1\right) 
\cdot
 \cos\left(z"_1 \pi -\delta^\mathbf{z"}_1\right) 
$ 
$= - \cos\left( \delta^\mathbf{z'}_1\right) 
\cdot
 \cos\left( \delta^\mathbf{z"}_1\right) 
$.

We show now that the following holds:
\begin{equation}\label{eq:zerosumcoscos}
\sum_{x_j=1}^{n_j}
 \cos\left(\frac{x_j-1}{n_j-1}\left(z'_j \pi -2\delta^\mathbf{z'}_j\right)+\delta^\mathbf{z'}_j\right)
\cdot
\cos\left(\frac{x_j-1}{n_j-1}\left(z"_j \pi -2\delta^\mathbf{z"}_j\right)+\delta^\mathbf{z"}_j\right) =0
\end{equation}

Note that 
$\frac{n_j-x_j}{n_j-1}\left(z'_j \pi -2\delta^\mathbf{z'}_j\right)+\delta^\mathbf{z'}_j$
$=\left(1-\frac{x_j-1}{n_j-1}\right)\left(z'_j \pi -2\delta^\mathbf{z'}_j\right)+\delta^\mathbf{z'}_j$
$=z'_j \pi -2\delta^\mathbf{z'}_j-\frac{x_j-1}{n_j-1}\left(z'_j \pi -2\delta^\mathbf{z'}_j\right)+\delta^\mathbf{z'}_j$
$=z'_j \pi -\left(\frac{x_j-1}{n_j-1}\left(z'_j \pi -2\delta^\mathbf{z'}_j\right)+\delta^\mathbf{z'}_j\right)$
So if $z'_j$ is odd, then
$cos\left(\frac{n_j-x_j}{n_j-1}\left(z'_j \pi -2\delta^\mathbf{z'}_j\right)+\delta^\mathbf{z'}_j\right)$
$=-cos\left(\frac{x_j-1}{n_j-1}\left(z'_j \pi -2\delta^\mathbf{z'}_j\right)+\delta^\mathbf{z'}_j\right)$
and if it is even 
then 
$cos\left(\frac{n_j-x_j}{n_j-1}\left(z'_j \pi -2\delta^\mathbf{z'}_j\right)+\delta^\mathbf{z'}_j\right)$
$=cos\left(\frac{x_j-1}{n_j-1}\left(z'_j \pi -2\delta^\mathbf{z'}_j\right)+\delta^\mathbf{z'}_j\right)$.
For this reason, if the sum $z'_j+z"_j$ is odd, then 
$\sum_{x_j=1}^{n_j}
 \cos\left(\frac{x_j-1}{n_j-1}\left(z'_j \pi -2\delta^\mathbf{z'}_j\right)+\delta^\mathbf{z'}_j\right)
\cdot
\cos\left(\frac{x_j-1}{n_j-1}\left(z"_j \pi -2\delta^\mathbf{z"}_j\right)+\delta^\mathbf{z"}_j\right)=0$ because for each $x_j$ there is a complementary $n_j+1-x_j$ element of the same absolute value and inverted sign so that they cancel out. 

Let us consider the other cases now. 

Recall that
\begin{align*}
&\cos\left(\frac{x_j-1}{n_j-1}\left(z'_j \pi -2\delta^\mathbf{z'}_j\right)+\delta^\mathbf{z'}_j\right)
\cdot
\cos\left(\frac{x_j-1}{n_j-1}\left(z"_j \pi -2\delta^\mathbf{z"}_j\right)+\delta^\mathbf{z"}_j\right)
\\&=0.5\cos\left(\frac{x_j-1}{n_j-1}
\left((z'_j+z"_j) \pi -2(\delta^\mathbf{z'}_j + 
\delta^\mathbf{z"}_j)\right)
 +(\delta^\mathbf{z'}_j+\delta^\mathbf{z"}_j) \right)
\\&+
0.5\cos\left(\frac{x_j-1}{n_j-1}\left((z'_j-z"_j) \pi -2(\delta^\mathbf{z'}_j  - 
\delta^\mathbf{z"}_j)\right)
+(\delta^\mathbf{z'}_j-\delta^\mathbf{z"}_j)\right)
\end{align*}
So we need to prove that 
$$0=0.5\sum_{x_j=1}^{n_j}\cos\left(\frac{x_j-1}{n_j-1}
\left((z'_j+z"_j) \pi -2(\delta^\mathbf{z'}_j + 
\delta^\mathbf{z"}_j)\right)
 +(\delta^\mathbf{z'}_j+\delta^\mathbf{z"}_j) \right)
$$ $$+
0.5\sum_{x_j=1}^{n_j}\cos\left(\frac{x_j-1}{n_j-1}\left((z'_j-z"_j) \pi -2(\delta^\mathbf{z'}_j  - 
\delta^\mathbf{z"}_j)\right)
+(\delta^\mathbf{z'}_j-\delta^\mathbf{z"}_j)\right)
$$ 

From Trigonometry we know that
$\sum_{k=1}^{n}\cos(a+(k-1)\cdot d) =
  \cos(a+(n-1)\cdot d/2)\frac{\sin(n\cdot d/2)}{sin(d/2)}$.
  
This allows us to reformulate our problem as 
\begin{align*}
  0=&  0.5 
     \cos\left(
     \left( \delta^\mathbf{z'}_j+\delta^\mathbf{z"}_j  \right)
     +(n_j-1)\cdot \left(\frac{1}{n_j-1}
\left((z'_j+z"_j) \pi -2(\delta^\mathbf{z'}_j + 
\delta^\mathbf{z"}_j)\right)\right)/2\right)
\\ & 
\cdot\frac{
\sin\left(n_j\cdot \left(\frac{1}{n_j-1}
\left((z'_j+z"_j) \pi -2(\delta^\mathbf{z'}_j + 
\delta^\mathbf{z"}_j)\right)\right)/2\right)
}{ 
\sin\left( \left(\frac{1}{n_j-1}
\left((z'_j+z"_j) \pi -2(\delta^\mathbf{z'}_j + 
\delta^\mathbf{z"}_j)\right)\right)/2\right)
}
\end{align*}
\begin{align*}
\\ & +
 0.5 
     \cos\left(
     \left( \delta^\mathbf{z'}_j-\delta^\mathbf{z"}_j  \right)
     +(n_j-1)\cdot \left(\frac{1}{n_j-1}
\left((z'_j-z"_j) \pi -2(\delta^\mathbf{z'}_j - 
\delta^\mathbf{z"}_j)\right)\right)/2\right)
\\ & 
\cdot\frac{
\sin\left(n_j\cdot \left(\frac{1}{n_j-1}
\left((z'_j-z"_j) \pi -2(\delta^\mathbf{z'}_j - 
\delta^\mathbf{z"}_j)\right)\right)/2\right)
}{ 
\sin\left( \left(\frac{1}{n_j-1}
\left((z'_j-z"_j) \pi -2(\delta^\mathbf{z'}_j - 
\delta^\mathbf{z"}_j)\right)\right)/2\right)
}
\end{align*}
 which simplifies to 
 \begin{align*}
  0=&  0.5 
     \cos\left((z'_j+z"_j) \pi/2\right)
\cdot\frac{
\sin\left(n_j\cdot \left(\frac{1}{n_j-1}
\left((z'_j+z"_j) \pi -2(\delta^\mathbf{z'}_j + 
\delta^\mathbf{z"}_j)\right)\right)/2\right)
}{ 
\sin\left( \left(\frac{1}{n_j-1}
\left((z'_j+z"_j) \pi -2(\delta^\mathbf{z'}_j + 
\delta^\mathbf{z"}_j)\right)\right)/2\right)
}
\\ & +
 0.5 
     \cos\left((z'_j-z"_j) \pi/2\right)
\cdot\frac{
\sin\left(n_j\cdot \left(\frac{1}{n_j-1}
\left((z'_j-z"_j) \pi -2(\delta^\mathbf{z'}_j - 
\delta^\mathbf{z"}_j)\right)\right)/2\right)
}{ 
\sin\left( \left(\frac{1}{n_j-1}
\left((z'_j-z"_j) \pi -2(\delta^\mathbf{z'}_j - 
\delta^\mathbf{z"}_j)\right)\right)/2\right)
}
\end{align*}

 We have treated already the case when $(z'_j+z"_j)$ was odd.
 Now either $(z'_j+z"_j)/2$ is either even or odd. 
 So we get
  \begin{align*}
  0=&  \pm 0.5 
\cdot\frac{
\sin\left(n_j\cdot \left(\frac{1}{n_j-1}
\left((z'_j+z"_j) \pi -2(\delta^\mathbf{z'}_j + 
\delta^\mathbf{z"}_j)\right)\right)/2\right)
}{ 
\sin\left( \left(\frac{1}{n_j-1}
\left((z'_j+z"_j) \pi -2(\delta^\mathbf{z'}_j + 
\delta^\mathbf{z"}_j)\right)\right)/2\right)
}
\\ & +
 0.5 
\cdot\frac{
\sin\left(n_j\cdot \left(\frac{1}{n_j-1}
\left((z'_j-z"_j) \pi -2(\delta^\mathbf{z'}_j - 
\delta^\mathbf{z"}_j)\right)\right)/2\right)
}{ 
\sin\left( \left(\frac{1}{n_j-1}
\left((z'_j-z"_j) \pi -2(\delta^\mathbf{z'}_j - 
\delta^\mathbf{z"}_j)\right)\right)/2\right)
}
\end{align*}
With $+$ for the even case and $-$ for the odd one.  
This implies
  \begin{align*}
  0=&  \pm 0.5 
\sin\left(n_j\cdot \left(\frac{1}{n_j-1}
\left((z'_j+z"_j) \pi -2(\delta^\mathbf{z'}_j + 
\delta^\mathbf{z"}_j)\right)\right)/2\right)
\\ & \cdot
\sin\left( \left(\frac{1}{n_j-1}
\left((z'_j-z"_j) \pi -2(\delta^\mathbf{z'}_j - 
\delta^\mathbf{z"}_j)\right)\right)/2\right)
\\ & +
 0.5 
\sin\left(n_j\cdot \left(\frac{1}{n_j-1}
\left((z'_j-z"_j) \pi -2(\delta^\mathbf{z'}_j - 
\delta^\mathbf{z"}_j)\right)\right)/2\right)
\\ & \cdot
\sin\left( \left(\frac{1}{n_j-1}
\left((z'_j+z"_j) \pi -2(\delta^\mathbf{z'}_j + 
\delta^\mathbf{z"}_j)\right)\right)/2\right)
\end{align*}
By applying the formula for the product of sines of two angles we get
\Bem{
PM*cos((pi*(z1+z2)-2*(dv[1]+dvB[1]))/2+ 1/(n-1)*(pi*(z2)-2*(dvB[1])) ) - PM*cos((pi*(z1+z2)-2*(dv[1]+dvB[1]))/2+ 1/(n-1)*(pi*(z1)-2*(dv[1]))  )  +
 cos((pi*(z1-z2)-2*(dv[1]-dvB[1]))/2+ 1/(n-1)*(pi*(-z2)-2*(-dvB[1]))  )   -
 cos((pi*(z1-z2)-2*(dv[1]-dvB[1]))/2+ 1/(n-1)*(pi*(z1)-2*(dv[1]))  ) 
 }
 
  \begin{align*}
  0=&  \pm 
\cos\left(\left((z'_j+z"_j) \pi -2(\delta^\mathbf{z'}_j + 
\delta^\mathbf{z"}_j)\right)/2+ \frac{1}{n_j-1}
\left(z"_j \pi -2
\delta^\mathbf{z"}_j\right) \right)
\\ & - \pm 
\cos\left(\left((z'_j+z"_j) \pi -2(\delta^\mathbf{z'}_j + 
\delta^\mathbf{z"}_j)\right)/2+ \frac{1}{n_j-1}
\left(z'_j \pi -2\delta^\mathbf{z'}_j \right) \right)
\\ &    +
\cos\left(\left((z'_j+z"_j) \pi -2(\delta^\mathbf{z'}_j + 
\delta^\mathbf{z"}_j)\right)/2+ \frac{1}{n_j-1}
\left(-z"_j \pi -2
\delta^\mathbf{z"}_j\right) \right)
\\ & - 
\cos\left(\left((z'_j+z"_j) \pi -2(\delta^\mathbf{z'}_j + 
\delta^\mathbf{z"}_j)\right)/2+ \frac{1}{n_j-1}
\left(z'_j \pi -2\delta^\mathbf{z'}_j \right) \right)
\end{align*}

Now we recombine the first and the third, the second and the forth summand using the cosine sum formula and we get after simplification:
\Bem{
2* cos((PMPI+pi*z1 -2*dv[1] )/2)  * 
 cos((PMPI+pi*z2-2*dvB[1])/2+ 1/(n-1)*(pi*z2-2*  dvB[1]) )  -
2* cos( (PMPI+pi* z1 -2*dv[1])/2+ 1/(n-1)*(pi*(z1)-2*(dv[1]))    )    *
 cos((PMPI+pi*z2-2*dvB[1])/2 )
 }
 
\newcommand{\pmPI}{\pi_\pm}
  \begin{align*}
  0=&   
2\cos\left(\left( \pmPI+ z'_j \pi -2\delta^\mathbf{z'}_j \right)/2  \right)
\\ & \cdot
\cos\left(\left(\pmPI+z"_j \pi -2
\delta^\mathbf{z"}_j\right)/2+
\frac{1}{n_j-1}
\left(z"_j \pi -2\delta^\mathbf{z"}_j \right) \right)
\\ &    -
2\cos\left(\left( \pmPI+ z"_j \pi -2\delta^\mathbf{z"}_j \right)/2  \right)
\\ & \cdot
\cos\left(\left(\pmPI+z'_j \pi -2
\delta^\mathbf{z'}_j\right)/2+
\frac{1}{n_j-1}
\left(z'_j \pi -2\delta^\mathbf{z'}_j \right) \right)
\end{align*}
where $\pmPI$ is equal zero, if $\pm$ was $+$, and equals $\pi$ constant otherwise.  

Taking into account evenness/oddness of $z'_j,z"_j$ we can simplify the above to:
  \begin{align*}
  0=&   
\cos\left(  \delta^\mathbf{z'}_j    \right)
\cdot
\cos\left( -
\delta^\mathbf{z"}_j+
\frac{1}{n_j-1}
\left(z"_j \pi -2\delta^\mathbf{z"}_j \right) \right)
\\ &    -
\cos\left(  \delta^\mathbf{z"}_j   \right)
\cdot
\cos\left(-
\delta^\mathbf{z'}_j+
\frac{1}{n_j-1}
\left(z'_j \pi -2\delta^\mathbf{z'}_j \right) \right)
\end{align*}
which follows directly from equation \eref{eq:deltafromlambda} applied once to $\mathbf{z'}$ and once for $\mathbf{z"}$ vectors assuming that the eigenvalues are equal. 

So, in order to prove \eref{eq:mainN}, after the  substitution of \eref{eq:res1} into it, we are left with proving that 
\begin{align*}
  0=&  
  \sum_{x_2=1}^{n_2}\dots   \sum_{x_d=1}^{n_d} 
  \prod_{j=2}^d   
\cos\left(\frac{x_j-1}{n_j-1}\left(z'_j \pi -2\delta^\mathbf{z'}_j\right)+\delta^\mathbf{z'}_j\right)
\cdot
\cos\left(\frac{x_j-1}{n_j-1}\left(z"_j \pi -2\delta^\mathbf{z"}_j\right)+\delta^\mathbf{z"}_j\right)
\\ &
 (-2)\cdot  
 \cos\left( \delta^\mathbf{z'}_1\right)
\cdot
\cos\left( \delta^\mathbf{z"}_1\right) ((z'_1+z"_1+1) \mod 2)
\end{align*}
That is
\begin{align*}
  0=&(-2)\cdot  
 \cos\left( \delta^\mathbf{z'}_1\right)
\cdot
\cos\left( \delta^\mathbf{z"}_1\right) ((z'_1+z"_1+1) \mod 2)
  \sum_{x_3=1}^{n_3}\dots   \sum_{x_d=1}^{n_d} 
\\ &
  \prod_{j=3}^d   
\cos\left(\frac{x_j-1}{n_j-1}\left(z'_j \pi -2\delta^\mathbf{z'}_j\right)+\delta^\mathbf{z'}_j\right)
\cdot
\cos\left(\frac{x_j-1}{n_j-1}\left(z"_j \pi -2\delta^\mathbf{z"}_j\right)+\delta^\mathbf{z"}_j\right)
\\ &
  \sum_{x_2=1}^{n_2}
\cos\left(\frac{x_2-1}{n_2-1}\left(z'_j \pi -2\delta^\mathbf{z'}_2\right)+\delta^\mathbf{z'}_2\right)
\cdot
\cos\left(\frac{x_2-1}{n_2-1}\left(z"_2 \pi -2\delta^\mathbf{z"}_2\right)+\delta^\mathbf{z"}_2\right)
  \end{align*}
As already shown (by analogy to \eref{eq:zerosumcoscos}),
$  \sum_{x_2=1}^{n_2}
\cos\left(\frac{x_2-1}{n_2-1}\left(z'_j \pi -2\delta^\mathbf{z'}_2\right)+\delta^\mathbf{z'}_2\right)
\cdot
\cos\left(\frac{x_2-1}{n_2-1}\left(z"_2 \pi -2\delta^\mathbf{z"}_2\right)+\delta^\mathbf{z"}_2\right)
=0$
so the entire expression is zero. 
This completes the proof. 

\emph{Note that at this point we were assuming that we have at least two dimensions ($d>1$). 
As shown in Theorem in subsection \ref{subsec:onDimGrid}, 
all eigenvalues of a one-dimensional grid graph are different, 
so the case of equal eigenvalues does not need to be treated. 
}.   
 
So we have shown the validity of the Theorem \ref{th:normalizedLap}.

\subsection{Graphs without Inner Nodes}\label{subsec:generalNoInner}

 Such graphs will occur if one or more 
$n_i$ is equal two.

While proving the Theorem \ref{th:normalizedLap}, we have shown that 
assuming the form \eref{eq:eigenvectorcomponentN} of eigenvectors, 
the eigenvalue can be expressed in the form 
\eref{eq:lambdaN2dm1}. 
So it is easily seen that the same will hold in case of graphs without an inner node. 
Therefore the very same method of computation of $\lambda$s and $\delta$s can be applied and representation of eigenvalue and eighenvectors is the same. 

\begin{theorem}
The Theorem  \ref{th:normalizedLap} is applicable also to grid graphs without inner nodes. Eigenvalues and eigenvectors are the same. 
\end{theorem}


\subsection{Special Case - One-Dimnesional Grid Graphs}\label{subsec:onDimGrid}

\begin{theorem}
For one-dimensional grid graph of $n$ nodes we have eigenvalues of the form
\begin{equation}\label{eq:lambdanormalizedonedim}
\lambda_{[z]}=
 2 \left(\cos\left(\frac{\pi  z}{2 (n-1)}\right)\right)^2
\end{equation}  
with $z$ ranging from 0 to $n-1$. 
The corresponding eigenvectors  $\mathbf{v}_{[z]}$ are of the form 

\begin{equation}
\nu_{[z],[x]}= 
 (-1)^x  \cos\left(\frac{\pi z}{n-1} \left(x-1\right)\right) / s
\end{equation}
\noindent
where   $x$ is an integer such that $1\le x\le n$,
and $s=\sqrt{2}$ when $x=1$ or $x=n$, and $s=1$ otherwise.  
Then  
 $\mathbf{v}_{[z]}$ is a vector   
such that 
$\mathbf{v}_{[z],i}=\nu_{[z],i}$. 
\\ All eigenvalues are different. 
\end{theorem}
This can be inferred from Theorem \ref{th:normalizedLap} as follows:
According to 
 \eref{eq:lambdaN}
$$\lambda_{\mathbf{z}}=1+  
\cos\left(\frac{1}{n_1-1}\left(z_1 \pi -2\delta_1\right) \right)$$
and at the same time 
due to \eref{eq:lambdaN2dm1}
$$ \lambda_{\mathbf{z}}
 =
1+\cos\left(\frac{1}{n_1-1}\left(z_1 \pi -2\delta^\mathbf{z}_1\right)\right)
+\tan(\delta^\mathbf{z}_1)\sin\left(\frac{1}{n_1-1}\left(z_1 \pi -2\delta^\mathbf{z}_1\right)\right)$$
which imply 
$$\tan(\delta^\mathbf{z}_1)\sin\left(\frac{1}{n_1-1}\left(z_1 \pi -2\delta^\mathbf{z}_1\right)\right)=0$$
As $-\pi < \delta^\mathbf{z}_1 <\pi$ is assumed, this can be true only for 
$\delta^\mathbf{z}_1=0$. Hence the above result. 
So in this case we have an explicit formula for eigenvalue and eigenvector. 
Beside this, 
equation \eref{eq:lambdanormalizedonedim} implies that all eigenvalues are different because 
because  for $z\in \{0,\dots,n-1\}$ the expression 
$\frac{\pi  z}{2 (n-1)}$ 
rnges from $0$ to $\pi/2$ and in this interval $\cos$ function is stricktly decreasing. 

\subsection{Special Case - Regular d-Dimensional Grid Graphs}
\label{subsec:regularGrid}

In a regular $d$-dimensional grid, that is with each dimension identical, some of the eigenvalues may be computed as 
\begin{equation}
\lambda_{[z]}=
 2 \left(\cos\left(\frac{\pi  z}{2 (n-1)}\right)\right)^2
\end{equation}  
with $z$ ranging from 0 to $n-1$. 
The corresponding eigenvectors  $\mathbf{v}_{[z]}$ are of the form 

\begin{equation}
\nu_{[z],[x_1,...,x_d]}= 
\sqrt{deg([x_1,...,x_d])}
\prod_{j=1}^d
 (-1)^x_j  \cos\left(\frac{\pi z}{n-1} \left(x_j-1\right)\right) 
\end{equation}
\noindent
where   $x_j$ is an integer such that $1\le x_j\le n$,
 and  $deg([x_1,...,x_d])$ is the degree of the node 
characterised by $ [x_1,...,x_d] $. This degree can be computed as 
$d+\sum_{j=1}^d(x_j\ne 1 \land x_j \ne n)$. 
Then  
 $\mathbf{v}_{[z]}$ is a vector such that 
$\mathbf{v}_{[z],i}=\nu_{[z],i}$. 
 This result is related to assuming same value of all $z_i=z$. 

As all $n_i=n$, we get for each $\delta$
due to \eref{eq:lambdaN2dm1}
$$ \lambda_{\mathbf{z}}
 =
1+\cos\left(\frac{1}{n-1}\left(z \pi -2\delta^\mathbf{z}_1\right)\right)
+\tan(\delta^\mathbf{z}_1)\sin\left(\frac{1}{n-1}\left(z \pi -2\delta^\mathbf{z}_1\right)\right)$$
which implies that all $\delta$ must be identical, equal to some $\delta$. 

Therefore, according  to 
 \eref{eq:lambdaN}
$$\lambda_{\mathbf{z}}=1+  
\cos\left(\frac{1}{n-1}\left(z \pi -2\delta\right) \right)$$
Using the same reasoning as in previous subsection we get $\delta=0$ which implies the explicit formulas presented. 

\section{Some Properties of Normalized Laplacian for a multidimensional  unweighted grid graph}\label{sec:NOLproperties}

The formula \eref{eq:lambdaN} confirms that the normalized Laplacian 
ranges from 0 to 2 for a grid graph.  
 
Another interesting aspect of the grid graph eigenvalues is whether or not they are uniformly distributed.  Compressive Spectral Cluster Analysis assumes the uniformity. 
By inspecting the histograms obtained from a simulation study based of the formulas for eigenvectors, it is visible that the uniformity is not granted for grid graphs. . 

\figVer{
\begin{figure}
\centering
 (a)\includegraphics[width=0.4\textwidth]{\figaddr{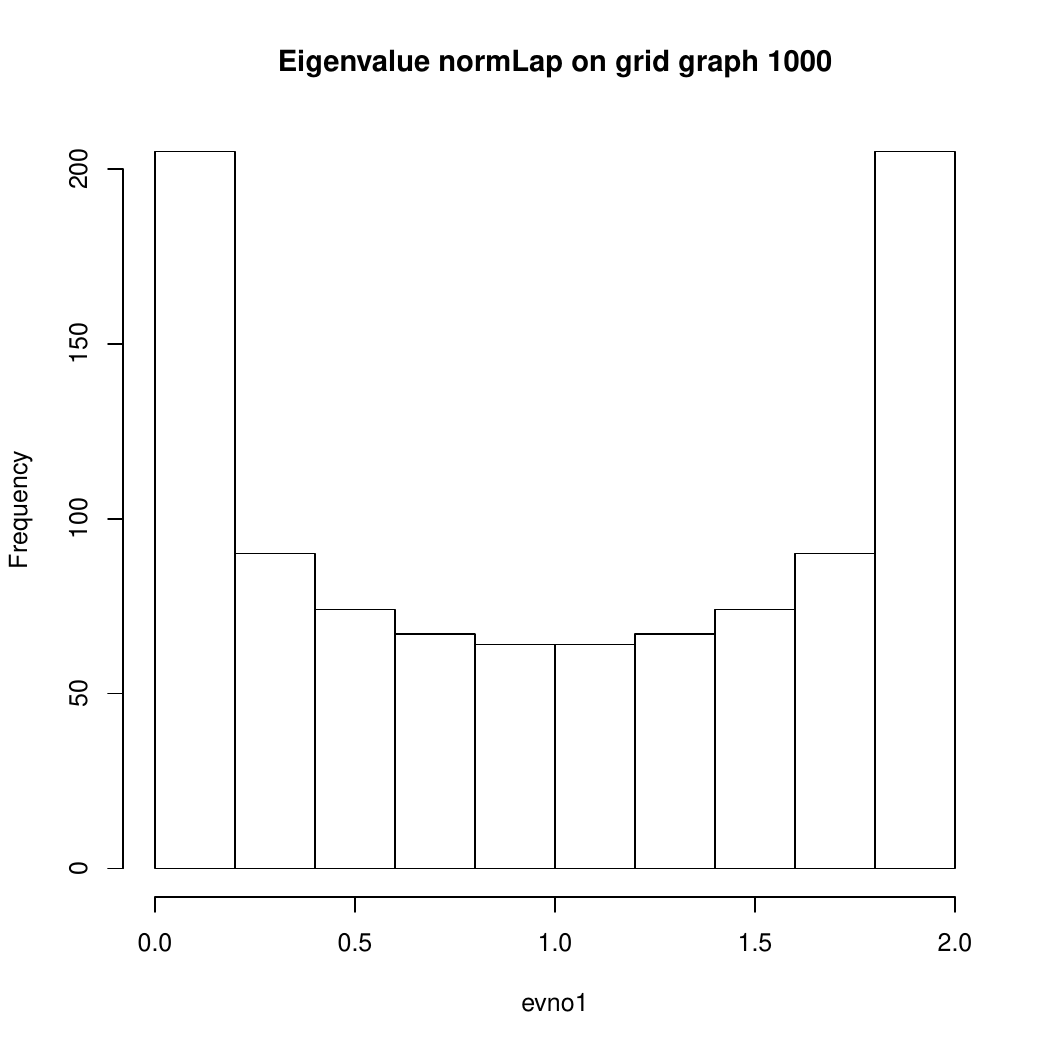}}  %
 (b)\includegraphics[width=0.4\textwidth]{\figaddr{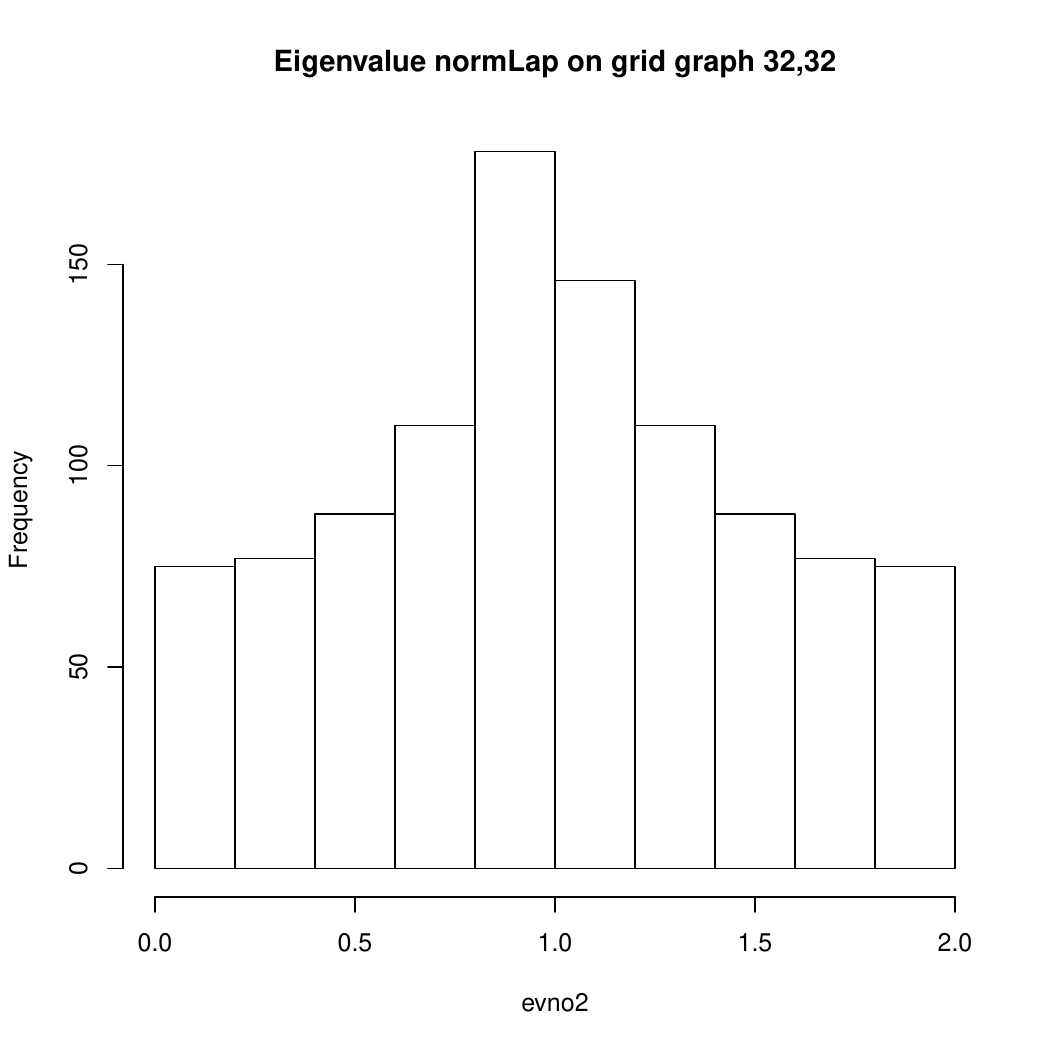}}  %
\\ (c)\includegraphics[width=0.4\textwidth]{\figaddr{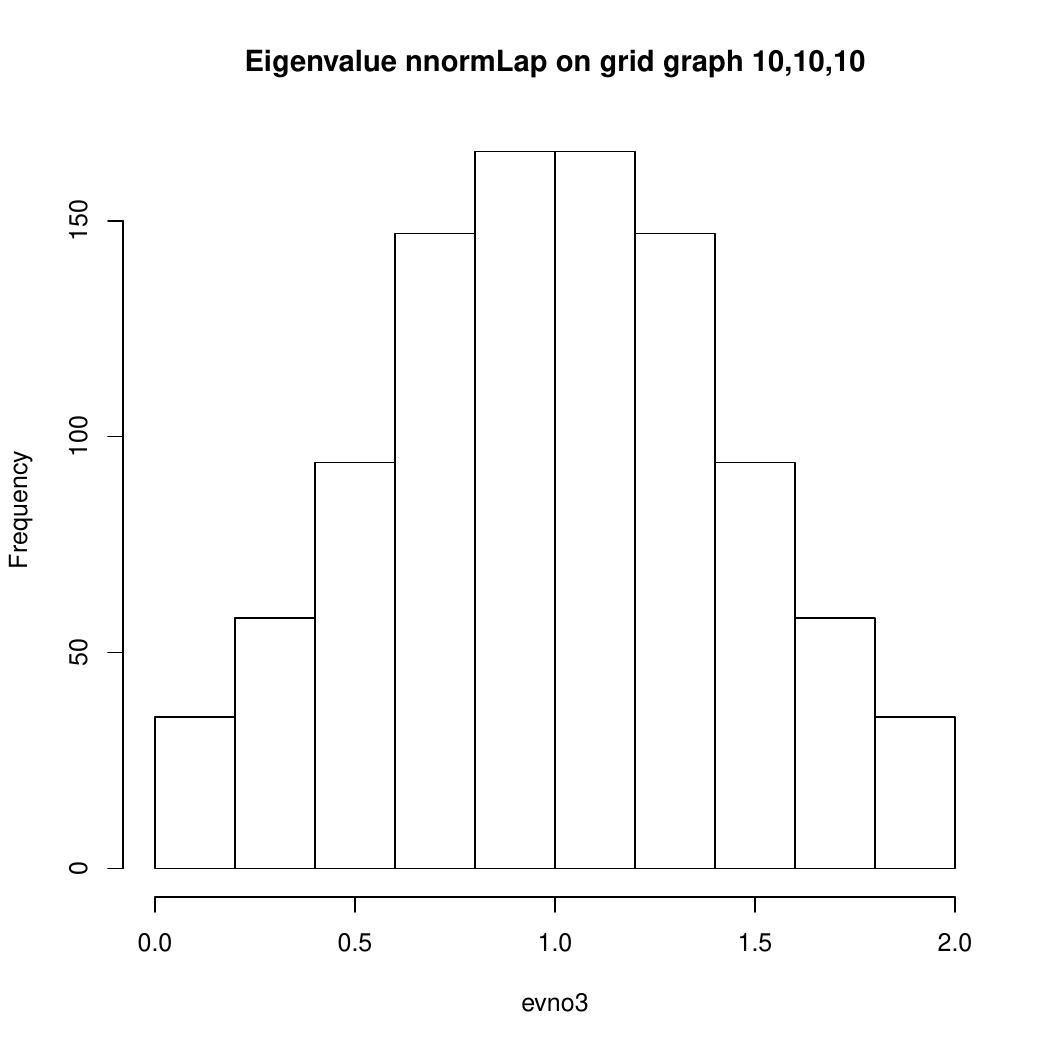}}  %
 (d)\includegraphics[width=0.4\textwidth]{\figaddr{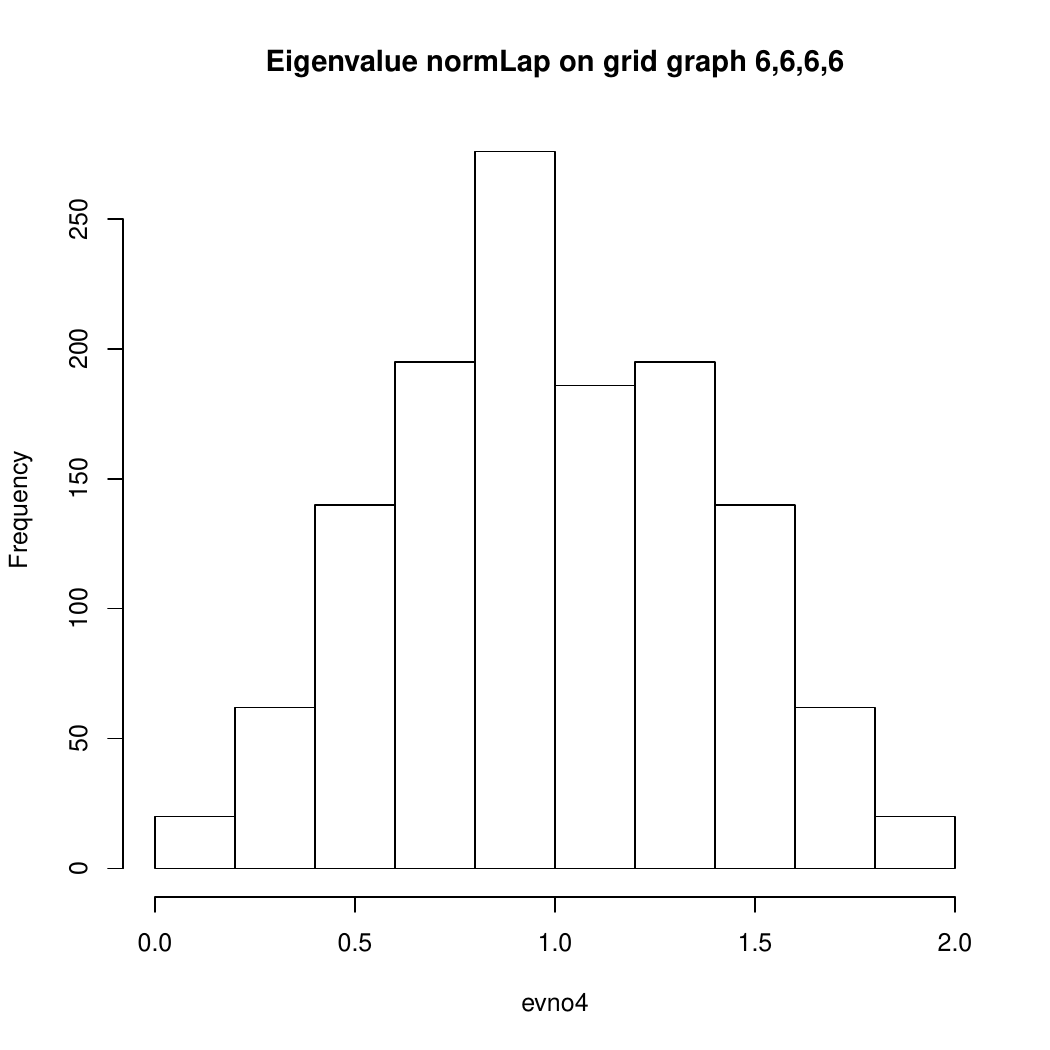}}  %
\caption{The histograms of  eigenvalues of normalized Laplacians of grid graphs of approximately 1,000 nodes.
(a) 1-dimensional grid graph, 
(b) 2-dimensional grid graph, 
(c) 3-dimensional grid graph, 
(d) 4-dimensional grid graph. 
}\label{fig:evno1000nodes}
\end{figure}

In Figure \ref{fig:evno1000nodes} you see the histograms of eigenvalue for grid graphs of approximately 1,000 nodes with dimensionality ranging between 1 and 4. 
The distributions do not resemble uniform distribution, at least for these small graphs.  
Rather, a similarity can be seen to the respective histograms  
of combinatorial Laplacians from  Figure \ref{fig:evco1000nodes}

\begin{figure}
\centering
 (a)\includegraphics[width=0.4\textwidth]{\figaddr{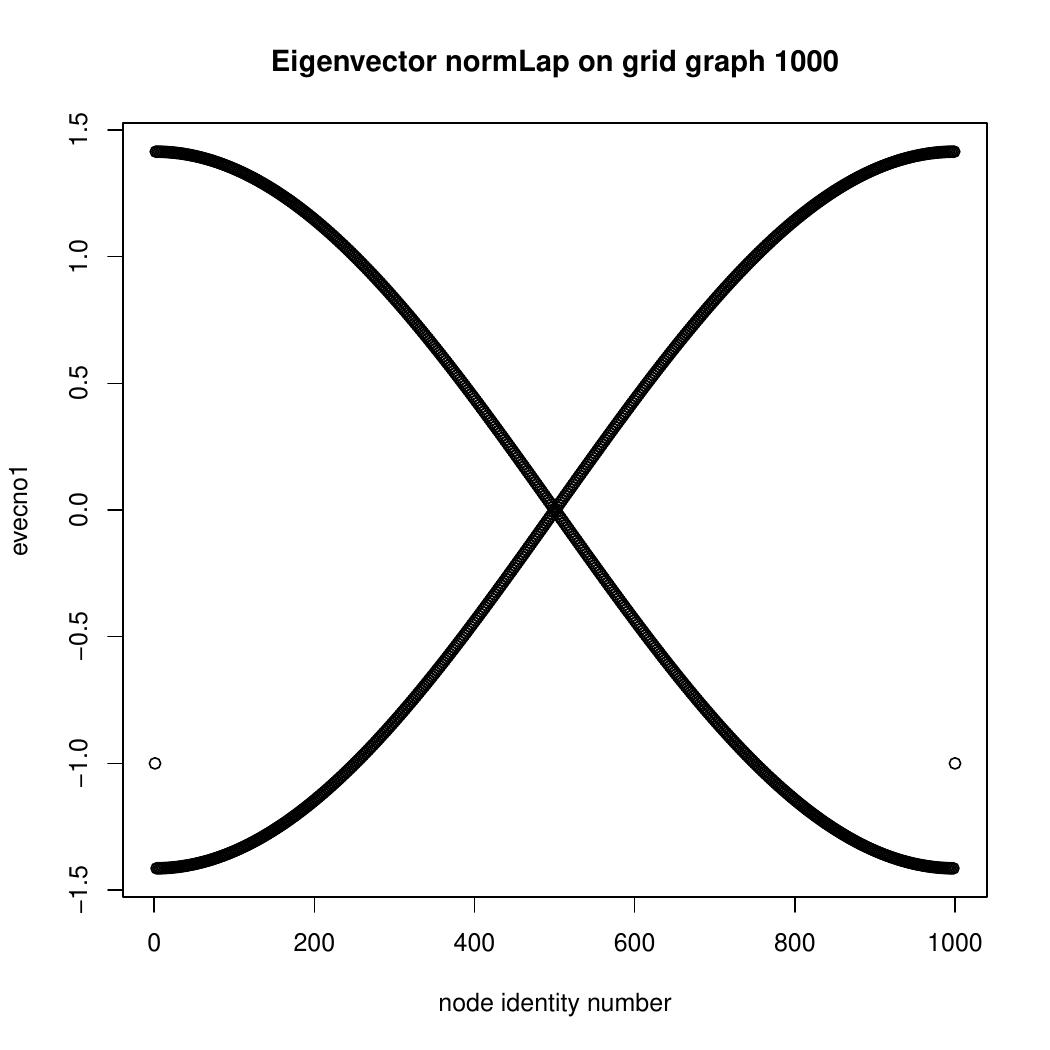}}  %
 (b)\includegraphics[width=0.4\textwidth]{\figaddr{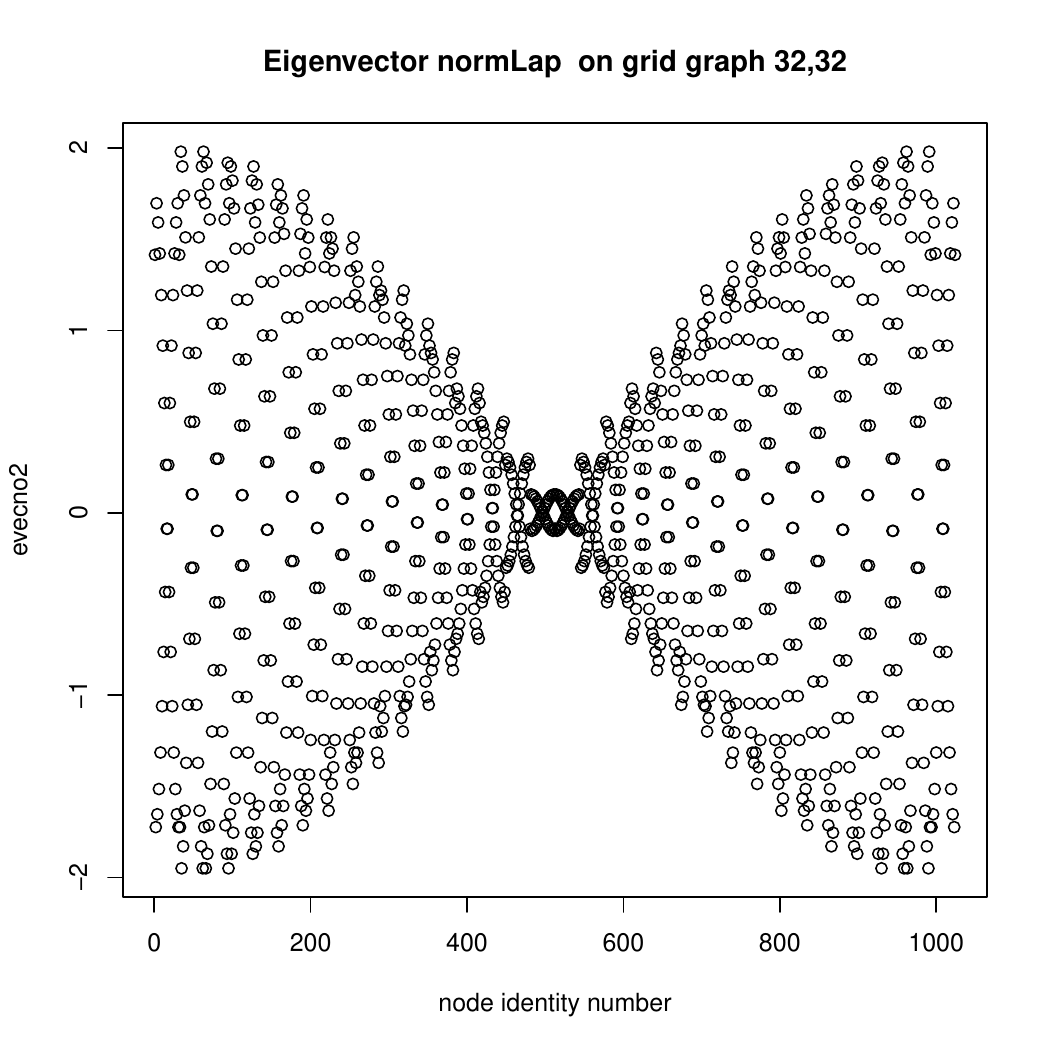}}  %
\\ (c)\includegraphics[width=0.4\textwidth]{\figaddr{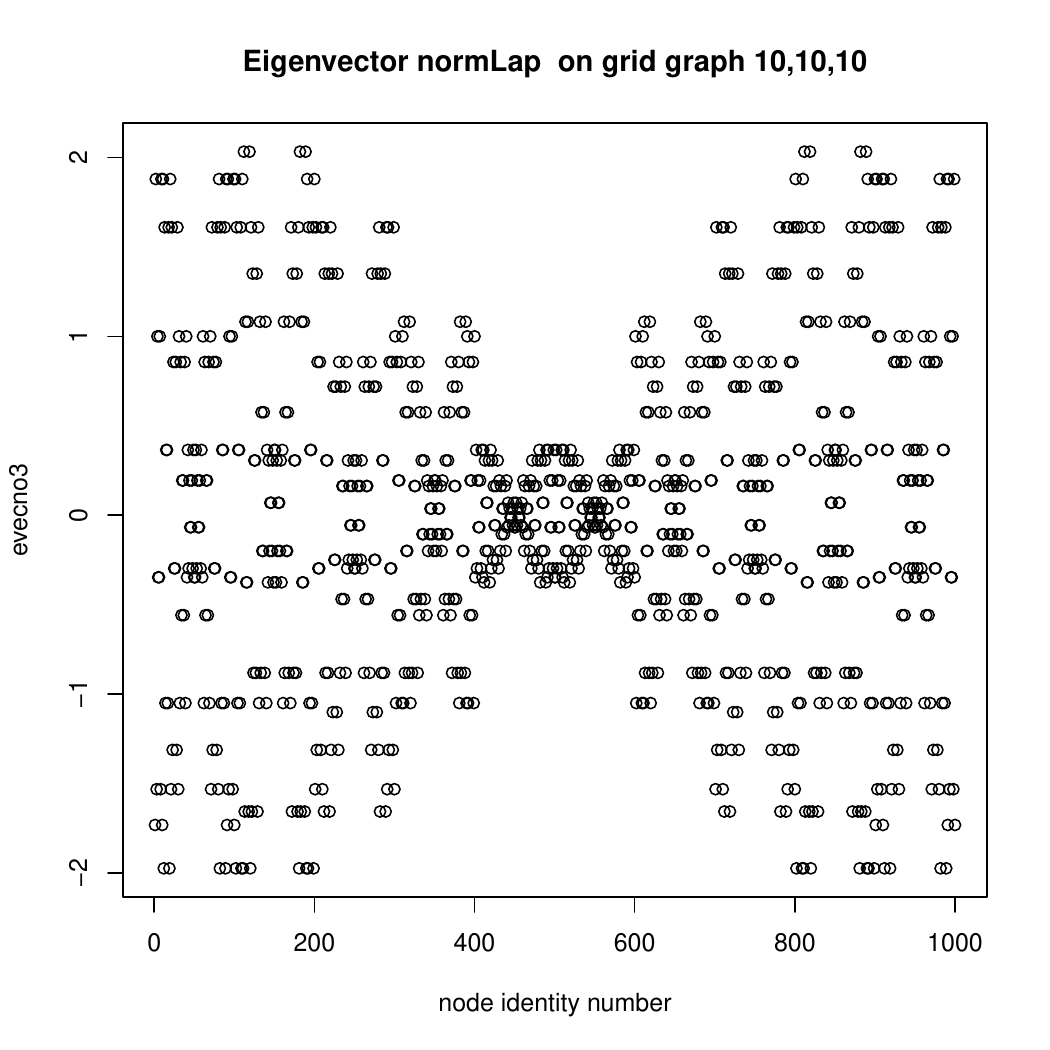}}  %
 (d)\includegraphics[width=0.4\textwidth]{\figaddr{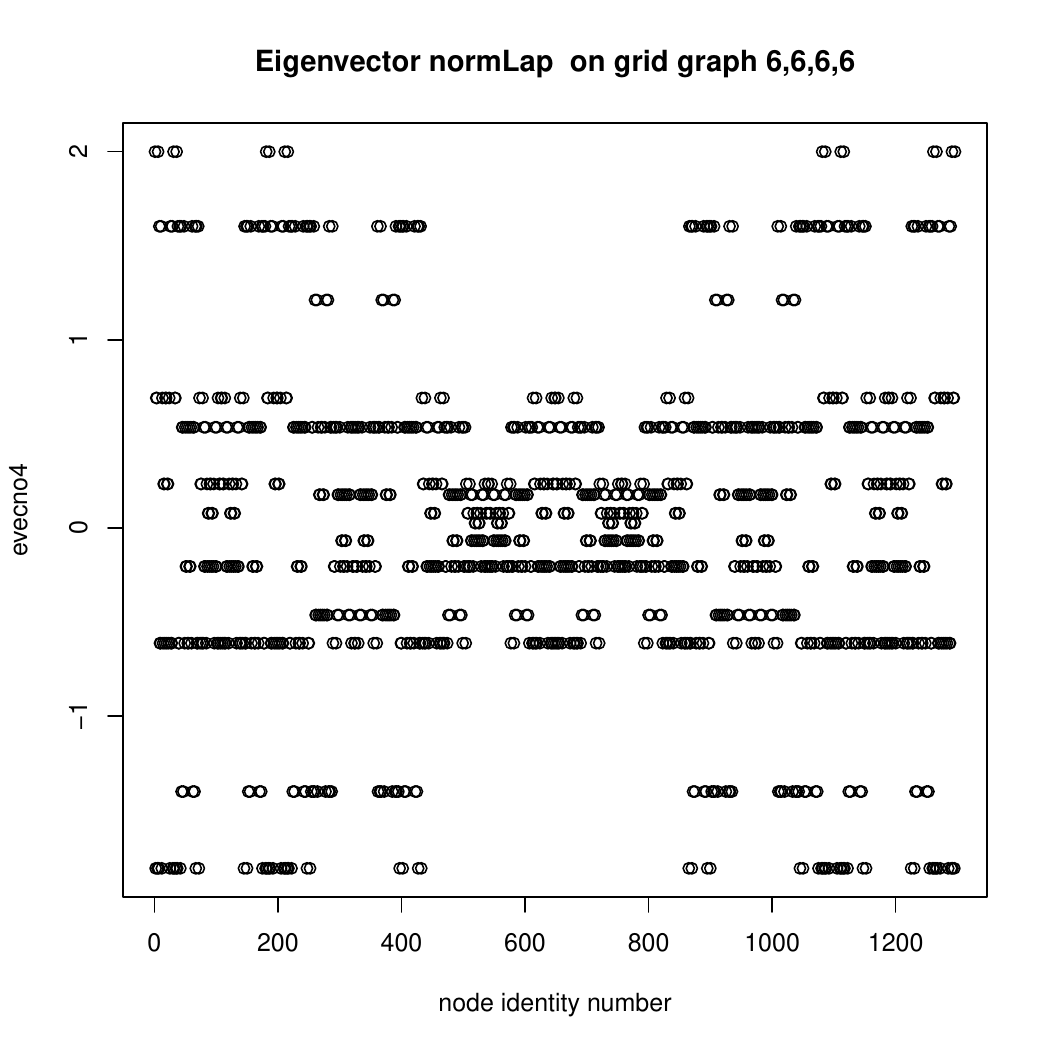}}  %
\caption{The plots of sample  eigenvectors of normalized Laplacians of grid graphs of approximately 1,000 nodes.
(a) 1-dimensional grid graph, $\z=[1]$, 
(b) 2-dimensional grid graph, $\z=[1,1]$, 
(c) 3-dimensional grid graph, $\z=[1,1,1]$, 
(d) 4-dimensional grid graph, $\z=[1,1,1,1]$. 
}\label{fig:evecno1000nodes}
\end{figure}

In Figure \ref{fig:evecno1000nodes} you see sample eigenvectors of the same graphs. 
Though aesthetic similarity can be seen to the respective plots for combinatorial Laplacians from Figure   \ref{fig:evecco1000nodes}, impact of sign alteration and of the shifts can be perceived.

\begin{figure}
\centering
 (a)\includegraphics[width=0.4\textwidth]{\figaddr{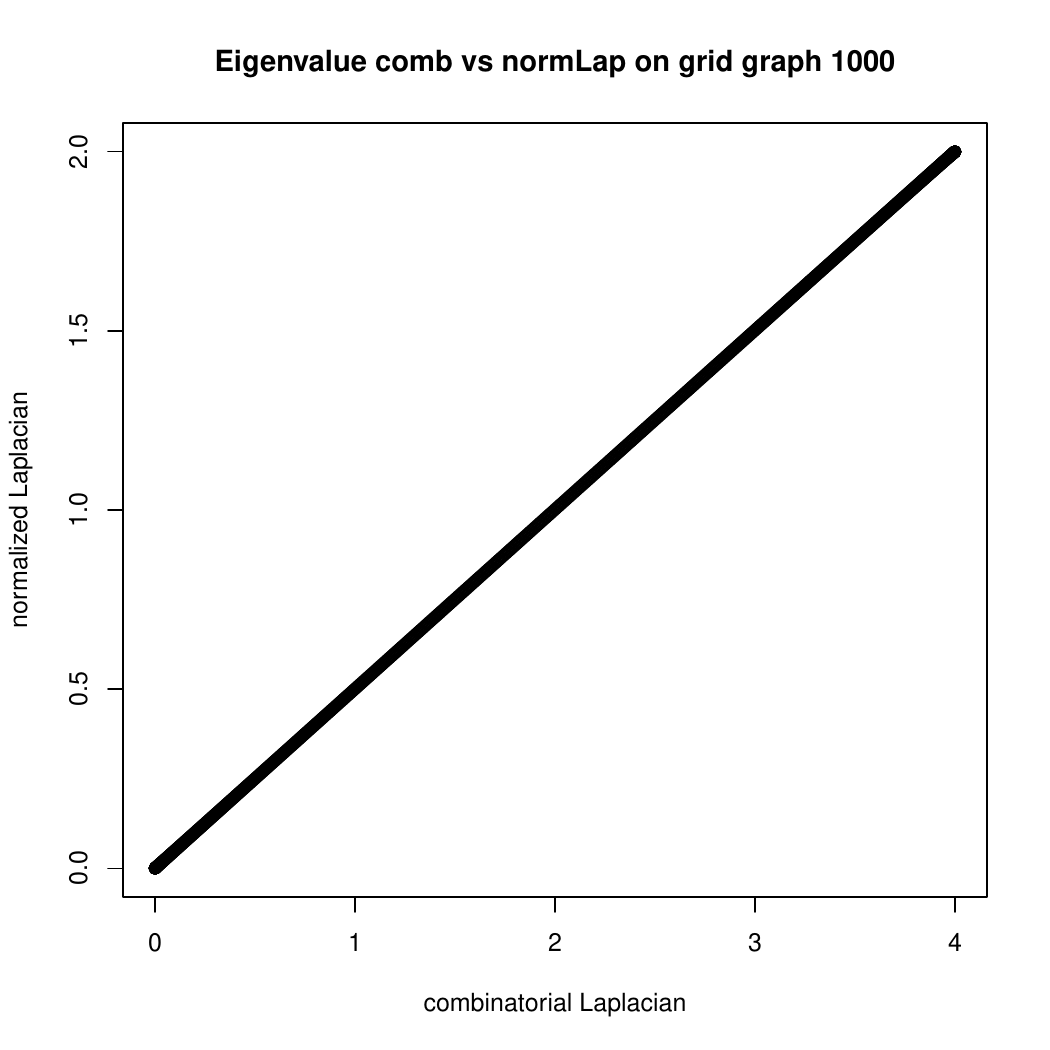}}  %
 (b)\includegraphics[width=0.4\textwidth]{\figaddr{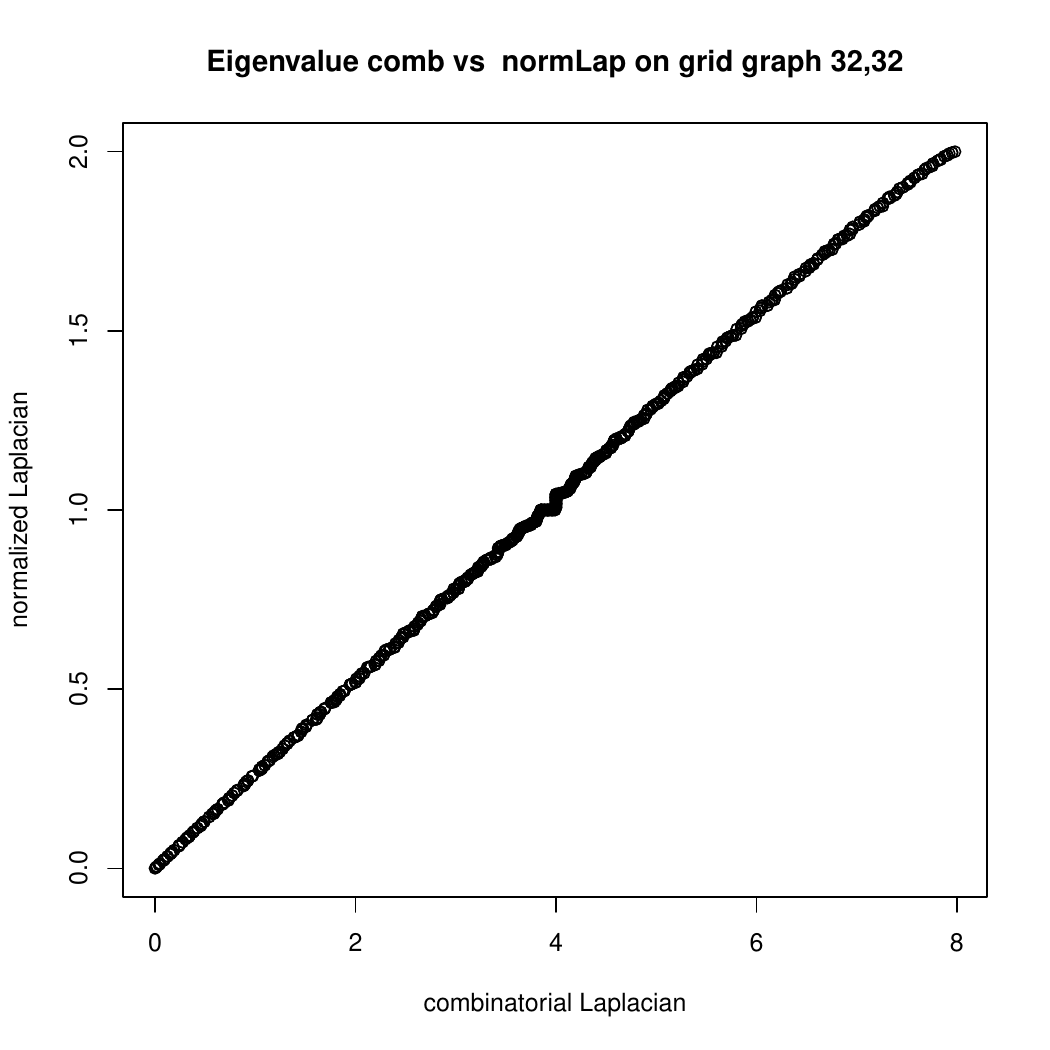}}  %
\\ (c)\includegraphics[width=0.4\textwidth]{\figaddr{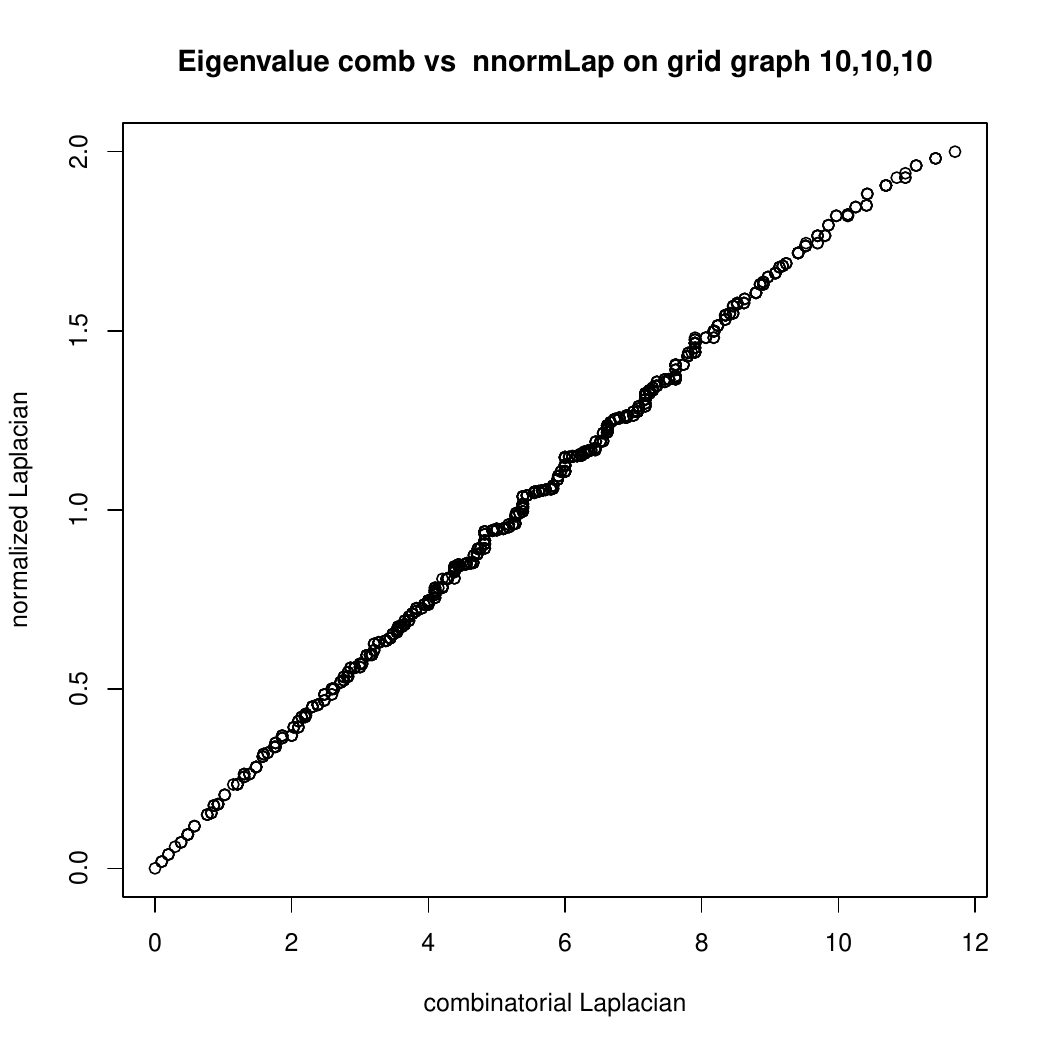}}  %
 (d)\includegraphics[width=0.4\textwidth]{\figaddr{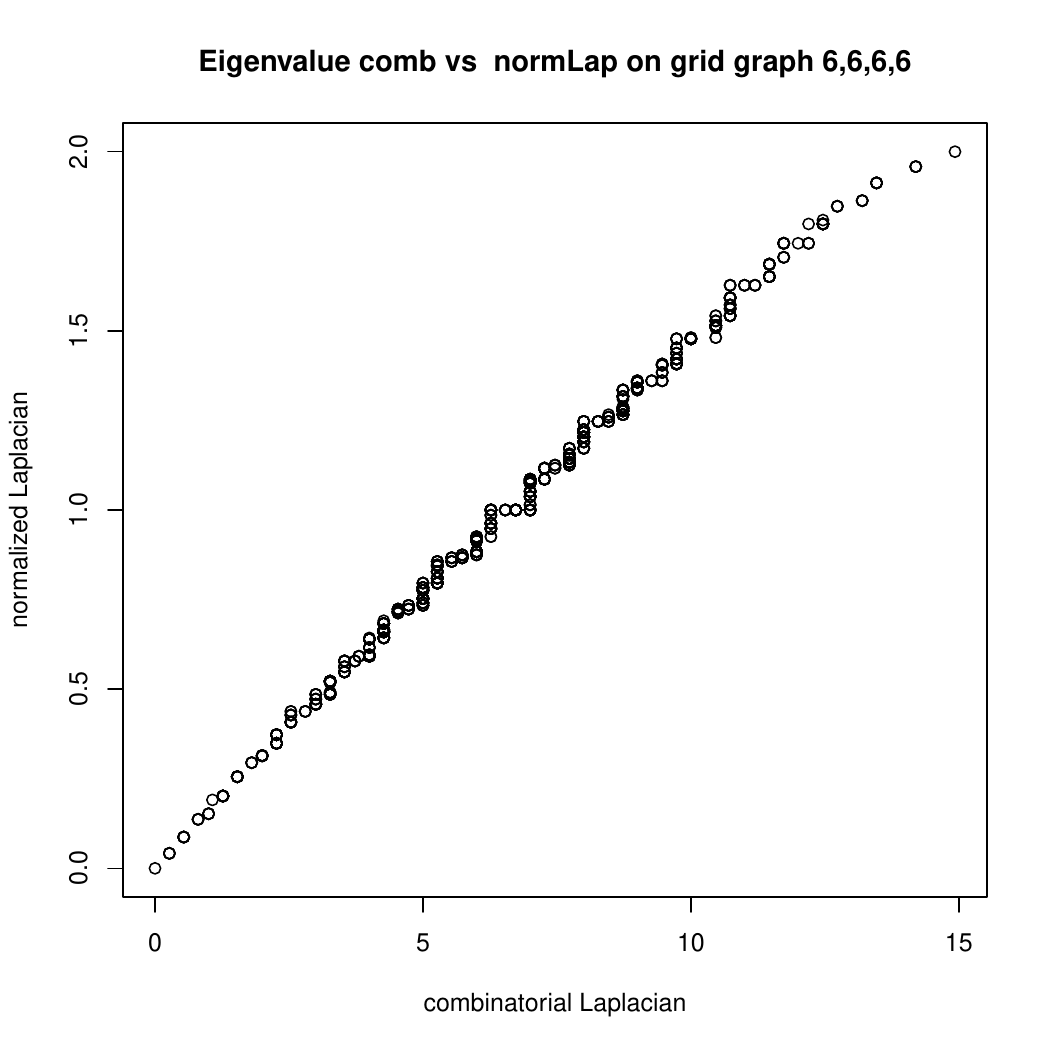}}  %
\caption{The plots of relationships of combinatorial and  normalized Laplacian eigenvalues of grid graphs of approximately 1,000 nodes.
(a) 1-dimensional grid graph, 
(b) 2-dimensional grid graph, 
(c) 3-dimensional grid graph, 
(d) 4-dimensional grid graph. 
}\label{fig:evcono1000nodes}
\end{figure}

The Figure \ref{fig:evcono1000nodes} illustrates the relationship between eigenvalues of combinatorial and normalized Laplacians of very same grid graphs. They do not deviate too much from one another (up to a scaling factor). 
So the observation that the non-uniformity pertains for combinatorial Laplacians of increasing grid graphs is also valid for normalized Laplacians. 

\begin{figure}
\centering
 (a)\includegraphics[width=0.4\textwidth]{\figaddr{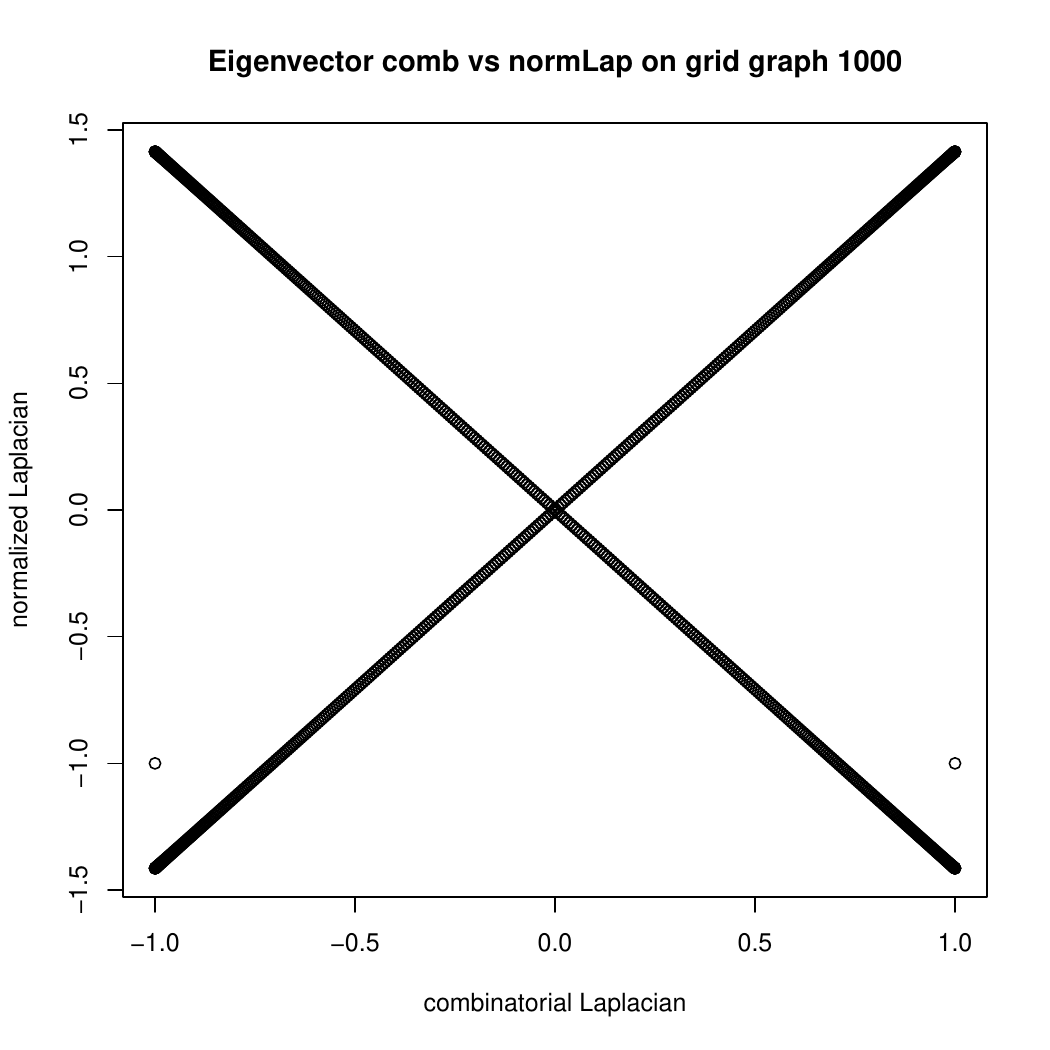}}  %
 (b)\includegraphics[width=0.4\textwidth]{\figaddr{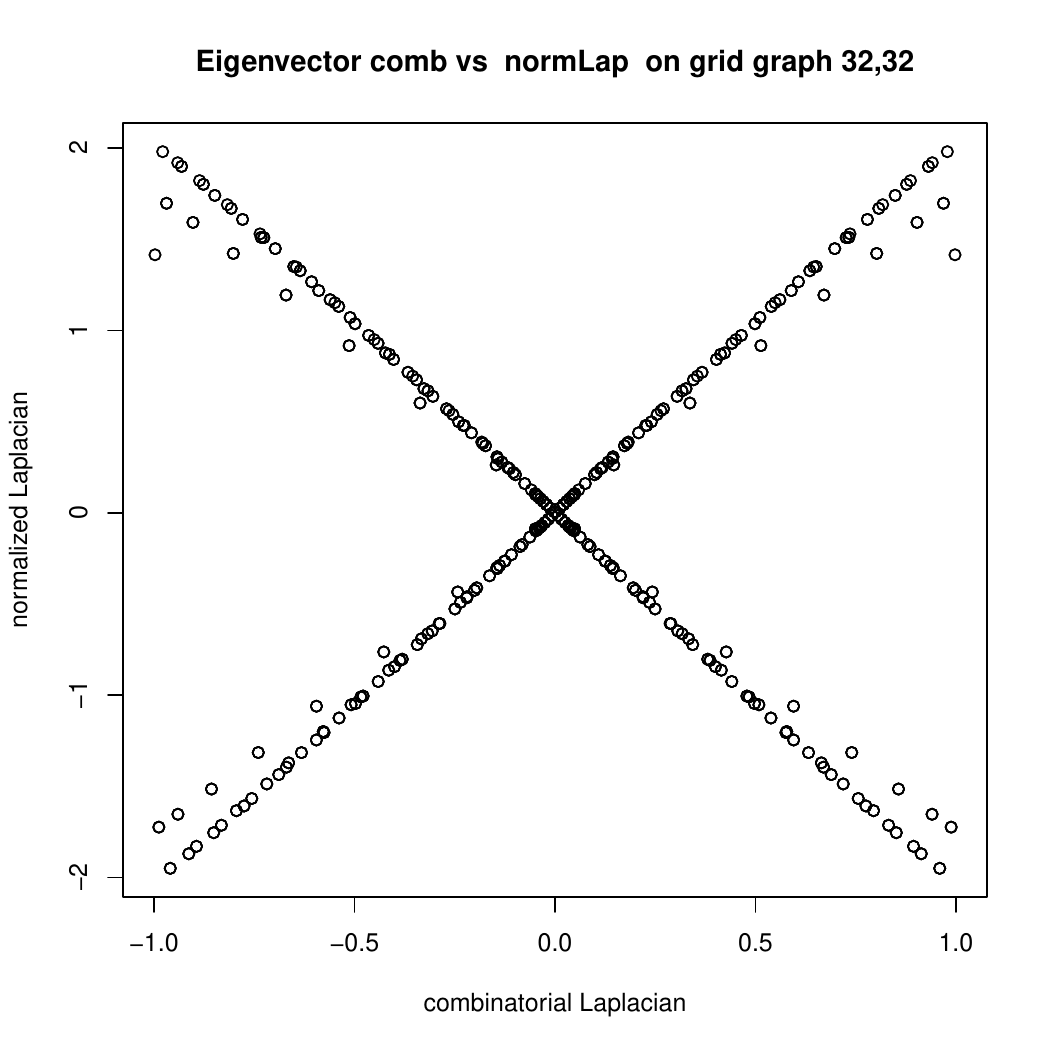}}  %
\\ (c)\includegraphics[width=0.4\textwidth]{\figaddr{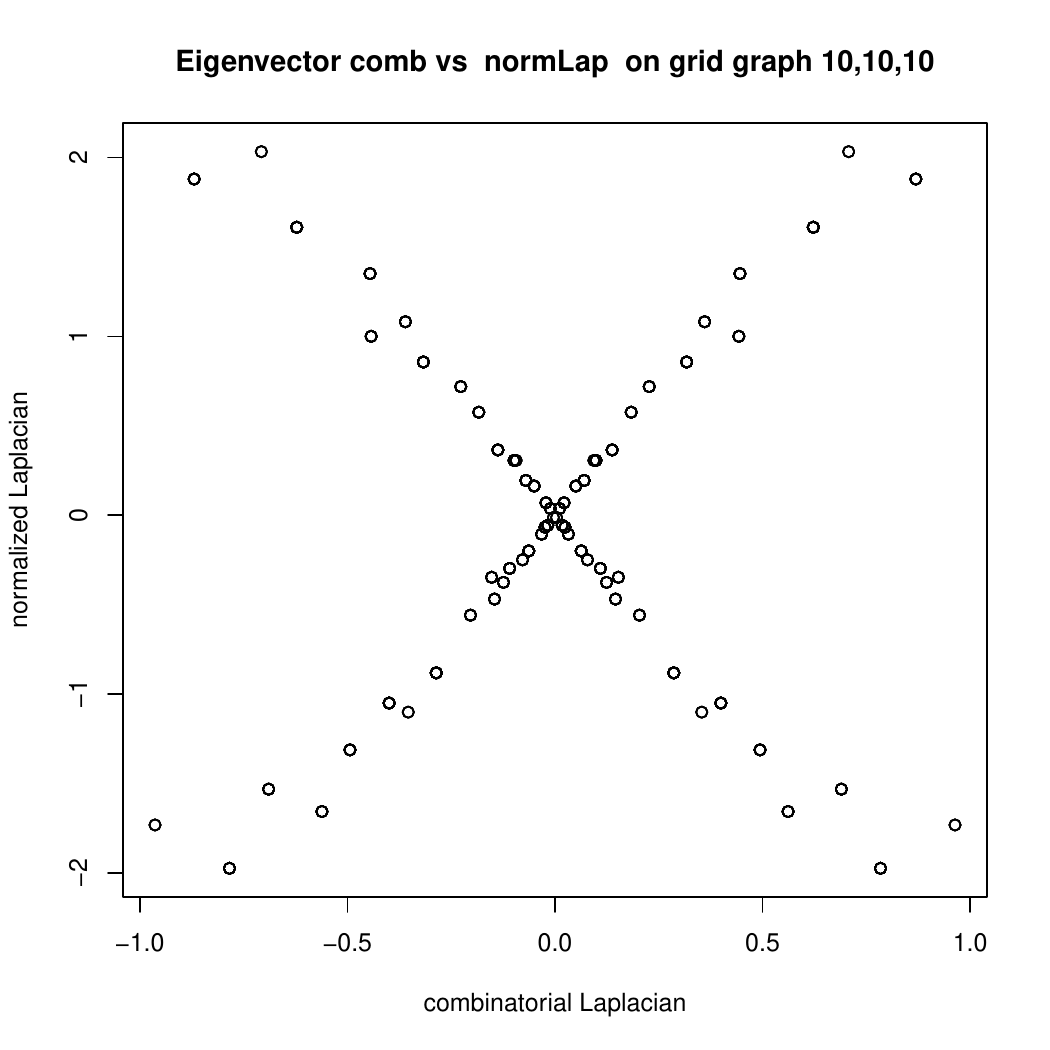}}  %
 (d)\includegraphics[width=0.4\textwidth]{\figaddr{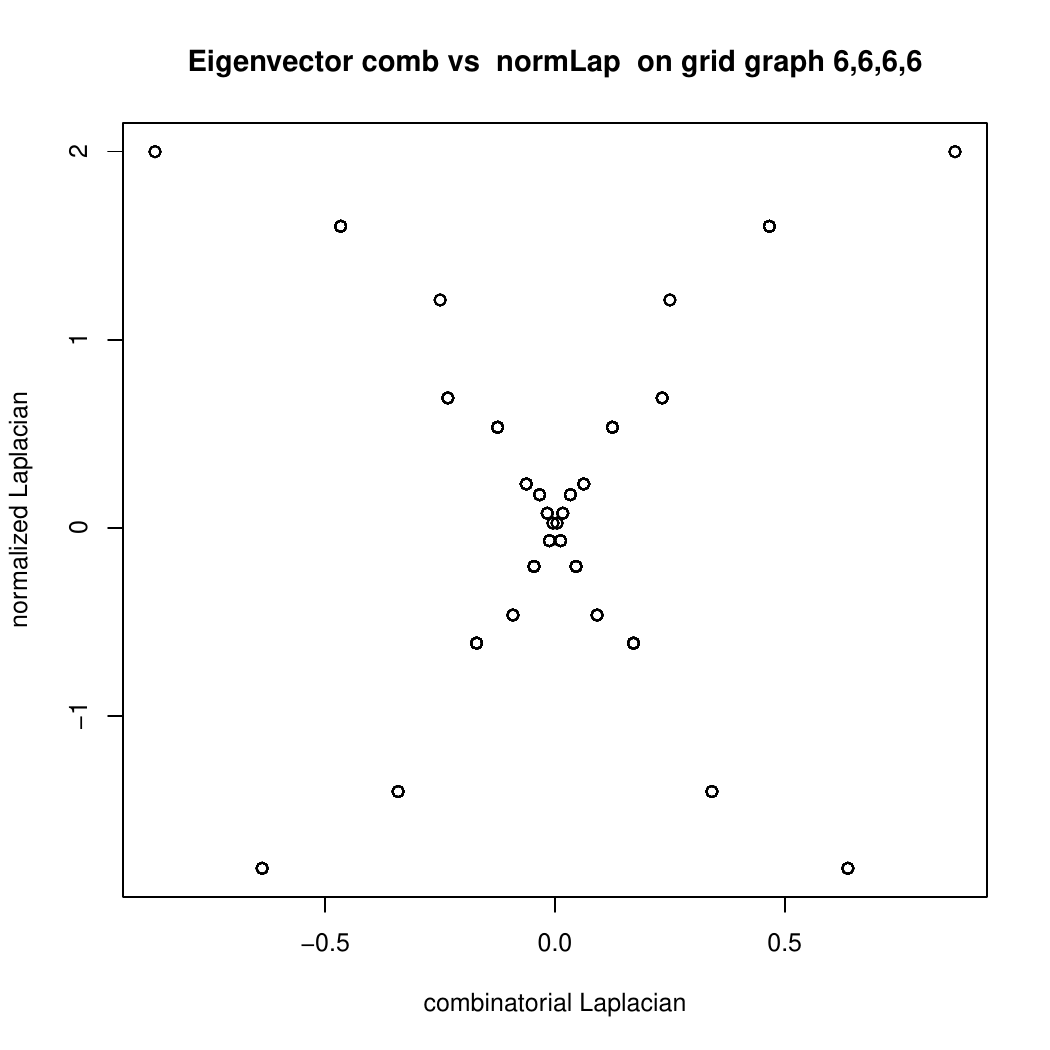}}  %
\caption{The plots of relationships of sample combinatorial and  normalized Laplacian eigenvectors of grid graphs of approximately 1,000 nodes.
(a) 1-dimensional grid graph, $\z=[1]$, 
(b) 2-dimensional grid graph, $\z=[1,1]$, 
(c) 3-dimensional grid graph, $\z=[1,1,1]$, 
(d) 4-dimensional grid graph, $\z=[1,1,1,1]$. 
}\label{fig:eveccono1000nodes}
\end{figure}

The Figure \ref{fig:eveccono1000nodes} illustrates the relationship between sample eigenvectors of combinatorial and normalized Laplacians of very same grid graphs.
One may get the impression   that they really do not differ too much (if one ignores the signs and some deviating points). 

\begin{figure}
\centering
 (a)\includegraphics[width=0.25\textwidth]{\figaddr{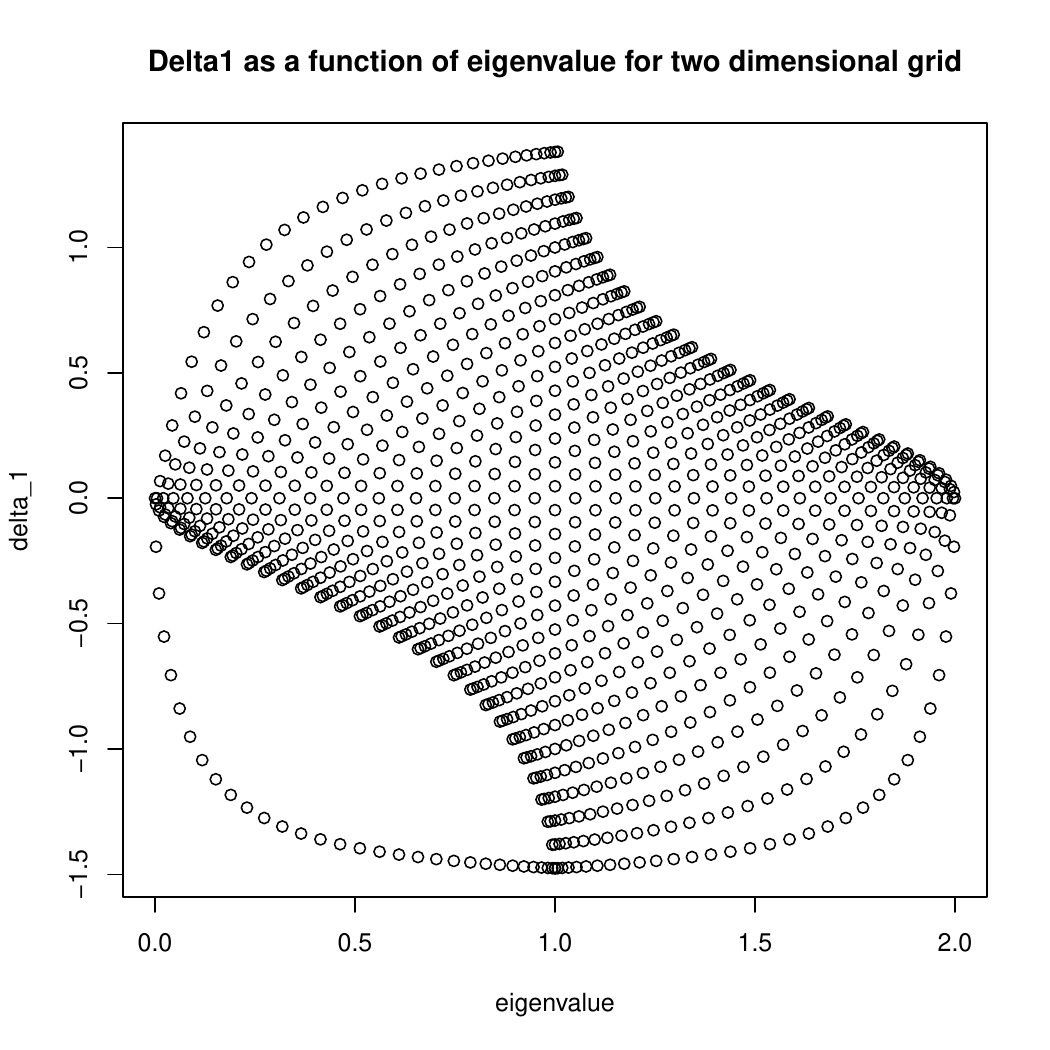}}  %
 (b)\includegraphics[width=0.25\textwidth]{\figaddr{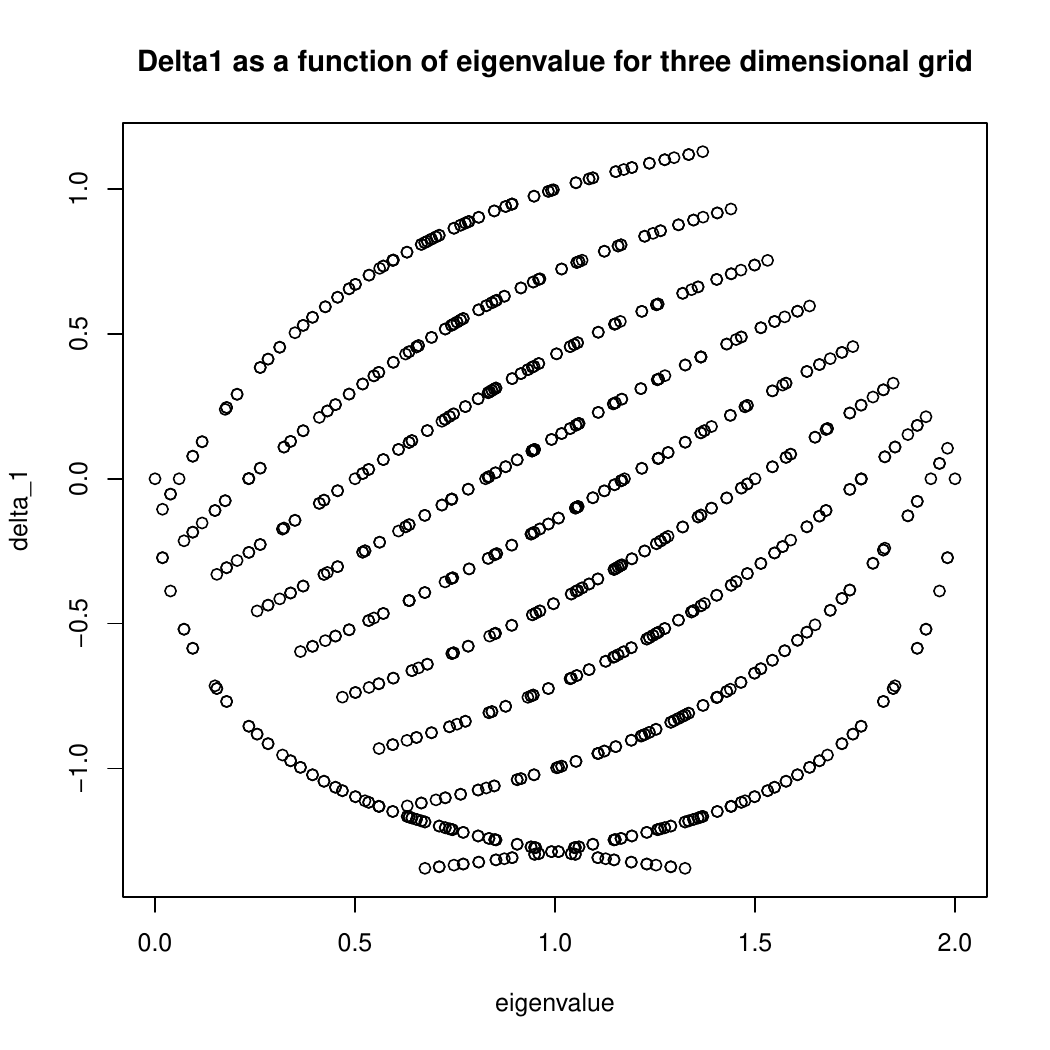}}  %
 (c)\includegraphics[width=0.25\textwidth]{\figaddr{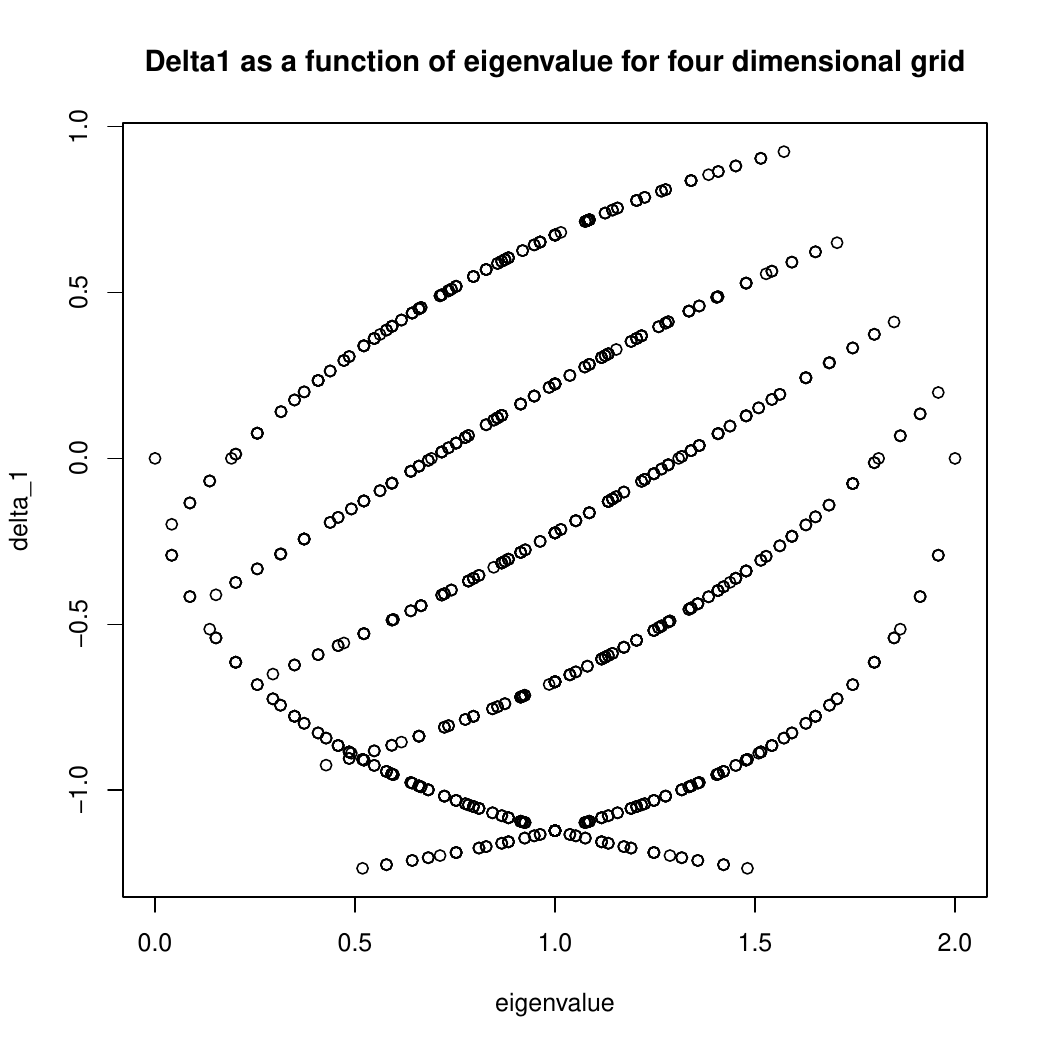}}  %
\caption{The plots of relationships of   normalized Laplacian eigenvalues
$\lambda$ 
 and shifts  $\delta$of grid graphs of approximately 1,000 nodes.
(a) 2-dimensional grid graph, 
(b) 3-dimensional grid graph, 
(c) 4-dimensional grid graph. 
}\label{fig:evnodeltanodes}
\end{figure}

Figure \ref{fig:evnodeltanodes} illustrates the relationship between eigenvalues and the shifts of normalized Laplacians in grid graphs.
This relationship seems not to be simplistic and may at least partially explain why we did not find a closed-form solution  for identifying eigenvalues and shifts. 
We did not plot such a relationship for one-dimensional grid as in this case shifts are equal zero.

\begin{figure}
\centering
 (a)\includegraphics[width=0.25\textwidth]{\figaddr{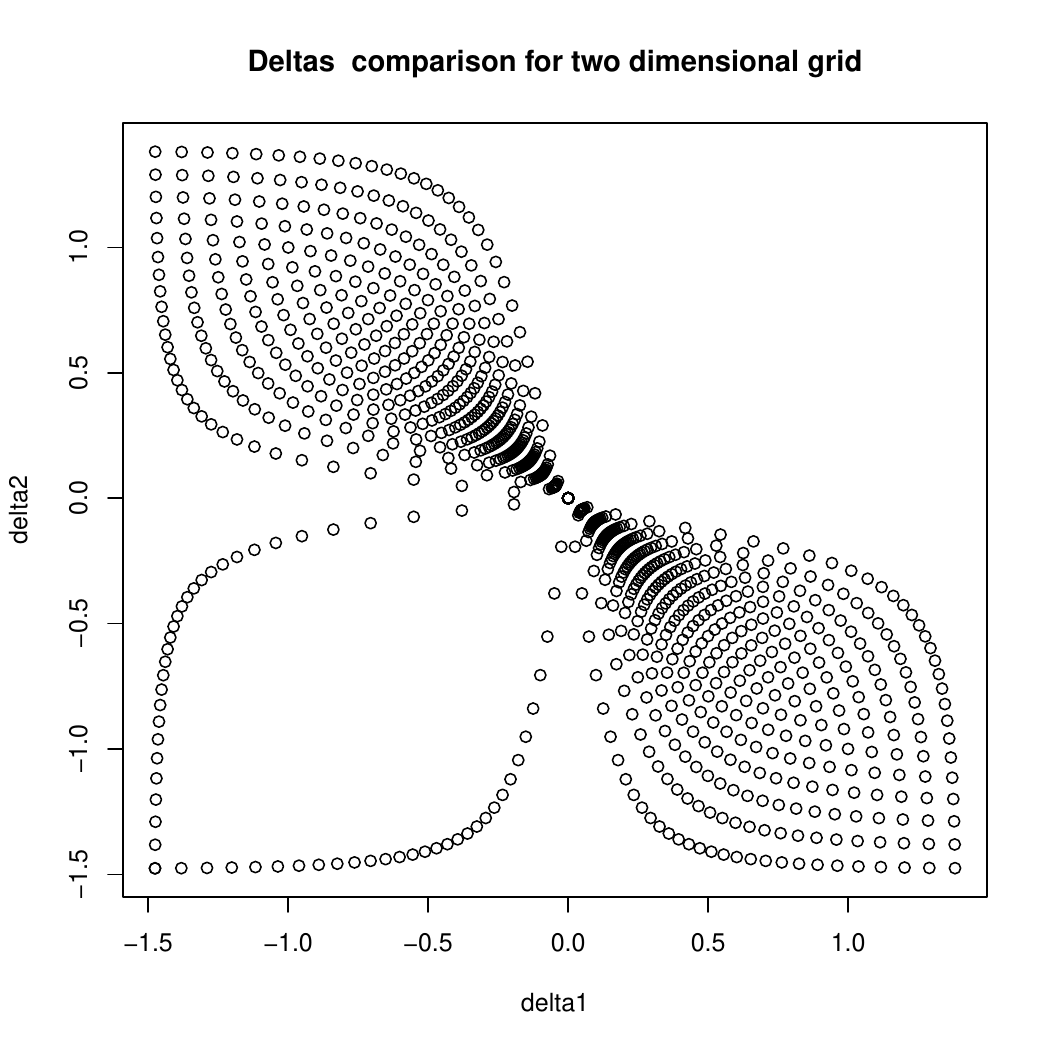}}  %
 (b)\includegraphics[width=0.25\textwidth]{\figaddr{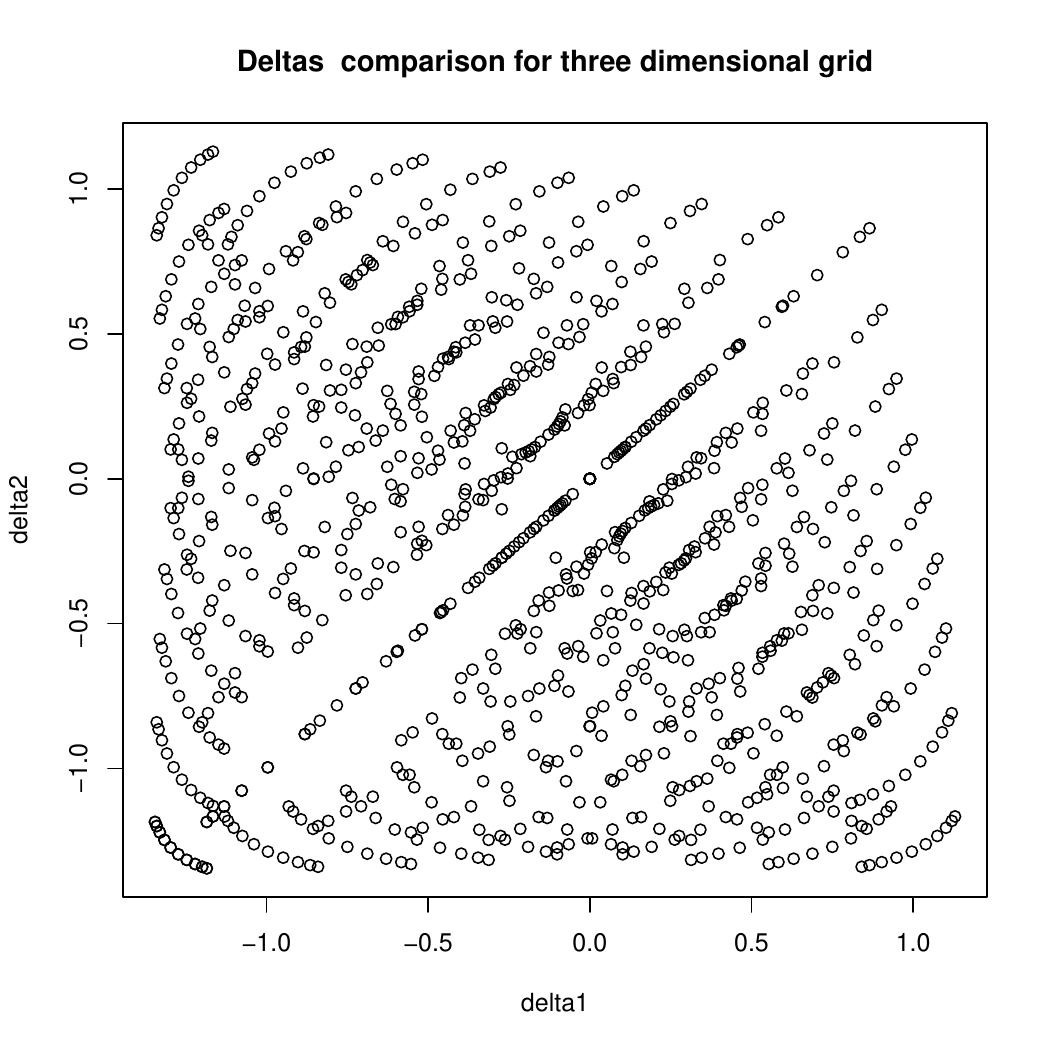}}  %
 (c)\includegraphics[width=0.25\textwidth]{\figaddr{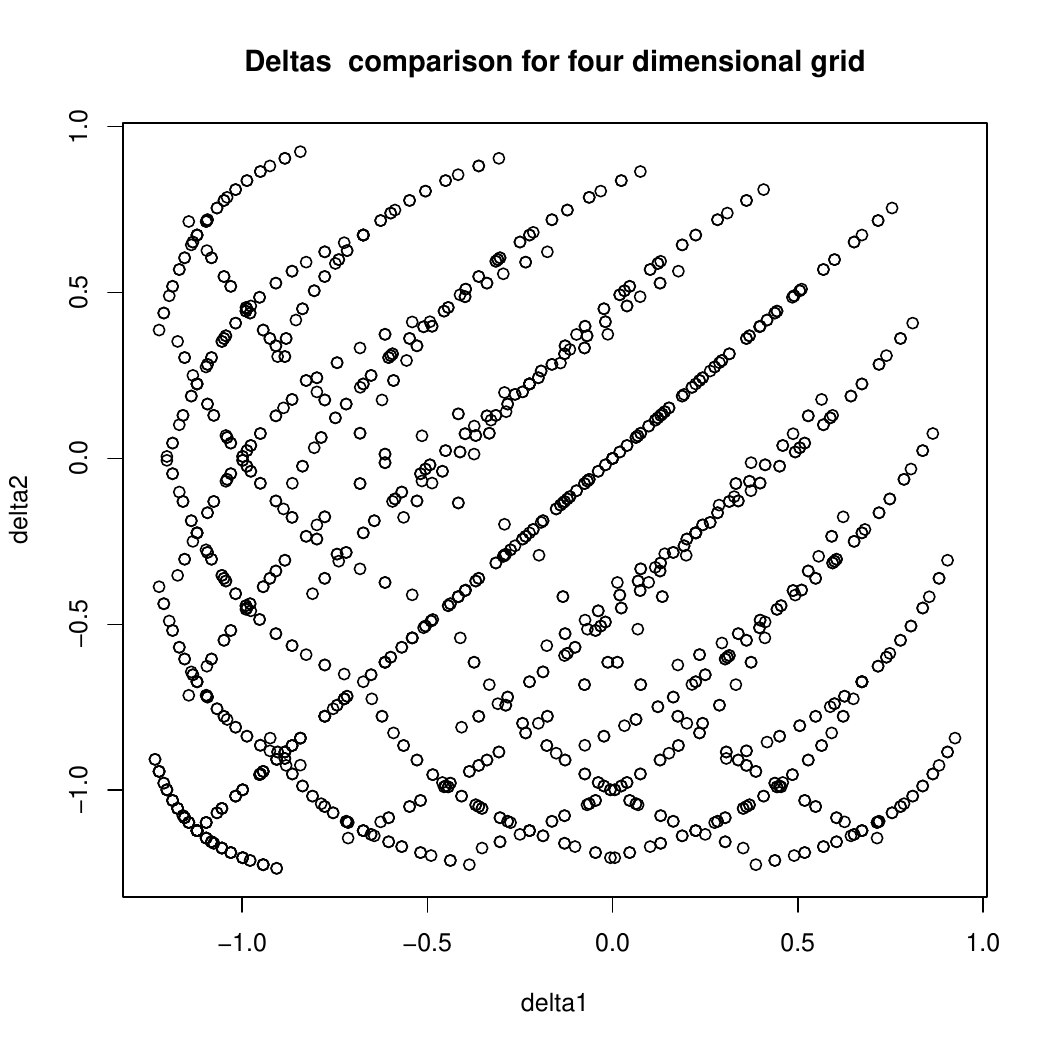}}  %
\caption{The plots of relationships of   normalized Laplacian   shifts of grid graphs of approximately 1,000 nodes in the first two dimensions.
(a) 2-dimensional grid graph, 
(b) 3-dimensional grid graph, 
(c) 4-dimensional grid graph. 
}\label{fig:deltas}
\end{figure}

Figure \ref{fig:deltas} depicts the relationship of shifts in various dimensions. 
In this case between the first and the second dimension of the grid. 
This relationship seems not to be a trivial one, though of its own beauty. The impact of presence of other dimensions can be clearly seen.  

}

Finally, we shall pose the question how the cumulative distribution of eigenvalues 
of a normalized Laplacian of a grid graph would look like in the limit (when the number of nodes grows). 
If we keep in mind that $|delta_j|<\pi$, then for sufficiently high $n_j$ and $z_j$
the contribution of $delta_j$ in the equation \eref{eq:lambdaN} will vanish and
$$ \lambda_{\mathbf{z}} \approx  1+\frac{1}{d} \sum_{j=1}^d 
\cos\left(\frac{z_j \pi}{n_j} \right)
=
 1+\frac{1}{d} \sum_{j=1}^d 
\left( 1-2\sin^2\left(\frac{z_j \pi}{2n_j} \right)\right)
=
 2-2\frac{1}{d} \sum_{j=1}^d 
 \sin^2\left(\frac{z_j \pi}{2n_j} \right) 
$$ 
which up to a scaling factor resembles the defining equation of 
combinatorial Laplacian eigenvalue \eref{eq:lambdaC}.
This means that the in the limit behavior of normalized Laplacian eigenvalues will resemble that combinatorial Laplacian eigenvalues that is \emph{no uniformity can be assumed}. 

\figVer{
\begin{figure}
\centering
 (a)\includegraphics[width=0.45\textwidth]{\figaddr{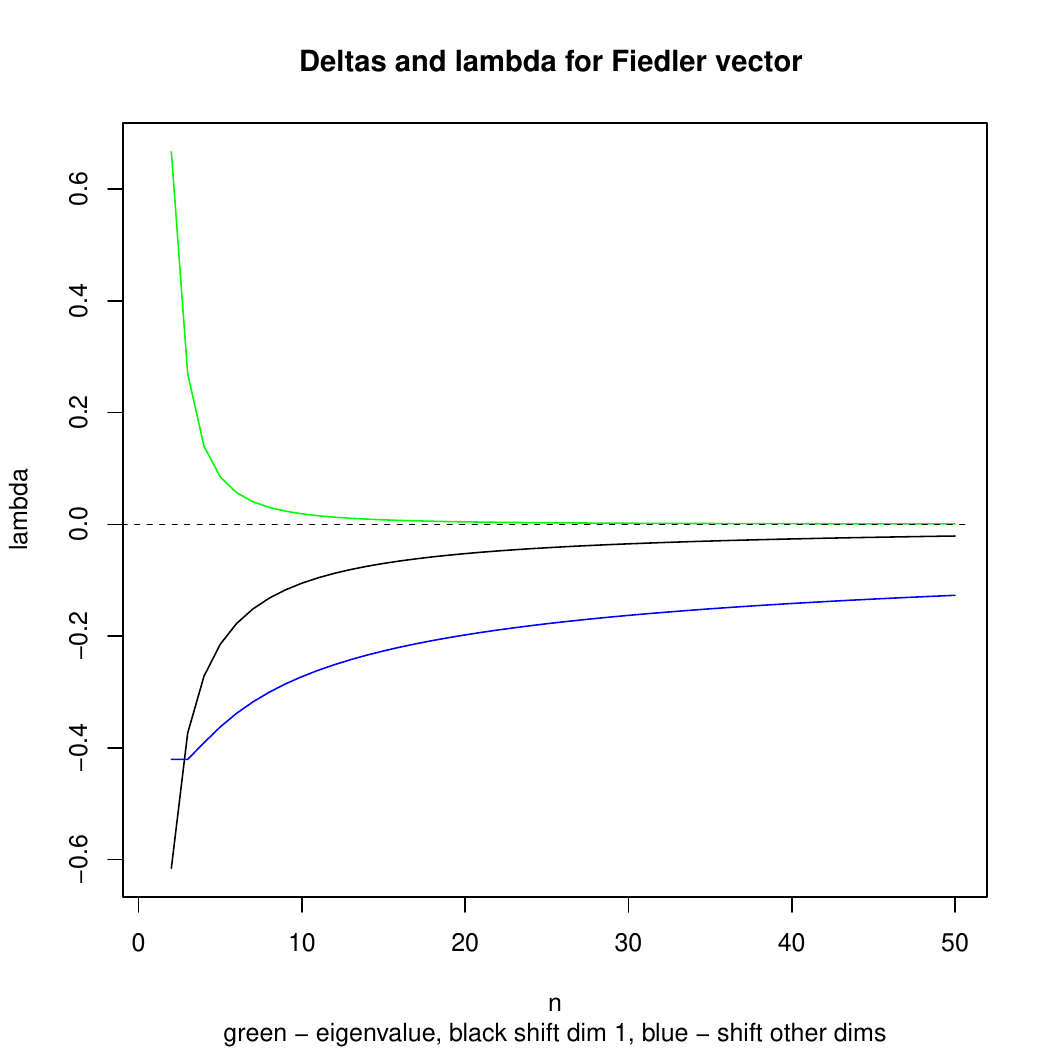}}  %
 (b)\includegraphics[width=0.45\textwidth]{\figaddr{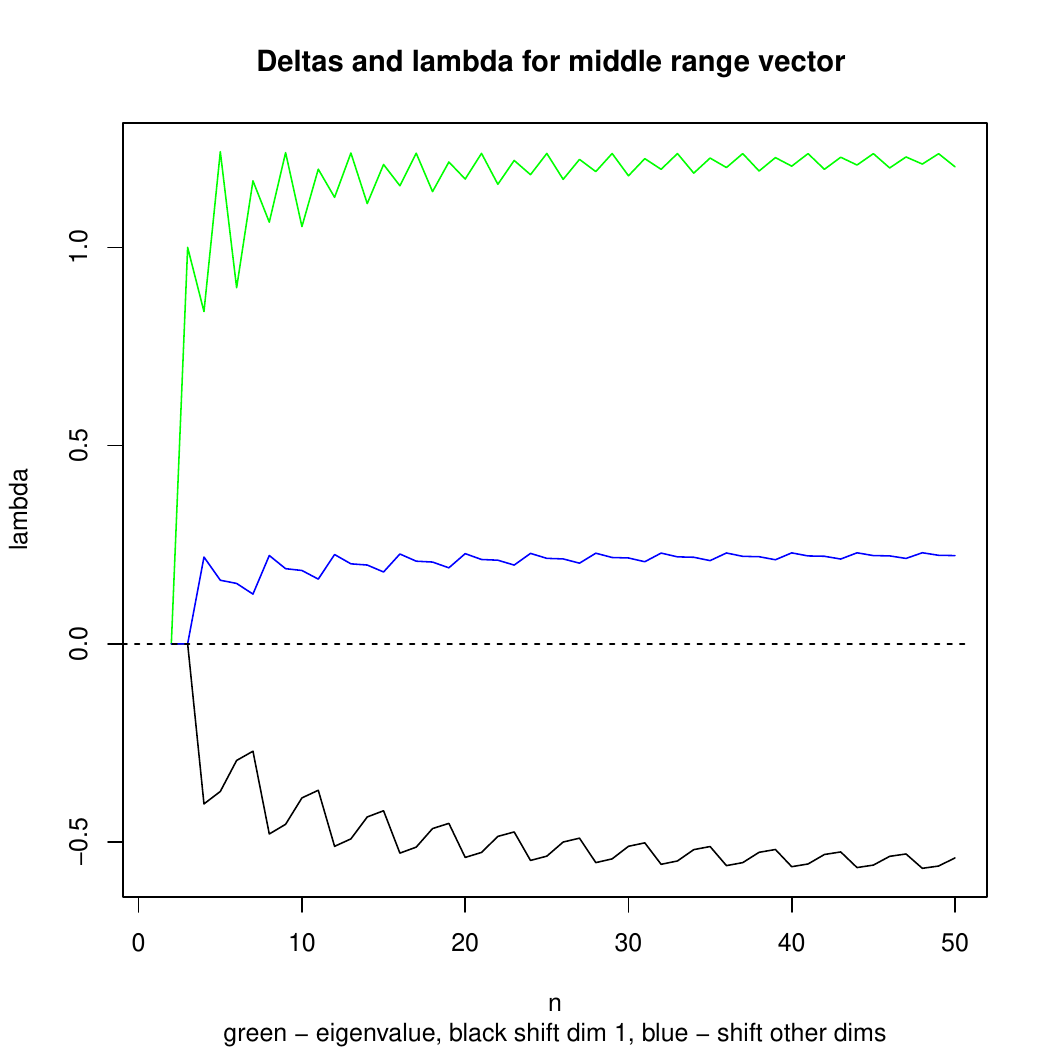}}  %
\\ (c)\includegraphics[width=0.45\textwidth]{\figaddr{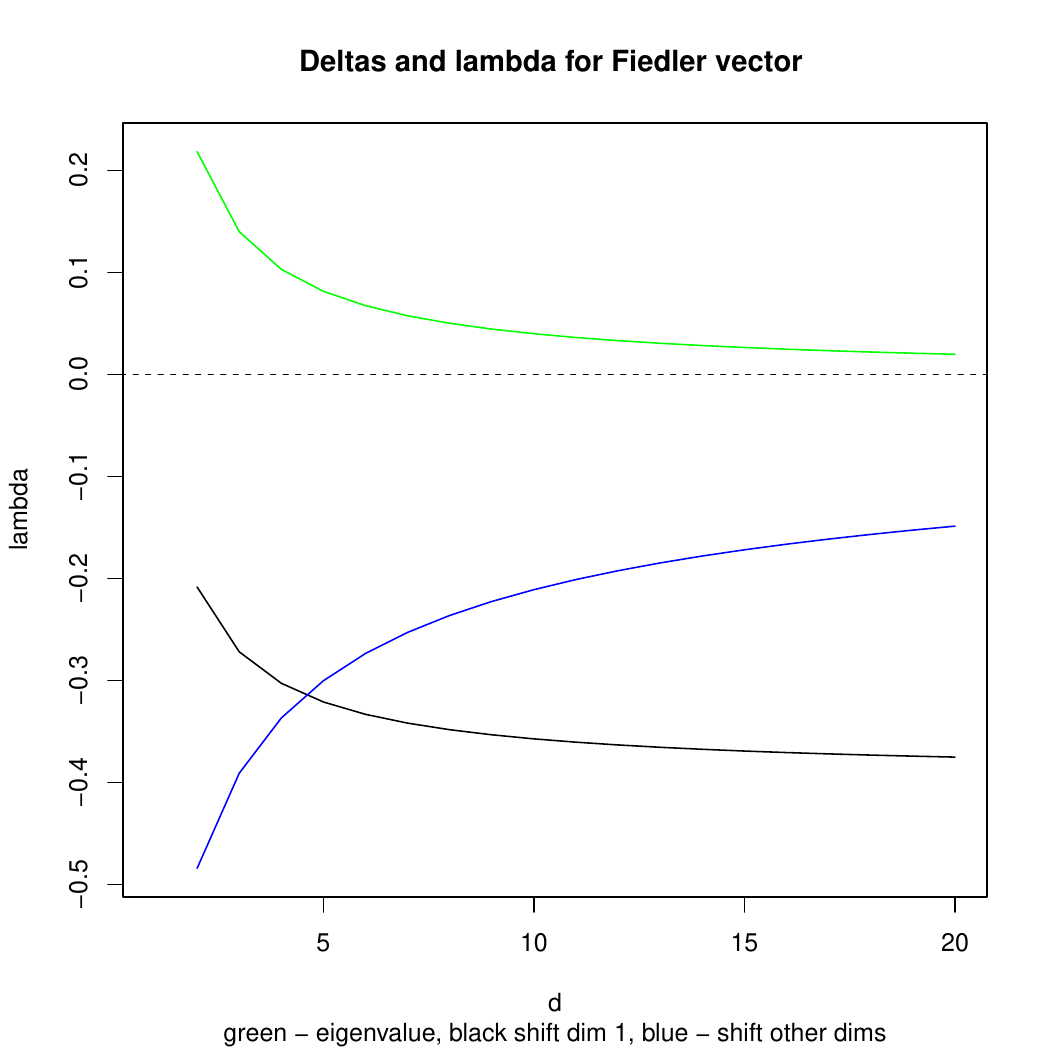}}  %
 (d)\includegraphics[width=0.45\textwidth]{\figaddr{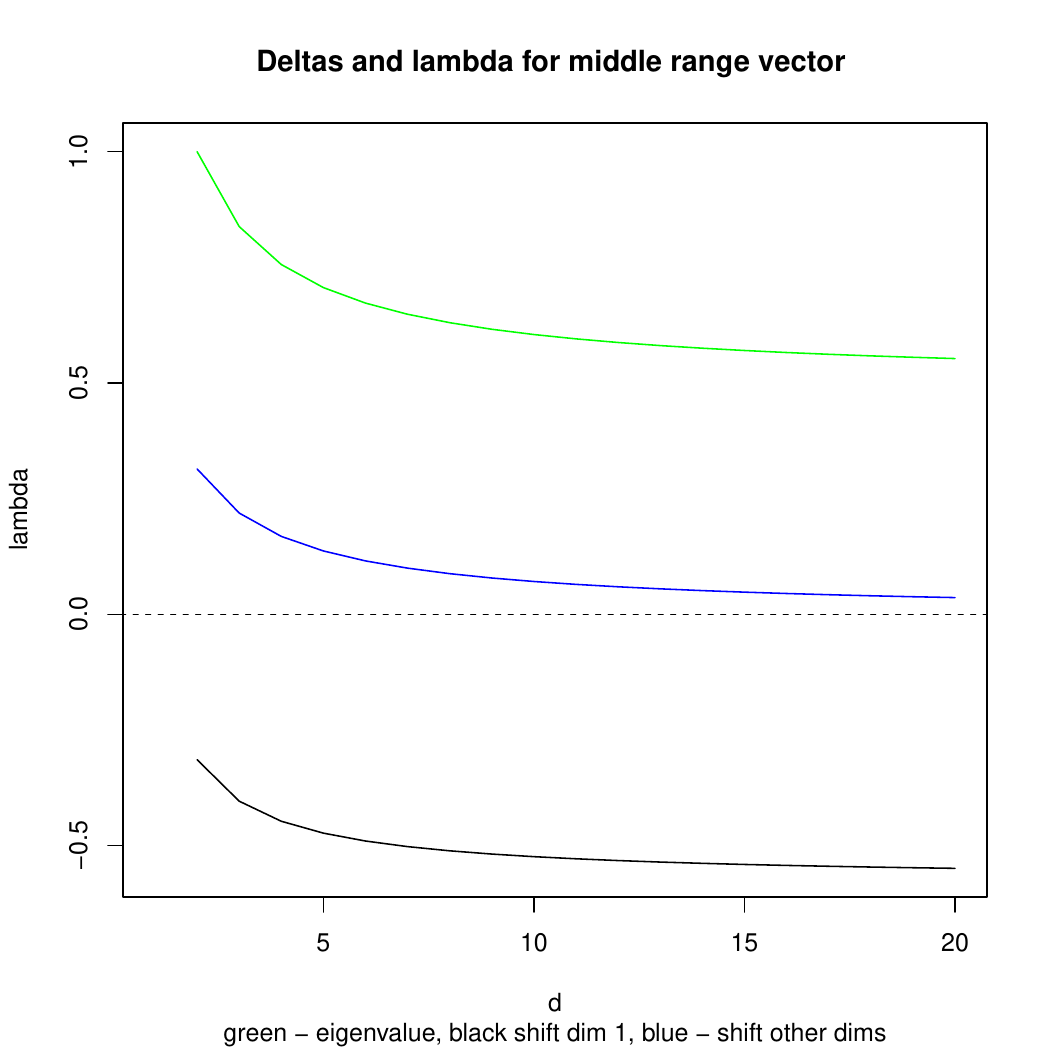}}  %
\caption{The plots of influence of the number layers 
and the  dimensionality  on shifts of   normalized Laplacian   of grid graphs  
(a) the Fiedler vector for a 
 3-dimensional grid graph with $n$ being the number of layers in each dimension, 
(b) a  "middle range" vector for a 
 3-dimensional grid graph with $n$ being the number of layers in each dimension, 
(a) the Fiedler vector for 4 layers in each dimension  
in a  grid graph with $d$ being the number of  dimensions, 
(b) a  "middle range" vector for  4 layers in each dimension   
in a  grid graph with $d$ being the number of  dimensions. 
Green line shows the eigenvalue. 
Black line - the shift in the first dimension. 
Blue line - the shifts in all other dimensions. 
Same number of layers in each direction assumed. 
All $z$s are identical for all dimensions except for the first one. 
For Fiedler vector $z_1=n-2$, and all other  $z_j=n-1$.
For middle vector $z_1=n/4$, and all other  $z_j=n/2$.
}\label{fig:Fiedler_n_d}
\end{figure}
}

Let us consider also the "in the limit" behavior of the 
normalized Laplacian eigenvectors, as described by the expression
\eref{eq:eigenvectorcomponentN}. 
For sufficiently large $\z$ and $\x$ 
$\nu_{\z,\x}= 
D^{1/2}_{\x,\x}
\prod_{j=1}^d (-1)^{x_j} 
\cos\left(\frac{x_j-1}{n_j-1}\left(z_j \pi -2\delta^\mathbf{z}_j\right)+\delta^\mathbf{z}_j\right)
\approx
D^{1/2}_{\x,\x}
\prod_{j=1}^d (-1)^{x_j} 
\cos\left(\frac{z_j \pi}{n_j}\left(x_j-1\right)+\delta^\mathbf{z}_j\right)
$
which differs nevertheless from combinatorial Laplacian eigenvector components 
\eref{eq:eigenvectorcomponentC}
$\nu_{\z,\x}= 
\prod_{j=1}^d  \cos\left(\frac{\pi z_j}{n_j} \left(x_j-0.5\right)\right) 
$ 
in terms of the shift. 

\figVer{
Figure \ref{fig:Fiedler_n_d} can be helpful in understanding the significance of the shift.
Depending on the eigen identity number, the shifts may or may not converge to zero (they may converge to some other value). 
}

\section{Random Walk Laplacians of  Unweighted Grid Graphs}\label{sec:RWLtheorems}
As already mentioned, the eigenvalues and eigenvectors for Random Walk Laplacians can be easily derived from those for Normalized Laplacians (see section \ref{sec:notation}. 
More formally:

\begin{theorem} \label{th:randomwalkLap}
For a regular $d$-dimensional grid with at least one inner node, its random walk  Laplacian  $\mathbb{L}$
has  the 
  eigenvalues  of the   form 
\begin{equation}\label{eq:lambdaRW}
\lambda_{\mathbf{z}}=1+\frac{1}{d} \sum_{j=1}^d 
\cos\left(\frac{1}{n_j-1}\left(z_j \pi -2\delta_j\right) \right)
\end{equation}
with the ${\boldsymbol\delta}^\mathbf{z}$ vector defined 
as a solution of the equation system consisting of 
  the preceding equation \eref{eq:lambdaN2d} and the 
 equations \eref{eq:lambdaN2dm1} for each $l=1,\dots,d$. 
The corresponding eigenvectors $\mathbf{v}_{\mathbf{z}}$ have components of the form  
\begin{equation}\label{eq:eigenvectorcomponentRW}
\nu_{\mathbf{z},[x_1,\dots,x_d]}= 
D_{[x_1,\dots,x_d],[x_1,\dots,x_d]}
\prod_{j=1}^d (-1)^{x_j} 
\cos\left(\frac{x_j-1}{n_j-1}\left(z_j \pi -2\delta^\mathbf{z}_j\right)+\delta^\mathbf{z}_j\right)
\end{equation} 
\end{theorem}
 




\section{Combinatorial Laplacians of Weighted Grid Graphs}\label{secW:CLtheorems}

In this section we extend the closed-form solution for the eigen-problem of combinatorial Laplacian of a $d$-dimensional unweighted grid graph to a weighted one.
The eigenvalues are of the form described by formula \eref{eqW:lambdaC} and the corresponding eigenvecvtors have the form \eref{eqW:coLvector}, 
 
Note that the form of eigenvectors is identical as in case of unweighted graphs \eref{eq:coLvector}, while the eigenvalues differ and are susceptible to scale (they increase when the weights of edges are propertionally increased). 

We will proceed as follows: 
Given the  Theorem \ref{th:reductiontoeigenvectors}, we know that the components of the eigen identity vector can be reduced to the range of $[0,n_j-1]$, because outside of this range the vectors described by formula  \eref{eqW:coLvector} are identical up to the sign to the vectors within this range so that they cannot constitute valid alternative eigenvectors. 
With Theorem \ref{thW:eigenpairs}, we demonstrate that indeed the numbers described by formula \eref{eqW:lambdaC} are eigenvalues and the vectors of the form \eref{eqW:coLvector} are the corresponding eigenvectors.The proof of this theorem is based on the idea of weighted grid graph adjacency matrix decomposition into (additive) parts related to individual directions and the auxiliary Theorem \ref{th:componenteigenpairs} is reused to prove eigenvalue and eigenvectior properties for these parts. 

As the eigenvectors are the same as for the unweighted graphs, 
the Theorems  \ref{th:zerosum} and  Theorem \ref{th:orthogonality} imply the orthogonality of the weighted eigenvectors  which implies that we identified all eigenvectors.  
The respective proofs presented in this section are in fact not so straight forward extension of already presented results for unweighted graphs in Section \ref{secW:CLtheorems}. 
But note that there exists a qualitative differernce in their applicability. 
The unweighted grid graphs represent graphs without an explicit structure, 
while the weighted ones can be tuned to express a graph with some regular structure , where the sharpness of the structure may be regulated by the applied weights. 
Hence the subtle differernces are worth to pay attention. 
 
Let us define  
\begin{equation}\label{eqW:lambdaC}
\lambda_{[z_1,\dots,z_d]}=\sum_{j=1}^d 
2 \waga_j\cdot \left(1-  \cos\left(\frac{\pi  z_j}{n_j}\right)\right) 
\end{equation}
\noindent
where for each $j=1,\dots,d$ $z_j$ is an integer such that $0\le z_j\le n_j-1$.

Define 
 $\lambda_{(j,z_j)}=  
2\waga_j\cdot \left(1-  \cos\left(\frac{\pi  z_j}{n_j}\right)\right) 
$. Then 
$\lambda_{[z_1,\dots,z_d]}=\sum_{j=1}^d 
\lambda_{(j,z_j )}
$. 
Define furthermore 
\begin{equation}\label{eqW:eigenvectorcomponentC}
\nu_{[z_1,\dots,z_d],[x_1,\dots,x_d]}= 
\prod_{j=1}^d  \cos\left(\frac{\pi z_j}{n_j} \left(x_j-0.5\right)\right) 
\end{equation}
\noindent
where for each $j=1,\dots,d$ $x_j$ is an integer such that $1\le x_j\le n_j$.

And finally define the $n$ dimensional vector 
 $\mathbf{v}_{[z_1,\dots,z_d]}$
such that 
 \begin{equation}\label{eqW:coLvector}
 \mathbf{v}_{[z_1,\dots,z_d],i}=\nu_{[z_1,\dots,z_d],[x_1,\dots,x_d] }
\end{equation}


Consider the following claim: 
\begin{theorem}\label{thW:eigenpairs}
Given the combinatorial Laplacian $L$
of the weighted grid graph $G_{(n_1,\dots,n_{d})(\waga_1,\dots,\waga_{d})}$,
for 
each 
vector of integers 
$[z_1,\dots,z_d]$ such that for each $j=1,\dots,d$ 
$0\le z_j\le n_j-1$, 
the $\lambda_{[z_1,\dots,z_d] }$ is an eigenvalue of $L$
and $\mathbf{v}_{[z_1,\dots,z_d]}$ is a corresponding eigenvector. 
\end{theorem}
\begin{proof}
Due to the nature of the grid graph,  
  the similarity matrix 
$S$ 
may be expressed as the sum of similarity matrices 
 $S=\sum_{j=1}^d S_j$,  
where $S_j$ is a connectivity matrix of a graph in which 
a node with identity vector $[x_1,\dots,x_d]$ is connected  
with the node $[x_1,\dots,x_j-1,x_d]$ if $x_j>1$ 
and with node $[x_1,\dots,x_j+1,x_d]$ if $x_j<n_j$ with the dimension $j$ induced weight $\waga_j$ and there are no other connections in the graph. 
Let us denote with  $L_j$  the Laplacian corresponding to the similarity matrix $S_j$. 
Then clearly
$$L=\sum_{j=1}^d L_j$$
As the  Theorem  \ref{th:componenteigenpairs} is applicable also in weighted case, 
$$L \mathbf{v}_{[z_1,\dots,z_d]}=(\sum_{j=1}^d L_j)\mathbf{v}_{[z_1,\dots,z_d]}
=  \sum_{j=1}^d (L_j\mathbf{v}_{[z_1,\dots,z_d]})
$$ $$
= \sum_{j=1}^d (\lambda_{(j,z_j )} \mathbf{v}_{[z_1,\dots,z_d]})
= \left(\sum_{j=1}^d \lambda_{(j,z_j )}\right) \mathbf{v}_{[z_1,\dots,z_d]} 
=  \lambda_{[z_1,\dots,z_d] } \mathbf{v}_{[z_1,\dots,z_d]} 
$$
This is by the way strikingly similar to the results in the Theorem \ref{th:eigenpairs},  but bear in mind that $\lambda$s are carrying the weight related information.  
\end{proof}

We would need now to  establish that all eigenvectors
  are orthogonal to one another. 
As the eigenvectors are identical in weighted and unweighted grid graphs, 
we need only to remind the respective theorems for unweighted graphs, that is  \ref{th:zerosum} and   \ref{th:orthogonality}. 
As all eigenvectors computed by our formulas are orthogonal, 
and the index vectors exhaust the number of nodes,
then the list of eigenvectors and eigenvalues is complete.

\section{Unoriented Laplacian of a Weighted Grid Graph}\label{secW:UOLgeneralization}

Like in case of unweighted grid graphs, 
 there exists an elegant solution to the eigen-problem of the unoriented Laplacian. 
The unoriented Laplacian is defined as 
 $K=D+S$.

\begin{theorem}
The unoriented Laplacian eigenvalues for a weighted grid graph are of the same form as for the combinatorial Laplacian that is
\begin{equation}\label{eqW:lambdaUO}
\lambda_{[z_1,\dots,z_d] }=\sum_{j=1}^d 
\waga_j\left(2 \sin\left(\frac{\pi  z_j}{2 n_j}\right)\right)^2
\end{equation}

The corresponding eigenvectors  have components of the form 

\begin{equation}\label{eqW:eigenvectorcomponentUO}
\nu_{[z_1,\dots,z_d],[x_1,\dots,x_d]}= 
\prod_{j=1}^d (-1)^{x_j} \cos\left(\frac{\pi z_j}{n_j} \left(x_j-0.5\right)\right) 
\end{equation}
\end{theorem}

Pay attention to the fact that the  weighted unoriented Lplacian eigenvectors differ slightly from respective combinatorial Laplacian, but are the same as those for unweighted unoriented Laplacian.   

The proofs of these properties, like in unweighted case, follow the same pattern as above with slight 
variations:
the sums of elements in these vectors are not equal zero any more in general
(an analogue of Theorem \ref{th:zerosum} is not there). 
However, as we multiply always pairs of values 
associated with the same $[x_1,\dots,x_d]$ vector, 
the factors $(-1)^{x_j}$ cancel out and the proofs of analogous other four theorems are essentially the same - we can proceed as if the eigenvectors were those of combinatorial Laplacians.

\section{Normalized Laplacians of Weighted Grid Graphs}\label{secW:NLtheorems}

Please keep in mind that the normalised Laplacian of a graph is defined as
$$\mathfrak{L}=D^{-1/2}LD^{-1/2}=D^{-1/2}(D-S)D^{-1/2}=I-D^{-1/2} S D^{-1/2}$$

The approach to the eigen-problem of normalised Laplacian would be very similar in spirit  to both combinatorial Laplacian for weighted graphs and normalized Laplacians for unweighted graphs. 
As in case of unweighted graph normalized Laplacians,  the solution is not completely closed-form. An iterative component is needed when identifying an eienvalue. Once the eigenvalue is identified, so-called \emph{shifts} or $\delta$s are also identified and then the eigenvalue and eigenvectors are in closed form with respect to these shifts $\delta$.
The problem of only a partial closed-from is strongly related to the fact that the eigen-problem for the normalised Laplacian cannot be decomposed in a way that could be done for the combinatorial Laplacians.

Normalization causes that the eigenvectors of weighted grid graph normalized Laplacians, contrary to their combinatorial counterparts, depend also on weights, because the respective eigenvalues depend on them. 

Therefore the proofs for the  weighted cases cannot be taken over  from the unweighted cases as it was the case with combinatorial Laplacians. Hence this section will be more lengthy.

This section is essentially devoted to the proof of the Theorem \ref{thW:normalizedLap} on the form of eigenvalues and eigenvectors of a normalised Laplacian of a weighted grid graph. 
The proof will be split into two cases of two types of weighted grid graph.
We shall divide the nodes of the weighted grid into two categories: the inner and the border ones. 
The inner ones are those that have two neighbours in the grid in each dimension. 
The border ones are the remaining ones. 
The two types of weighted grid graphs are ones that have inner nodes, and they are handled in Subsection \ref{subsecW:generalInner}, while the graphs without inner nodes are treated in Subsection \ref{subsecW:generalNoInner}. 

\subsection{Weighted Grid Graphs with inner nodes}\label{subsecW:generalInner}

In this subsection we prove the validity of our suggested forms of eigenvalues and eigenvectors of normalized Laplacians of weighted grid graphs, as formulated in the   Theorem \ref{thW:normalizedLap}.
Note that the normalized Laplacian is insensitive to scaling of edge weights and in case of identical weights in all directions it is the same as for unweighted graph. 
Eigenvectors, unlike those for combinatorial Lasplacian, are not identical with ones of unweighted graph. 

The proof will be divided into subsubsections in order not to get lost in the multitude of formulas. 
The Subsubsection\ref{subsubW:deltas} is devoted to finding a simple equation system allowing to find the values of  shifts $\delta$  occurring in the formulas for eigenvalue and eigenvector based on selected nodes. 
The Subsubsection\ref{subsubW:howcompute} contains practical hints on simple solving of the equation system for $\delta$s (shifts).
The Subsubsection\ref{subsubW:othernodes} is devoted to demonstrating, that once the above equation system is solved, the shifts $\delta$  fit also other nodes, not considered in Subsubsection \ref{subsubW:deltas}.
The Subsubsection\ref{subsubW:noLorthogonality} demonstrates that all the eigenvectors are orthogonal to each other so that it is assured that all the eigenvectors have been found.

As in the previous sections, we shall index the eigenvalues and eigenvectors with the vector $\z=[z_1,\dots,z_d]$ such that 
$0\le z_j<n_j$ for $j=1,\dots,d$.   

Note that if $\mathbf{v}$ is the eigenvector of $\mathfrak{L}$ for some eigenvector $\lambda$, then 
$\lambda \mathbf{v}=D^{-1/2}LD^{-1/2}  \mathbf{v}$,
$\lambda \mathbf{v}=D^{-1/2}LD^{-1/2}  \mathbf{v}$,
$\lambda  (D^{-1/2} \mathbf{v})=D^{-1}L (D^{-1/2}  \mathbf{v})$,
$\lambda  D (D^{-1/2} \mathbf{v})= L (D^{-1/2}  \mathbf{v})$.
Denote $\mathbf{w}=(D^{-1/2}  \mathbf{v}) $. 
Consequently we seek $\lambda D \mathbf{w} = L \mathbf{w}$,  
$\lambda D \mathbf{w} = (D-S) \mathbf{w}$, 
$(1-\lambda) D \mathbf{w} =  S  \mathbf{w}$, 
$((1-\lambda) D-S) \mathbf{w} = 0$.

We will subsequently show that
\begin{theorem} \label{thW:normalizedLap}
The normalized Laplacian  $\mathfrak{L}$
of a $d$-dimensional weighted grid graph with at least one inner node  has  the 
  eigenvalues  of the   form

\begin{equation}\label{eqW:lambdaN}
\lambda_{\mathbf{z}}=1+ \sum_{j=1}^d \frac{\waga_j}{\sum_{j=1}^d \waga_j}
\cos\left(\frac{1}{n_j-1}\left(z_j \pi -2\delta_j\right) \right)
\end{equation}
with the ${\boldsymbol\delta}^\mathbf{z}$ vector, called shift vector, defined 
as a solution of the equation system consisting of 
  the subsequent equation \eref{eqW:lambdaN2d} and the 
 equations \eref{eqW:lambdaN2dm1} for each $l=1,\dots,d$. 
The corresponding eigenvectors $\mathbf{v}_{\mathbf{z}}$ have components of the form  
\begin{equation}\label{eqW:eigenvectorcomponentN}
\nu_{\mathbf{z},[x_1,\dots,x_d]}= 
D^{1/2}_{[x_1,\dots,x_d],[x_1,\dots,x_d]}
\prod_{j=1}^d (-1)^{x_j} 
\cos\left(\frac{x_j-1}{n_j-1}\left(z_j \pi -2\delta^\mathbf{z}_j\right)+\delta^\mathbf{z}_j\right)
\end{equation}
\end{theorem}

\subsubsection{Derivation of defining equations for shifts}\label{subsubW:deltas}
Let us now derive the defining equations for the  $\boldsymbol\delta^\mathbf{z}$ vector.
However, instead of the vector $\mathbf{v}$, consider the vector $\mathbf{w}$ 
with the components
\begin{equation}\label{eqW:omegaN}
\omega_{\mathbf{z},[x_1,\dots,x_d]}= 
D ^{-1/2}_{[x_1,\dots,z_d],[x_1,\dots,x_d]} 
\nu_{\mathbf{z},[x_1,\dots,x_d]}$$
$$=\prod_{j=1}^d (-1)^{x_j} 
\cos\left(\frac{x_j-1}{n_j-1}\left(z_j \pi -2\delta^\mathbf{z}_j\right)+\delta^\mathbf{z}_j\right)
\end{equation}

Consider an inner node $[x_1,\dots,x_d]$. 
In order for the $\mathbf{v}$ to be a valid  eigenvector, the next must hold:
\begin{align*}
\lambda_{\mathbf{z}} &
D _{[x_1,\dots,x_d],[x_1,\dots,x_d] } 
\omega_{\mathbf{z},[x_1,\dots,x_d]}
\\
=
&\sum_{j=1}^{d}
\left( 
	\waga_j\left(\omega_{\mathbf{z},[x_1,\dots,x_d]}
-\omega_{\mathbf{z},[x_1,\dots,x_j-1,\dots,x_d]}\right)
\right.
\\
&\left. +	\waga_j\left(\omega_{\mathbf{z},[x_1,\dots,x_d]}
-\omega_{\mathbf{z},[x_1,\dots,x_j+1,\dots,x_d]}\right)
\right)
\end{align*}
As for any inner node 
$D _{[x_1,\dots,x_d],[x_1,\dots,x_d] } =2\sum_{j=1}^d\waga_j$,
and denoting $\sum_{j=1}^d\waga_j$ with $\waga_\Sigma$ 
  we obtain
\begin{align*}
2\waga_\Sigma& \lambda_{\mathbf{z}}
\prod_{j=1}^d  
\cos\left(\frac{x_j-1}{n_j-1}\left(z_j \pi -2\delta^\mathbf{z}_j\right)+\delta^\mathbf{z}_j\right)
\\
=
&\sum_{j=1}^{d} \waga_j
\left( 
\cos\left(\frac{x_j-1}{n_j-1}\left(z_j \pi -2\delta^\mathbf{z}_j\right)+\delta^\mathbf{z}_j\right)
\prod_{i=1,i\ne j}^d  
\cos\left(\frac{x_i-1}{n_i-1}\left(z_i \pi -2\delta^\mathbf{z}_i\right)+\delta^\mathbf{z}_i\right)
\right. \\&+ \left. 
\cos\left(\frac{x_j-1-1}{n_j-1}\left(z_j \pi -2\delta^\mathbf{z}_j\right)+\delta^\mathbf{z}_j\right)
\prod_{i=1,i\ne j}^d  
\cos\left(\frac{x_i-1}{n_i-1}\left(z_i \pi -2\delta^\mathbf{z}_i\right)+\delta^\mathbf{z}_i\right)
\right. \\&+ \left. 
\cos\left(\frac{x_j-1}{n_j-1}\left(z_j \pi -2\delta^\mathbf{z}_j\right)+\delta^\mathbf{z}_j\right)
\prod_{i=1,i\ne j}^d  
\cos\left(\frac{x_i-1}{n_i-1}\left(z_i \pi -2\delta^\mathbf{z}_i\right)+\delta^\mathbf{z}_i\right)
\right. \\&+ \left. 
\cos\left(\frac{x_j-1+1}{n_j-1}\left(z_j \pi -2\delta^\mathbf{z}_j\right)+\delta^\mathbf{z}_j\right)
\prod_{i=1,i\ne j}^d  
\cos\left(\frac{x_i-1}{n_i-1}\left(z_i \pi -2\delta^\mathbf{z}_i\right)+\delta^\mathbf{z}_i\right)
\right)
\end{align*}

As 
\begin{align}
\cos&\left(\frac{x_j-1-1}{n_j-1}\left(z_j \pi -2\delta^\mathbf{z}_j\right)+\delta^\mathbf{z}_j\right)
\nonumber \\=& \cos\left(\frac{x_j -1}{n_j-1}\left(z_j \pi -2\delta^\mathbf{z}_j\right)+\delta^\mathbf{z}_j\right)
\cos\left(\frac{1}{n_j-1}\left(z_j \pi -2\delta^\mathbf{z}_j\right) \right)
\nonumber \\&+ \sin\left(\frac{x_j -1}{n_j-1}\left(z_j \pi -2\delta^\mathbf{z}_j\right)+\delta^\mathbf{z}_j\right)
\sin\left(\frac{1}{n_j-1}\left(z_j \pi -2\delta^\mathbf{z}_j\right) \right)
\label{eqW:cosminus}
\end{align}
and
\begin{align}
\cos&\left(\frac{x_j-1+1}{n_j-1}\left(z_j \pi -2\delta^\mathbf{z}_j\right)+\delta^\mathbf{z}_j\right)
\nonumber \\=& \cos\left(\frac{x_j -1}{n_j-1}\left(z_j \pi -2\delta^\mathbf{z}_j\right)+\delta^\mathbf{z}_j\right)
\cos\left(\frac{1}{n_j-1}\left(z_j \pi -2\delta^\mathbf{z}_j\right) \right)
\nonumber \\&- \sin\left(\frac{x_j -1}{n_j-1}\left(z_j \pi -2\delta^\mathbf{z}_j\right)+\delta^\mathbf{z}_j\right)
\sin\left(\frac{1}{n_j-1}\left(z_j \pi -2\delta^\mathbf{z}_j\right) \right)
\label{eqW:cosplus}
\end{align}
we obtain
\begin{align*}
2\waga_\Sigma &\lambda_{\mathbf{z}}
\prod_{j=1}^d  
\cos\left(\frac{x_j-1}{n_j-1}\left(z_j \pi -2\delta^\mathbf{z}_j\right)+\delta^\mathbf{z}_j\right)
\\&=
\sum_{j=1}^{d}\waga_j
\left( 
2\cos\left(\frac{x_j-1}{n_j-1}\left(z_j \pi -2\delta^\mathbf{z}_j\right)+\delta^\mathbf{z}_j\right)
\prod_{i=1,i\ne j}^d  
\cos\left(\frac{x_i-1}{n_i-1}\left(z_i \pi -2\delta^\mathbf{z}_i\right)+\delta^\mathbf{z}_i\right)
\right.\\&+\left.
2\cos\left(\frac{x_j-1}{n_j-1}\left(z_j \pi -2\delta^\mathbf{z}_j\right)+\delta^\mathbf{z}_j\right)
\cos\left(\frac{1}{n_j-1}\left(z_j \pi -2\delta^\mathbf{z}_j\right) \right)
\right.\\&\cdot\left.
\prod_{i=1,i\ne j}^d  
\cos\left(\frac{x_i-1}{n_i-1}\left(z_i \pi -2\delta^\mathbf{z}_i\right)+\delta^\mathbf{z}_i\right)
\right)
\end{align*}
which, upon division by $\prod_{j=1}^d  
\cos\left(\frac{x_j-1}{n_j-1}\left(z_j \pi -2\delta^\mathbf{z}_j\right)+\delta^\mathbf{z}_j\right)
$,  reduces to:
\begin{equation}\label{eqW:lambdaN2d}
 2\waga_\Sigma \lambda_{\mathbf{z}}
 =\sum_{j=1}^{d}\waga_j \left(2+2 \cos\left(\frac{1}{n_j-1}\left(z_j \pi -2\delta^\mathbf{z}_j\right) \right)
\right)
\end{equation}
which, after dividing by $2\waga_\Sigma$  
reduces to the formula 
\eref{eqW:lambdaN}. 
Thence for inner nodes the formula \eref{eqW:lambdaN} is a valid description of the eigenvalues, without any assumptions on the 
$\boldsymbol{\delta}^\mathbf{z}$. 

Let us turn to the border nodes. 
The peculiarity of these points within a grid is that they have a lower number of neighbours than the inner nodes that have been just considered. 
They may lack the one or more neighbours compared to the inner nodes, depending how many indices $x_l$ in their node identity vector are equal 1 or the maximum number of layers in a given dimension. 
But the issue is not so much the absense of a neighbour but rather the problem in the eigenequation for that node because 
we do not have available 
constituent expressions of the form 
\eref{eqW:cosminus}
 and \eref{eqW:cosplus}
the sinus components of which delete each other as for inner nodes. 

So we will need to compensate for the presence of the uncompensated sinus component in such eigen-equations via special choice of the shifts. As we will see, no explicit formula can be obtained for the computation of $\delta$s, hence implicit methods will be necessary.

Consider the ones that have one neighbour less than the inner nodes (one neighbour missing), say along the dimension $l$, $x_l=1$.
The subsequent formula must hold:
\begin{align*}
\lambda_{\mathbf{z}}&
(2\waga_\Sigma-\waga_l) 
\omega_{\mathbf{z},[x_1,\dots,x_d]}
\\=&
\waga_l	\left(\omega_{\mathbf{z},[x_1,\dots,x_d]}
-\omega_{\mathbf{z},[x_1,\dots,x_l+1,\dots,x_d]}\right)
\\&+\sum_{j=1,j\ne l}^{d}
\waga_j\left( 
	\left(\omega_{\mathbf{z},[x_1,\dots,x_d]}
-\omega_{\mathbf{z},[x_1,\dots,x_j-1,\dots,x_d]}\right)
\right.\\&+\left.	\left(\omega_{\mathbf{z},[x_1,\dots,x_d]}
-\omega_{\mathbf{z},[x_1,\dots,x_j+1,\dots,x_d]}\right)
\right)
\end{align*}

Let us now expand the terms $\omega$ and $\lambda$ via formulas \eref{eqW:omegaN} and  \eref{eqW:lambdaN}. 
We arrive at the formula: 
\begin{align*}
(2\waga_\Sigma-\waga_l)& \lambda_{\mathbf{z}}
\prod_{j=1}^d  
\cos\left(\frac{x_j-1}{n_j-1}\left(z_j \pi -2\delta^\mathbf{z}_j\right)+\delta^\mathbf{z}_j\right)
\\
=
&
\waga_l \left( 
\cos\left(\frac{x_l-1}{n_l-1}\left(z_l \pi -2\delta^\mathbf{z}_l\right)+\delta^\mathbf{z}_l\right)
\prod_{i=1,i\ne l}^d  
\cos\left(\frac{x_i-1}{n_i-1}\left(z_i \pi -2\delta^\mathbf{z}_i\right)+\delta^\mathbf{z}_i\right)
\right. \\&+ \left. 
\cos\left(\frac{x_l-1+1}{n_l-1}\left(z_l \pi -2\delta^\mathbf{z}_l\right)+\delta^\mathbf{z}_l\right)
\prod_{i=1,i\ne l}^d  
\cos\left(\frac{x_i-1}{n_i-1}\left(z_i \pi -2\delta^\mathbf{z}_i\right)+\delta^\mathbf{z}_i\right)
\right)
 \\&+\sum_{j=1,j\ne l }^{d}
\waga_j \left( 
\cos\left(\frac{x_j-1}{n_j-1}\left(z_j \pi -2\delta^\mathbf{z}_j\right)+\delta^\mathbf{z}_j\right)
\prod_{i=1,i\ne j}^d  
\cos\left(\frac{x_i-1}{n_i-1}\left(z_i \pi -2\delta^\mathbf{z}_i\right)+\delta^\mathbf{z}_i\right)
\right. \\&+ \left. 
\cos\left(\frac{x_j-1-1}{n_j-1}\left(z_j \pi -2\delta^\mathbf{z}_j\right)+\delta^\mathbf{z}_j\right)
\prod_{i=1,i\ne j}^d  
\cos\left(\frac{x_i-1}{n_i-1}\left(z_i \pi -2\delta^\mathbf{z}_i\right)+\delta^\mathbf{z}_i\right)
\right. \\&+ \left. 
\cos\left(\frac{x_j-1}{n_j-1}\left(z_j \pi -2\delta^\mathbf{z}_j\right)+\delta^\mathbf{z}_j\right)
\prod_{i=1,i\ne j}^d  
\cos\left(\frac{x_i-1}{n_i-1}\left(z_i \pi -2\delta^\mathbf{z}_i\right)+\delta^\mathbf{z}_i\right)
\right. \\&+ \left. 
\cos\left(\frac{x_j-1+1}{n_j-1}\left(z_j \pi -2\delta^\mathbf{z}_j\right)+\delta^\mathbf{z}_j\right)
\prod_{i=1,i\ne j}^d  
\cos\left(\frac{x_i-1}{n_i-1}\left(z_i \pi -2\delta^\mathbf{z}_i\right)+\delta^\mathbf{z}_i\right)
\right)
\end{align*}

Hence
\begin{align*}
(2\waga_\Sigma-\waga_l)& \lambda_{\mathbf{z}}
\prod_{j=1}^d  
\cos\left(\frac{x_j-1}{n_j-1}\left(z_j \pi -2\delta^\mathbf{z}_j\right)+\delta^\mathbf{z}_j\right)
\\
=
&
\waga_l\left( 
\cos\left(
         \frac{x_l-1}{n_l-1}
          \left(z_l \pi -2\delta^\mathbf{z}_l\right)
   +\delta^\mathbf{z}_l
    \right) 
\right.\\ &\left.+ \cos\left(\frac{x_l -1}{n_l-1}\left(z_l \pi -2\delta^\mathbf{z}_l\right)+\delta^\mathbf{z}_l\right)
\cos\left(\frac{1}{n_l-1}\left(z_l \pi -2\delta^\mathbf{z}_l\right) \right)
\right.\\ &\left.- \sin\left(\frac{x_l -1}{n_l-1}\left(z_l \pi -2\delta^\mathbf{z}_l\right)+\delta^\mathbf{z}_j\right)
\sin\left(\frac{1}{n_l-1}\left(z_l \pi -2\delta^\mathbf{z}_l\right) \right)
\right)
\\& \cdot 
\prod_{i=1,i\ne l}^d  
\cos\left(\frac{x_i-1}{n_i-1}\left(z_i \pi -2\delta^\mathbf{z}_i\right)+\delta^\mathbf{z}_i\right)
 \\&+\sum_{j=1,j\ne l }^{d}
\waga_j\left( 
2\cos\left(\frac{x_j-1}{n_j-1}\left(z_j \pi -2\delta^\mathbf{z}_j\right)+\delta^\mathbf{z}_j\right)
\prod_{i=1,i\ne j}^d  
\cos\left(\frac{x_i-1}{n_i-1}\left(z_i \pi -2\delta^\mathbf{z}_i\right)+\delta^\mathbf{z}_i\right)
\right.\\&+\left.
2\cos\left(\frac{x_j-1}{n_j-1}\left(z_j \pi -2\delta^\mathbf{z}_j\right)+\delta^\mathbf{z}_j\right)
\cos\left(\frac{1}{n_j-1}\left(z_j \pi -2\delta^\mathbf{z}_j\right) \right)
\right.\\&\cdot\left.
\prod_{i=1,i\ne j}^d  
\cos\left(\frac{x_i-1}{n_i-1}\left(z_i \pi -2\delta^\mathbf{z}_i\right)+\delta^\mathbf{z}_i\right)
\right)
\end{align*}

 By dividing as previously we get 
\begin{align*}
(2\waga_\Sigma-\waga_l)& \lambda_{\mathbf{z}}
\\
=
&
\waga_l\left( 
1 
\right.\\ &\left.+ 
\cos\left(\frac{1}{n_l-1}\left(z_l \pi -2\delta^\mathbf{z}_l\right) \right)
\right.\\ &\left.- 
\tan\left(\frac{x_l -1}{n_l-1}\left(z_l \pi -2\delta^\mathbf{z}_l\right)+\delta^\mathbf{z}_j\right)
\sin\left(\frac{1}{n_l-1}\left(z_l \pi -2\delta^\mathbf{z}_l\right) \right)
\right)
 \\&+\sum_{j=1,j\ne l }^{d}
\waga_j\left( 
2 
\right.\\&+\left.
2 
\cos\left(\frac{1}{n_j-1}\left(z_j \pi -2\delta^\mathbf{z}_j\right) \right)
\right)
\end{align*}

Considerations  above lead to the conclusion (as $x_l -1=0$)  
\begin{align*}
(2\waga_\Sigma-\waga_l) \lambda_{\mathbf{z}}
 =&
\waga_l(1+\cos\left(\frac{1}{n_l-1}\left(z_l \pi -2\delta^\mathbf{z}_l\right)\right)
-\tan(\delta^\mathbf{z}_l)\sin\left(\frac{1}{n_l-1}\left(z_l \pi -2\delta^\mathbf{z}_l\right)\right) )
\\&+
\sum_{j=1,j\ne l}^{d} \waga_j\left(2+2 \cos\left(\frac{1}{n_j-1}\left(z_j \pi -2\delta^\mathbf{z}_j\right) \right)
\right)
\end{align*}

A subtraction of the preceding formula from the expression  \eref{eqW:lambdaN2d}   leads to
\begin{equation}\label{eqW:lambdaN2dm1000}
\waga_l \lambda_{\mathbf{z}}
 =
\waga_l(1+\cos\left(\frac{1}{n_l-1}\left(z_l \pi -2\delta^\mathbf{z}_l\right)\right)
+\tan(\delta^\mathbf{z}_l)\sin\left(\frac{1}{n_l-1}\left(z_l \pi -2\delta^\mathbf{z}_l\right)\right))
\end{equation}

\begin{equation}\label{eqW:lambdaN2dm1}
 \lambda_{\mathbf{z}}
 =
 1+\cos\left(\frac{1}{n_l-1}\left(z_l \pi -2\delta^\mathbf{z}_l\right)\right)
+\tan(\delta^\mathbf{z}_l)\sin\left(\frac{1}{n_l-1}\left(z_l \pi -2\delta^\mathbf{z}_l\right)\right)
\end{equation}

Interestingly, the last equation \eref{eqW:lambdaN2dm1} does not depend explicitly on weights. 
So it is formally identical with the very same equation for unweighted graphs. 
One shall keep in mind, however, that 
$ \lambda_{\mathbf{z}}$ depends on the weights and therefore the impact of weighting is present also in this equation.

By combining the equation \eref{eqW:lambdaN2d} 
with equations \eref{eqW:lambdaN2dm1} for each $l$ 
we get an equation system of $d+1$ equations  from which 
$\lambda$ and $\delta$s can be determined.

\subsubsection{Computing eigenvalues and shifts}\label{subsubW:howcompute}
The equation \eref{eqW:lambdaN2dm1} may be transformed to:
$$ (\lambda_{\mathbf{z}}-1) \cos(\delta^\mathbf{z}_l)
 =
 +\cos(\delta^\mathbf{z}_l)\cos\left(\frac{1}{n_l-1}\left(z_l \pi -2\delta^\mathbf{z}_l\right)\right)
+\sin(\delta^\mathbf{z}_l)\sin\left(\frac{1}{n_l-1}\left(z_l \pi -2\delta^\mathbf{z}_l\right)\right)
$$
 
\begin{equation}\label{eqW:deltafromlambda}
 (\lambda_{\mathbf{z}}-1) \cos(\delta^\mathbf{z}_l)
 =
 \cos(\delta^\mathbf{z}_l-\frac{1}{n_l-1}\left(z_l \pi -2\delta^\mathbf{z}_l\right) 
\end{equation}
which is simpler to solve for $\delta$ knowing $\lambda$. 
The solution can be obtained using the bisectional method 
on $\lambda$ using the above formula to obtain $\delta$s, and the using \eref{eqW:lambdaN} to get the value of $\lambda'$ and then 
reducing bisectionally the difference between $\lambda$ and $\lambda'$ down to zero. 

\Niepotrzebne{
}

\subsubsection{Validity of the derived eigenvalues and shifts for other nodes}\label{subsubW:othernodes}

The method of computing $\delta$s and  $\lambda$s from section \ref{subsubW:howcompute} is based only on fitting the equations for inner nodes and for selected border nodes.  One has to demonstrate, however, that   the solution would fit all the other border nodes. 

Consider first 
 the ones that have one neighbour less than the 
inner nodes (one neighbour missing), 
say along the dimension $l$ where $x_l=n_l$.
The subsequent relationship must hold:
\begin{align*}
\lambda_{\mathbf{z}}&
(2\waga_\Sigma-\waga_l) 
\omega_{\mathbf{z},[x_1,\dots,x_d]}
\\=&
\waga_l	\left(\omega_{\mathbf{z},[x_1,\dots,x_d]}
-\omega_{\mathbf{z},[x_1,\dots,x_l-1,\dots,x_d]}\right)
\\&
+\sum_{j=1,j\ne l}^{d}
\waga_j \left( 
	\left(\omega_{\mathbf{z},[x_1,\dots,x_d]}
-\omega_{\mathbf{z},[x_1,\dots,x_j-1,\dots,x_d]}\right)
\right.\\&+\left.	\left(\omega_{\mathbf{z},[x_1,\dots,x_d]}
-\omega_{\mathbf{z},[x_1,\dots,x_j+1,\dots,x_d]}\right)
\right)
\end{align*}
Considerations analogous to the above lead to the conclusion 
\begin{align*}
(2\waga_\Sigma-\waga_l)& \lambda_{\mathbf{z}}
\\ = &
\waga_l(1+\cos\left(\frac{1}{n_l-1}\left(z_l \pi -2\delta^\mathbf{z}_l\right)\right)
\\&+\tan(
\frac{n_l-1}{n_l-1}\left(z_l \pi -2\delta^\mathbf{z}_l\right)+
\delta^\mathbf{z}_l)\sin\left(\frac{1}{n_l-1}\left(z_l \pi -2\delta^\mathbf{z}_l\right)\right) )
\\&+
\sum_{j=1,j\ne l}^{d} \waga_j \left(2+2 \cos\left(\frac{1}{n_j-1}\left(z_j \pi -2\delta^\mathbf{z}_j\right) \right)
\right)
\end{align*}
A subtraction of the preceding formula from the expression  \eref{eqW:lambdaN2d}  
and division by $\waga_l$ leads to

\begin{equation} 
 \lambda_{\mathbf{z}}
 =
1-\cos\left(\frac{1}{n_l-1}\left(z_l \pi -2\delta^\mathbf{z}_l\right)\right)
-\tan(-\delta^\mathbf{z}_l)\sin\left(\frac{1}{n_l-1}\left(z_l \pi -2\delta^\mathbf{z}_l\right)\right)
\end{equation}
which is the same as equation
\eref{eqW:lambdaN2dm1}. 

Now look at other border nodes. 
Consider the ones that have one neighbour less than the inner nodes (one neighbour missing), 
along multiple dimensions, 
say along the dimensions $l^+_{1},l^+_{2},\dots,l^+_{m^+}$, $x_{l^+_k}=1$ 
and along the dimensions $l^-_{1},l^-_{2},\dots,l^-_{m^-}$, 
  $x_{l^-_k}==n_{l^-_k}$, with $1<m^++m^-\le d$.
The equation below has to  hold:
\begin{align*}
\lambda_{\mathbf{z}}& 
(2\sum_{j=1}^d\waga_j - \sum_{j=1}^{m^+}\waga_j- \sum_{j=1}^{m^-}\waga_j) 
\omega_{\mathbf{z},[x_1,\dots,x_d]}
\\=&
\sum_{k=1}^{m^+}
\waga_k	\left(\omega_{\mathbf{z},[x_1,\dots,x_d]}
-\omega_{\mathbf{z},[x_1,\dots,x_{l^+_k} + 1,\dots,x_d]}\right)
\\& +
 \sum_{k=1}^{m^-}
\waga_k		\left(\omega_{\mathbf{z},[x_1,\dots,x_d]}
-\omega_{\mathbf{z},[x_1,\dots,x_{l^-_k} - 1,\dots,x_d]}\right)
\\& +
\sum_{j=1,j\not\in \{ l^+_1,\dots,l^+_{m^+},l^-_1,\dots,l^-_{m^-}\}}^{d}
\waga_j	\left( 
	\left(\omega_{\mathbf{z},[x_1,\dots,x_d]}
-\omega_{\mathbf{z},[x_1,\dots,x_j-1,\dots,x_d]}\right)
\right. \\& + \left.
	\left(\omega_{\mathbf{z},[x_1,\dots,x_d]}
-\omega_{\mathbf{z},[x_1,\dots,x_j+1,\dots,x_d]}\right)
\right)
\end{align*}
This will lead to 
(after subtraction from the expression  \eref{eqW:lambdaN2d})  
\begin{align*} 
(\sum_{j=1}^{m^+}\waga_j+ \sum_{j=1}^{m^-}\waga_j)& \lambda_{\mathbf{z}}
\\ = &
\sum_{k=1}^{m^+} \waga_k \left(1-\cos\left(\frac{1}{n_{l^+_k}-1}\left(z_{l^+_k} \pi -2\delta^\mathbf{z}_{l^+_k}\right)\right)
\right. \\  & + \left.
 \tan(\delta^\mathbf{z}_{l^+_k})\sin\left(\frac{1}{n_{l^+_k}-1}\left(z_{l^+_k} \pi -2\delta^\mathbf{z}_{l^+_k}\right)\right)
\right)
\\  & + 
\sum_{k=1}^{m^-} \waga_k \left(1-\cos\left(\frac{1}{n_{l^-_k}-1}\left(z_{l^-_k} \pi -2\delta^\mathbf{z}_{l^-_k}\right)\right)
\right. \\  & + \left.
\tan(\delta^\mathbf{z}_{l^-_k})\sin\left(\frac{1}{n_{l^-_k}-1}\left(z_{l^-_k} \pi -2\delta^\mathbf{z}_{l^-_k}\right)\right)
\right)
\end{align*}
This equation results from adding equations \eref{eqW:lambdaN2dm1}
multiplied by an appropiate weight $\waga_l$ 
 for respective dimensions
$l  \in \{ l^+_1,\dots,l^+_{m^+},l^-_1,\dots,l^-_{m^-}\}$. 
Thus, once the equation system was solved for $\boldsymbol{\delta}^\mathbf{z}$ for the set of nodes mentioned in section \ref{subsubW:howcompute}, 
all the other nodes fit. 
Thus the validity of the formula for eigenvalues $\lambda$  was proven along with the correctness of eigenvector formulas. 

\subsubsection{The validity of the eigenvectors for identical eigenvalues}\label{subsubW:noLorthogonality}

In order to be sure that all eigenvectors have been correctly identified, 
it is sufficient to show that the eigenvectors proposed are 
orthogonal for different $\mathbf{z}$. 
Recall that eigenvectors are orthogonal whenever the eigenvalues are different. 
However, there are cases when they may be identical. 
One possibility is implied by symmetries between various dimensions.
The symmetry means that in both dimensions we have e.g. the same number of layers and in both dimensions the weights are identical. 
Other cases of identical eigenvalues may be due to "unhappy" proportions between the weights of edges. 

Let us demonstrate that whenever $ \lambda_{\mathbf{z'}}= \lambda_{\mathbf{z"}}$ 
and $\mathbf{z'} \ne  \mathbf{z"}$, then 
$\mathbf{v}^\mathbf{z'}$ and  $\mathbf{v}^  \mathbf{z"}$ are orthogonal, that is their dot product is equal zero. This means: 
$$0 = \sum_\mathbf{x} D_{[\x],[\x]}
\prod_{j=1}^d   
\cos\left(\frac{x_j-1}{n_j-1}\left(z'_j \pi -2\delta^\mathbf{z'}_j\right)+\delta^\mathbf{z'}_j\right)
\cdot
\cos\left(\frac{x_j-1}{n_j-1}\left(z"_j \pi -2\delta^\mathbf{z"}_j\right)+\delta^\mathbf{z"}_j\right)
$$ 
or expressed differently:
\begin{align}
  0=&  
  \sum_{x_2=1}^{n_2}\dots   \sum_{x_d=1}^{n_d} 
  \prod_{j=2}^d   
\cos\left(\frac{x_j-1}{n_j-1}\left(z'_j \pi -2\delta^\mathbf{z'}_j\right)+\delta^\mathbf{z'}_j\right)
\cdot
\cos\left(\frac{x_j-1}{n_j-1}\left(z"_j \pi -2\delta^\mathbf{z"}_j\right)+\delta^\mathbf{z"}_j\right)
\nonumber \\ &
 \sum_{x_1=1}^{n_1}
 D_{[\x],[\x]}
 \cos\left(\frac{x_1-1}{n_1-1}\left(z'_1 \pi -2\delta^\mathbf{z'}_1\right)+\delta^\mathbf{z'}_1\right)
\cdot
\cos\left(\frac{x_1-1}{n_1-1}\left(z"_1 \pi -2\delta^\mathbf{z"}_1\right)+\delta^\mathbf{z"}_1\right) \label{eqW:mainN}
\end{align}

Let $d^{[M]}_{x_2,\dots,x_d}=\max_{x_1 \in \{1,\dots,n_1\}} D_{[x_1,x_2,\dots,x_d],[x_1,x_2,\dots,x_d]}$.
(In the special case of $n_1=2$ $d^{[M]}_{x_2,\dots,x_d}$ should be increased by $waga_1$.). 

What we need to demonstrate is:
$$ \sum_{x_1=1}^{n_1}
 D_{[\x],[\x]}
 \cos\left(\frac{x_1-1}{n_1-1}\left(z'_1 \pi -2\delta^\mathbf{z'}_1\right)+\delta^\mathbf{z'}_1\right)
\cdot
\cos\left(\frac{x_1-1}{n_1-1}\left(z"_1 \pi -2\delta^\mathbf{z"}_1\right)+\delta^\mathbf{z"}_1\right)
$$
$$=
 -2\waga_1 \cos(\delta^\mathbf{z'}_1)\cos(\delta^\mathbf{z"}_1)
((z'_1+z"_1+1) \mod 2)$$ $$
 +d^{[M]}_{x_2,\dots,\x_d}
\sum_{x_1=1}^{n_1}
 \cos\left(\frac{x_1-1}{n_1-1}\left(z'_1 \pi -2\delta^\mathbf{z'}_1\right)+\delta^\mathbf{z'}_1\right)
\cdot
\cos\left(\frac{x_1-1}{n_1-1}\left(z"_1 \pi -2\delta^\mathbf{z"}_1\right)+\delta^\mathbf{z"}_1\right)
$$
  \begin{equation}\label{eqW:res1}=
 -2\waga_1\cos(\delta^\mathbf{z'}_1)\cos(\delta^\mathbf{z"}_1) ((z'_1+z"_1+1) \mod 2)
\end{equation}

To prove it, we will proceeed in two consecutive steps.
Validity of the first equal sign will be shown first. 
Subsequently,  we will prove that   
$\sum_{x_j=1}^{n_j}
 \cos\left(\frac{x_j-1}{n_j-1}\left(z'_j \pi -2\delta^\mathbf{z'}_j\right)+\delta^\mathbf{z'}_j\right)
\cdot
\cos\left(\frac{x_j-1}{n_j-1}\left(z"_j \pi -2\delta^\mathbf{z"}_j\right)+\delta^\mathbf{z"}_j\right)=0$

Note that if we 
fix the values $x_2,\dots,x_d$, we select a path in the graph 
of nodes with identities 
$[1,x_2,\dots,x_d],\dots [n_1,x_2,\dots,x_d]$.
On this path,  
  each node is connected
to the same number of  outside nodes with same edge weights. 
If we  confined the graph to this path only, then 
 the first and the last nodes would 
have "degree" $\waga_1$  and the other have the degree $2\waga_1$.
In brief, on this path, 
the endpoints are of degree 
$d^{[M]}_{x_2,\dots,x_d}-\waga_1$, while the remaining nodes of the path are 
of  degree  
 $d^{[M]}_{x_2,\dots,x_d}$. 

Therefore
$ \sum_{x_1=1}^{n_1}
 D_{[\x],[\x]}
 \cos\left(\frac{x_1-1}{n_1-1}\left(z'_1 \pi -2\delta^\mathbf{z'}_1\right)+\delta^\mathbf{z'}_1\right)
\cdot
\cos\left(\frac{x_1-1}{n_1-1}\left(z"_1 \pi -2\delta^\mathbf{z"}_1\right)+\delta^\mathbf{z"}_1\right)
$
$=(d^{[M]}_{x_2,\dots,x_d}-\waga_1)
 \cos\left(\frac{ 1-1}{n_1-1}\left(z'_1 \pi -2\delta^\mathbf{z'}_1\right)+\delta^\mathbf{z'}_1\right)
\cdot
\cos\left(\frac{ 1-1}{n_1-1}\left(z"_1 \pi -2\delta^\mathbf{z"}_1\right)+\delta^\mathbf{z"}_1\right)
+(d^{[M]}_{x_2,\dots,x_d}-\waga_1)
 \cos\left(\frac{n_1-1}{n_1-1}\left(z'_1 \pi -2\delta^\mathbf{z'}_1\right)+\delta^\mathbf{z'}_1\right)
\cdot
\cos\left(\frac{n_1-1}{n_1-1}\left(z"_1 \pi -2\delta^\mathbf{z"}_1\right)+\delta^\mathbf{z"}_1\right)
+ 
  \sum_{x_1=2}^{n_1-1}
d^{[M]}_{x_2,\dots,x_d}
 \cos\left(\frac{x_1-1}{n_1-1}\left(z'_1 \pi -2\delta^\mathbf{z'}_1\right)+\delta^\mathbf{z'}_1\right)
\cdot
\cos\left(\frac{x_1-1}{n_1-1}\left(z"_1 \pi -2\delta^\mathbf{z"}_1\right)+\delta^\mathbf{z"}_1\right)
$
$= -  \waga_1
 \cos\left(\frac{ 1-1}{n_1-1}\left(z'_1 \pi -2\delta^\mathbf{z'}_1\right)+\delta^\mathbf{z'}_1\right)
\cdot
\cos\left(\frac{ 1-1}{n_1-1}\left(z"_1 \pi -2\delta^\mathbf{z"}_1\right)+\delta^\mathbf{z"}_1\right)
$ \\ $
-\waga_1 
 \cos\left(\frac{n_1-1}{n_1-1}\left(z'_1 \pi -2\delta^\mathbf{z'}_1\right)+\delta^\mathbf{z'}_1\right)
\cdot
\cos\left(\frac{n_1-1}{n_1-1}\left(z"_1 \pi -2\delta^\mathbf{z"}_1\right)+\delta^\mathbf{z"}_1\right)
+ $ \\ $
  \sum_{x_1=1}^{n_1 }
d^{[M]}_{x_2,\dots,x_d}
 \cos\left(\frac{x_1-1}{n_1-1}\left(z'_1 \pi -2\delta^\mathbf{z'}_1\right)+\delta^\mathbf{z'}_1\right)
\cdot
\cos\left(\frac{x_1-1}{n_1-1}\left(z"_1 \pi -2\delta^\mathbf{z"}_1\right)+\delta^\mathbf{z"}_1\right)
$
$= -  2 \waga_1 
 \cos\left( \delta^\mathbf{z'}_1\right)
\cdot
\cos\left( \delta^\mathbf{z"}_1\right) ((z'_1+z"_1+1) \mod 2)
+ 
  \sum_{x_1=1}^{n_1 }
d^{[M]}_{x_2,\dots,x_d}
 \cos\left(\frac{x_1-1}{n_1-1}\left(z'_1 \pi -2\delta^\mathbf{z'}_1\right)+\delta^\mathbf{z'}_1\right)
\cdot
\cos\left(\frac{x_1-1}{n_1-1}\left(z"_1 \pi -2\delta^\mathbf{z"}_1\right)+\delta^\mathbf{z"}_1\right)
$

The factor $((z'_1+z"_1+1) \mod 2)$ occurs because 
when $z'_1+z"_1$ is odd, then 
$ \cos\left(\frac{n_1-1}{n_1-1}\left(z'_1 \pi -2\delta^\mathbf{z'}_1\right)+\delta^\mathbf{z'}_1\right)
\cdot
\cos\left(\frac{n_1-1}{n_1-1}\left(z"_1 \pi -2\delta^\mathbf{z"}_1\right)+\delta^\mathbf{z"}_1\right)$
$= \cos\left(z'_1 \pi -\delta^\mathbf{z'}_1\right) 
\cdot
 \cos\left(z"_1 \pi -\delta^\mathbf{z"}_1\right) 
$ 
$= - \cos\left( \delta^\mathbf{z'}_1\right) 
\cdot
 \cos\left( \delta^\mathbf{z"}_1\right) 
$.

To complete the proof of  the Theorem  \ref{thW:normalizedLap}, we need to show only  that the ensuing relation holds:
\begin{equation}\label{eqW:zerosumcoscos}
\sum_{x_j=1}^{n_j}
 \cos\left(\frac{x_j-1}{n_j-1}\left(z'_j \pi -2\delta^\mathbf{z'}_j\right)+\delta^\mathbf{z'}_j\right)
\cdot
\cos\left(\frac{x_j-1}{n_j-1}\left(z"_j \pi -2\delta^\mathbf{z"}_j\right)+\delta^\mathbf{z"}_j\right) =0
\end{equation}

But this has already been demonstrated within the proof of  the  Theorem  \ref{thW:normalizedLap} 
- as that proof above equation does not refer to weights (that is elimination of $\delta$s does not require any reference to $\lambda$s.

\subsection{Grid Graphs without Inner Nodes}\label{subsecW:generalNoInner}

We assumed so far, that for each dimension  $i$
its $n_i\ge 3$. What what if this is not the case. 
Consideration of $n_i=1$ is pointless because we would just throw away such a dimension. So let us discuss   graphs with one or more 
$n_i=2$.

 We have shown in the proof of the Theorem \ref{thW:normalizedLap}, that 
the eigenvalue can be expressed in the form 
\eref{eqW:lambdaN2dm1},
given  the form \eref{eqW:eigenvectorcomponentN} of eigenvectors. 
Thence   it is a matter of an excercise to prove  that the same will hold in case of graphs without an inner node. 

\begin{theorem}
The formulas of the 
 Theorem  \ref{thW:normalizedLap} apply also to weighted grid graphs without inner nodes. Eigenvalues and eigenvectors are the same. 
\end{theorem}

No difference occurs in the representation of eigenvalues and eigenvectors. Therefore  the  method of computation of $\lambda$s and $\delta$s   can be  reused here.

\section{Random Walk Laplacians of Weighted Grid Graph}\label{secW:RWLtheorems}
As already mentioned, the eigenvalues and eigenvectors for Random Walk Laplacians could  be conviniently derived from those for Normalized Laplacians (compare Section \ref{sec:notation}. 
Thus 

\begin{theorem} \label{thW:randomwalkLap}
The random walk   Laplacian  $\mathbb{L}$ of
a weighted $d$-dimensional grid with at least one inner node, 
has  the 
  eigenvalues  of the   form 
\begin{equation}\label{eqW:lambdaRW}
\lambda_{\mathbf{z}}=1+ \sum_{j=1}^d  \frac{\waga_j}{\waga_\Sigma}
\cos\left(\frac{1}{n_j-1}\left(z_j \pi -2\delta_j\right) \right)
\end{equation}
with the ${\boldsymbol\delta}^\mathbf{z}$ vector defined 
as a solution of the equation system consisting of 
  the preceding equation \eref{eqW:lambdaN2d} and the 
 equations \eref{eqW:lambdaN2dm1} for each $l=1,\dots,d$. 
The corresponding eigenvectors $\mathbf{v}_{\mathbf{z}}$ have components of the form  
\begin{equation}\label{eqW:eigenvectorcomponentRW}
\nu_{\mathbf{z},[x_1,\dots,x_d]}= 
D_{[x_1,\dots,x_d],[x_1,\dots,x_d]}
\prod_{j=1}^d (-1)^{x_j} 
\cos\left(\frac{x_j-1}{n_j-1}\left(z_j \pi -2\delta^\mathbf{z}_j\right)+\delta^\mathbf{z}_j\right)
\end{equation} 
\end{theorem}

\section{Some Properties of Laplacians  of weighted grid graphs}\label{sec:weightedLproperties}

\begin{figure}
\centering
  \includegraphics[width=0.9\textwidth]{\figaddr{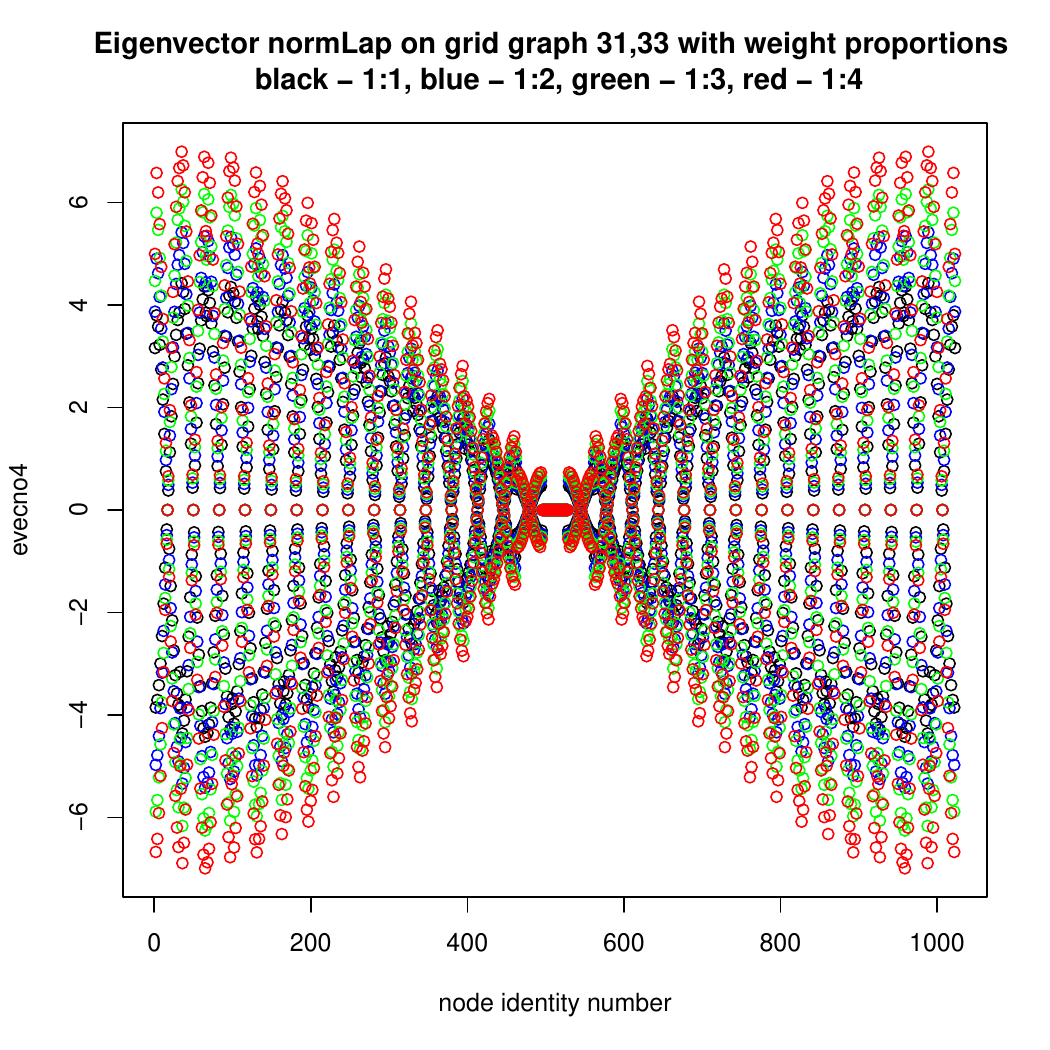}}  %
\caption{ 
The plots of sample  eigenvectors of normalized Laplacians of weighted two-dimensional grid graphs of approximately 1,000 nodes for $\z=[1,1]$ with various propertions of weights in both directions.   
Colors indicate:
black - 1:1 (the unweighted case),
blue - 1:2, green - 1:3, red - 1:4. 
}\label{fig:w_evecno}
\end{figure}

We refrained from drawing sample eigenvctors of combinatorial Laplacians of sample unweighted eigenvectors. 
Instead, in 
  Figure \ref{fig:w_evecno} you see sample eigenvectors of the 
very same two-dimensional grid (32x32 nodes) but for different weight propertions between dimensions. As one can see, higher discrepances between the weights in different directions lead to broader spreading of the values of components of the eigenvector. 
  The effect of the  cosine product can be seen anyway. 

\begin{figure}
\centering
  \includegraphics[width=0.9\textwidth]{\figaddr{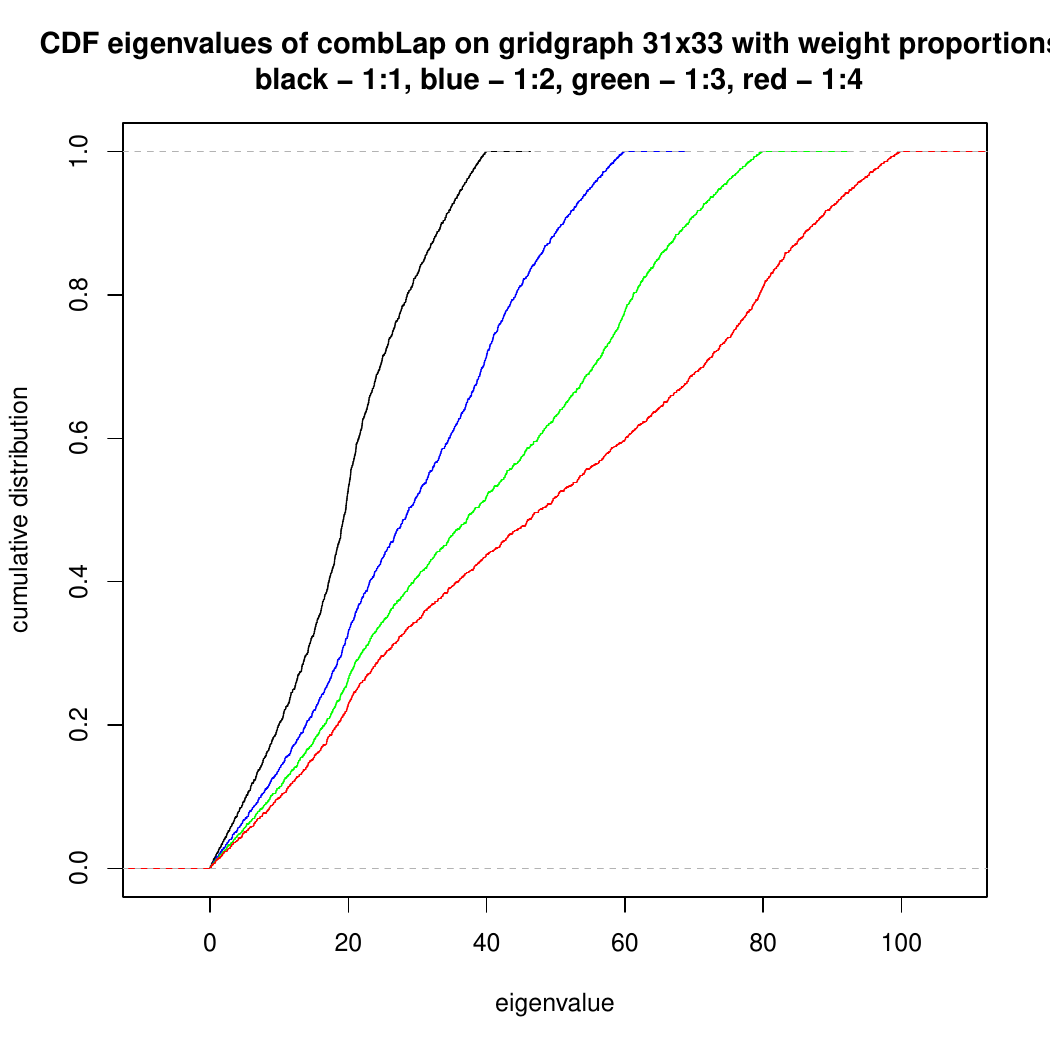}}  %
\caption{ 
The cumulative distributions of  eigenvalues of combinatorial Laplacians of grid graphs of approximately 1,000 nodes  with various propertions of weights in both directions.   
Colors indicate:
black - 1:1 (the unweighted case),
blue - 1:2, green - 1:3, red - 1:4. 
}\label{fig:w_evco_dist}
\end{figure}

The cumulative distribution function of eigenvalues of combinatorial Laplacian of 2-dimensional weighted grid graphs with varying weight proportions, as visible in 
\ref{fig:w_evco_dist}, departs from uniformity more and more as the weights become more unbalanced. 

\begin{figure}
\centering
  \includegraphics[width=0.9\textwidth]{\figaddr{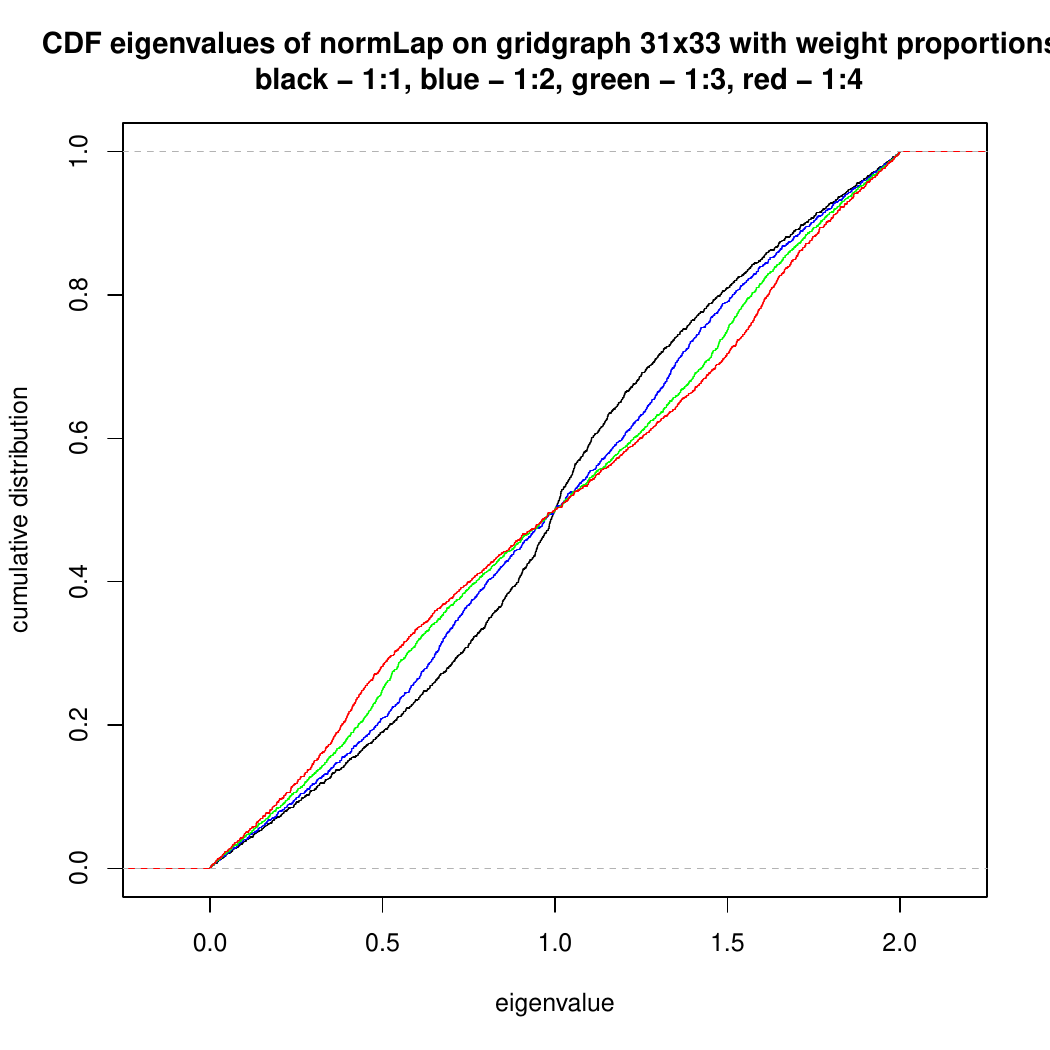}}  %
\caption{ 
The cumulative distributions of  eigenvalues of normalized Laplacians of grid graphs of approximately 1,000 nodes  with various propertions of weights in both directions.   
Colors indicate:
black - 1:1 (the unweighted case),
blue - 1:2, green - 1:3, red - 1:4. 
}\label{fig:w_evno_dist}
\end{figure}

This departure is even better visible for normalized Laplacians  as visible in 
\ref{fig:w_evno_dist}.

\begin{figure}
\centering
  \includegraphics[width=0.9\textwidth]{\figaddr{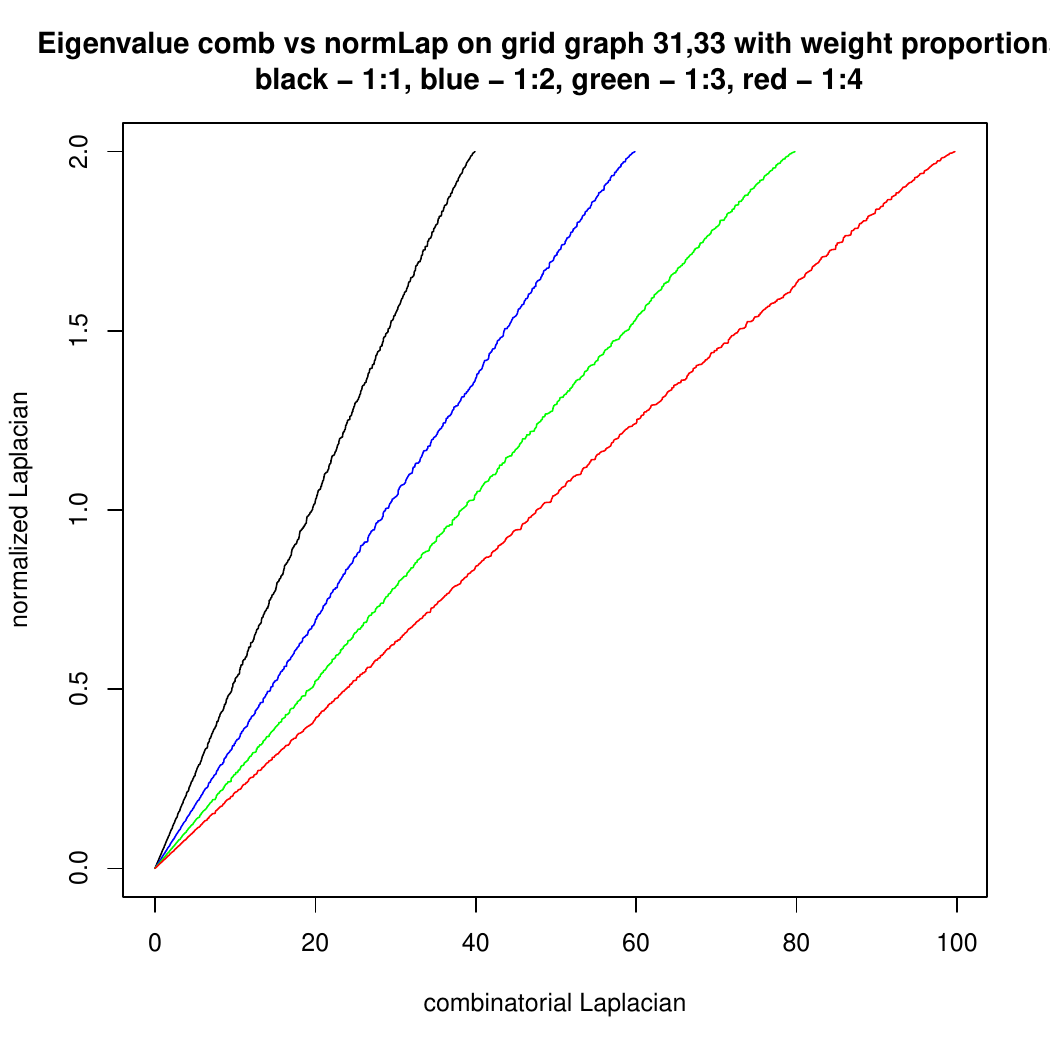}}  %
\caption{ 
A comparison of eigenvalue distributions between combinatorial and normalized Laplacians, for structurally the same grid graph, but  
  with various propertions of weights in both directions.   
Colors indicate:
black - 1:1 (the unweighted case),
blue - 1:2, green - 1:3, red - 1:4. 
}\label{fig:w_evco_evno}
\end{figure}

It is also worth having a look at the comparison of aligned eigenvalues of combinatorial and normalized Laplacians, as visible in Figure \ref{fig:w_evco_evno}.
Though they appear to be nearly placed on a straight line, they are not, they lie above it in the middle.

\begin{figure}
\centering
  \includegraphics[width=0.9\textwidth]{\figaddr{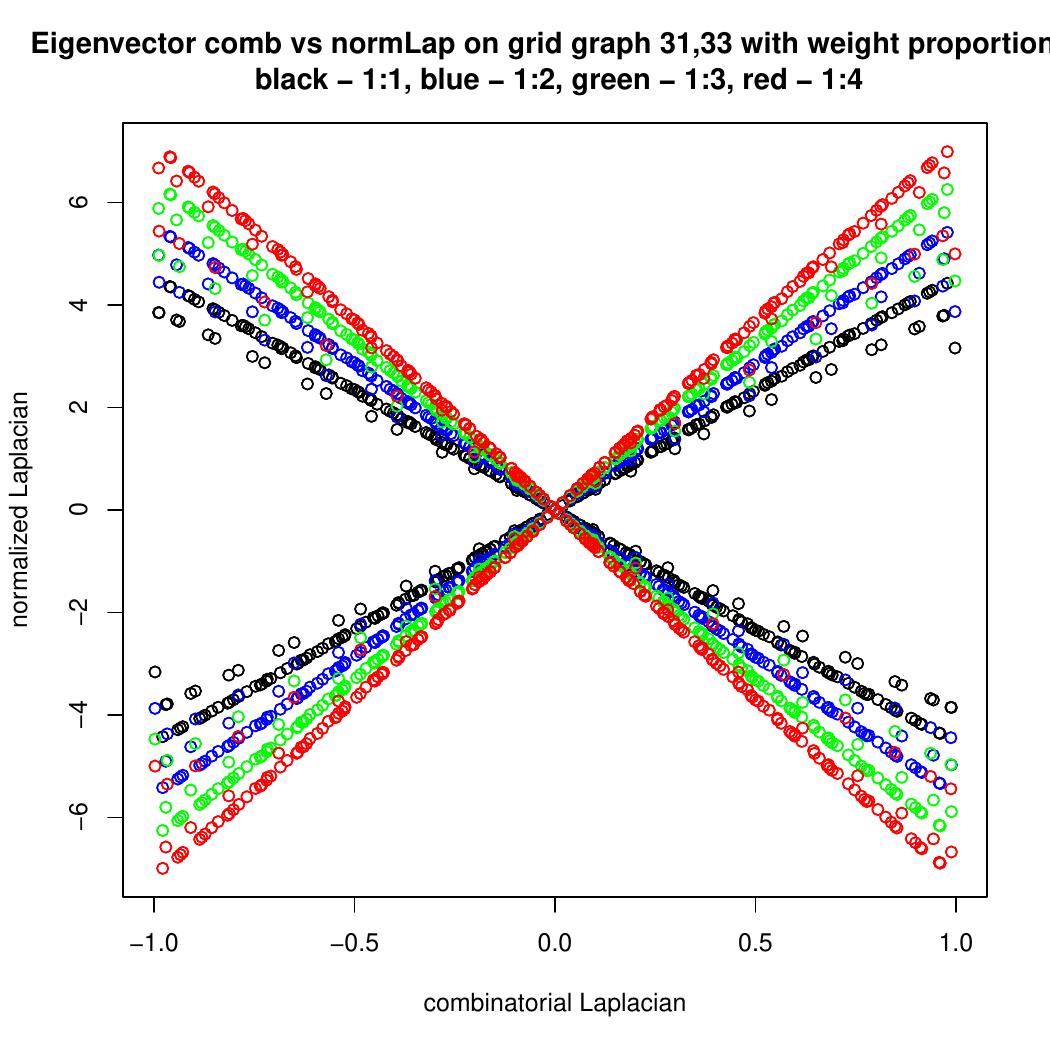}}  %
\caption{ 
A comparison of  sample  eigenvectors of combinatorial and  normalized Laplacians of weighted two-dimensional grid graphs of approximately 1,000 nodes for $\z=[1,1]$ with various proportions of weights in both directions.   
Colors indicate:
black - 1:1 (the unweighted case),
blue - 1:2, green - 1:3, red - 1:4. 
}\label{fig:w_evecco_evecno}
\end{figure}

In Figure \ref{fig:w_evecco_evecno}, one eigenvector for combinatorial and normalized Laplacians is compared for each weighting of edges. They exhibit similar patterns, withv weights being responsible for some spreading of the values.

\begin{figure}
\centering
  \includegraphics[width=0.9\textwidth]{\figaddr{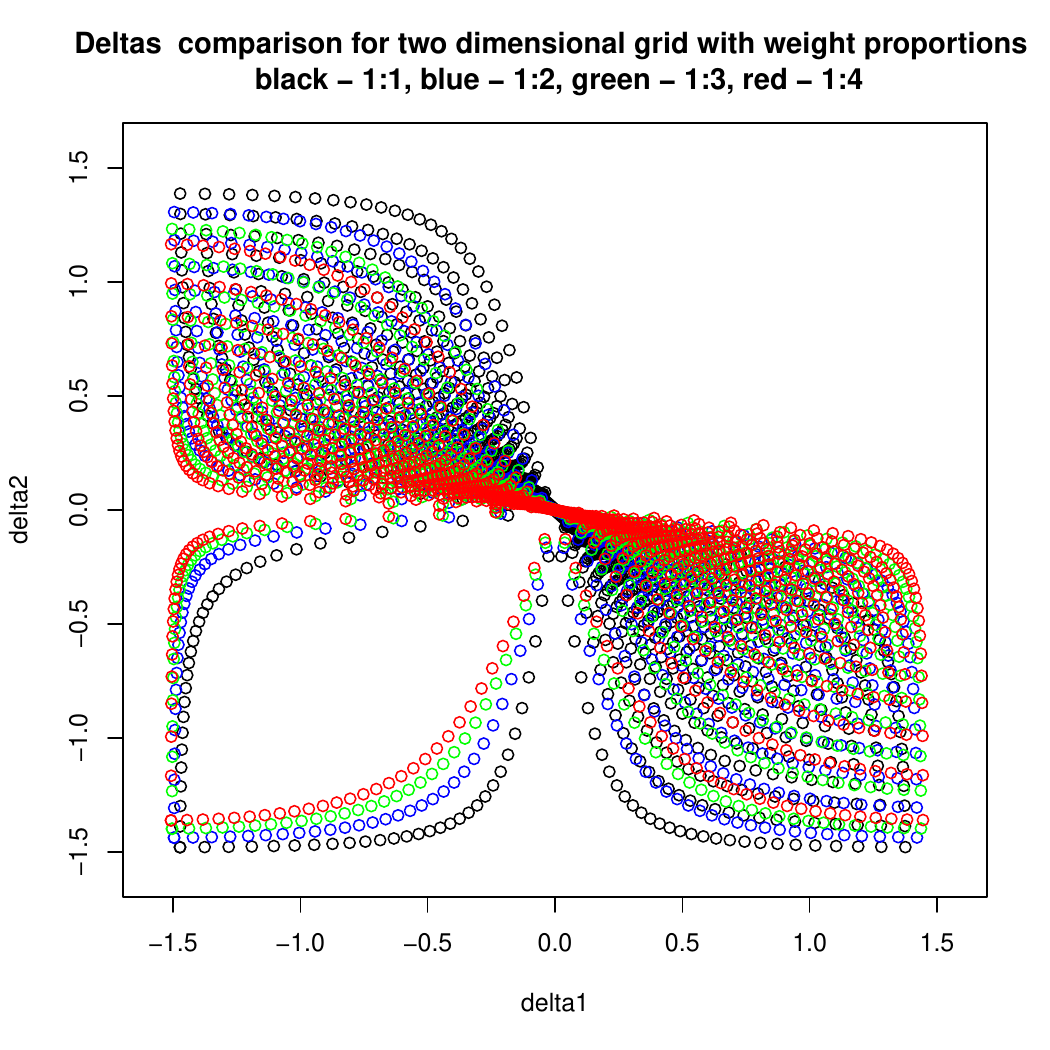}}  %
\caption{ 
The plots of relationships of   normalized Laplacian eigenvalues
$\lambda$ 
 and shifts  $\delta$of grid graphs of approximately 1,000 nodes  with various propertions of weights in both directions.   
Colors indicate:
black - 1:1 (the unweighted case),
blue - 1:2, green - 1:3, red - 1:4. 
}\label{fig:w_evno_delta}
\end{figure}

Figure \ref{fig:w_evno_delta} illustrates the relationship between eigenvalues and the shifts of normalized Laplacians in grid graphs.
This relationship seems not to be simplistic and may at least partially explain why we did not find a closed-form solution  for identifying eigenvalues and shifts. 
We see, however, that the patterns are similar for various proportions of weights of edges of the grid graph.




\Bem{
\subsection*{Software}
Please feel free to experiment with an R package (source code)   demonstrating computation of the eigenvalues and the eigenvectors for the mentioned Laplacians from closed-form or nearly closed-form formulas. It is available at 
\url{install.packages("http://www.ipipan.waw.pl/staff/m.klopotek/ipi_archiv/WeightGridGraph_1.0.tar.gz", repos = NULL, type ="source")}
\subsection*{Acknowledgement}
This research was funded by Polish goverment budget for scientific research. 
}

\newcommand{\bwaga}{\mathfrak{v}}

\section{Biweighted Grid Graphs}\label{secW:biweighted}

Let us define a biweighted  
  generalized grid graph   as 
$G_{(n_1)(\waga_1)(\bwaga_1))}$ being a biweighted path graph of $n_1$ vertices with weight $\waga_1$ for any link in this graph
from an odd node $i$ to the even node $i+1$
and with weight $\bwaga_1$ for any link in this graph
from an odd node $i$ to the even node $i-1$
,
and the $d$ dimensional biweighted grid graph
$G_{(n_1,\dots,n_{d})(\waga_1,\dots,\waga_{d})(\bwaga_1,\dots,\bwaga_{d})}$ being 
the weighted graph Cartesian product 
$G_{(n_1,\dots,n_{d-1})(\waga_1,\dots,\waga_{d-1})(\bwaga_1,\dots,\bwaga_{d-1})}  \times   G_{(n_d)(\waga_d)(\bwaga_d)}$

Thus a $d$-dimensional biweighted grid graph is uniquely defined by a \emph{biweighted grid graph identity vector pair} 
$[n_1,\dots,n_d][\waga_1,\dots,\waga_d][\bwaga_1,\dots,\bwaga_d]$ where $n_j$ is the number of layers in the $j$th dimension and   $\waga_j,\bwaga_j$ are  the alternating weights of links between layers in the $j$th dimension.  
Iinteger  identities
to nodes are assigned as in 
weighted grid graph. 

\subsection{Eigensolutions of Combinatorial Laplacians of Biweighted Grid Graphs}\label{secW:CLdoubleweighted}

Let us consider first biweighted grid path. 
The biweighted grid graph treatment, like in the case of weighted grid graph  is the product of  biweighted grid paths. 

Our working hypothesis is that the eigenvectors are of the form 
$(-1)^{nodeid}\sin(\alpha)$, and the angles $\alpha$ differ between neighbouring  nodes by either $\delta_\waga$ or $\delta_\bwaga$. 
Hence upon multiplication of the Laplacian matrix with the eigenvector the result for  an non-border odd node 
would be 
$$\bwaga \left( \sin(\alpha)+\sin(\alpha-\delta_\bwaga)\right)
+\waga \left( \sin(\alpha)+\sin(\alpha+\delta_\waga)\right)
$$
and for the even non-border nodes of the form 
$$\waga \left( \sin(\alpha)+\sin(\alpha-\delta_\waga)\right)
+\bwaga \left( \sin(\alpha)+\sin(\alpha+\delta_\bwaga)\right)
$$
Let us consider the odd nodes 
$$\bwaga \left( \sin(\alpha)+\sin(\alpha-\delta_\bwaga)\right)
+\waga \left( \sin(\alpha)+\sin(\alpha+\delta_\waga)\right)
$$
$$=\bwaga \left( \sin(\alpha)+\sin(\alpha)\cos(\delta_\bwaga)-\cos(\alpha)\sin(\delta_\bwaga)\right)
+\waga \left( \sin(\alpha)+\sin(\alpha)\cos(\delta_\waga)+\cos(\alpha)\sin(\delta_\waga)\right)
$$
$$=\sin(\alpha) 
\left( 
	\bwaga \left(1+ \cos(\delta_\bwaga)\right)
	+\waga \left(1+ \cos(\delta_\waga)\right)
\right)
+\cos(\alpha)\left(-\bwaga\sin(\delta_\bwaga)
+\waga\sin(\delta_\waga)
\right)
$$
Obviously, the property of eigenvector requires that 
$-\bwaga\sin(\delta_\bwaga)
+\waga\sin(\delta_\waga=0$, which will imply an eigenvalue of
$ 
	\bwaga \left(1+ \cos(\delta_\bwaga)\right)
	+\waga \left(1+ \cos(\delta_\waga)\right)
$.
Same final result is achieved if we consider the even nodes.

In order to ensure that also the eigenproperty holds at the end points of the path, we have to take $\alpha=\frac12\delta_\bwaga$ for the first node. because in this case the result of the product with the Laplacian marix will amount to:
$$ \waga \left( \sin(\frac12\delta_\bwaga)+\sin(\frac12\delta_\bwaga+\delta_\waga)\right)
$$
$$=\bwaga \left( \sin(\frac12\delta_\bwaga)+\sin(\frac12\delta_\bwaga-\delta_\bwaga)\right)
+\waga \left( \sin(\frac12\delta_\bwaga)+\sin(\frac12\delta_\bwaga+\delta_\waga)\right)
$$
so it is immediately visible that we have the same "eigenvalue"
if the condition 
$-\bwaga\sin(\delta_\bwaga)
+\waga\sin(\delta_\waga)=0$ holds. 
By analogy, the last node would have an $\alpha$ 
either $\frac12\delta_\bwaga$ or  $\frac12\delta_\waga$ before a zero point of the $\sin$ function to apply the same trick. 
In summary, if $n$ is the number of nodes on the path, 
the eigenvalues are of the form 
$$\lambda_{[z]}=
	\bwaga \left(1+ \cos(\delta_\bwaga)\right)
	+\waga \left(1+ \cos(\delta_\waga)\right)
$$
and eigenvectors have components ($x$ - the node id)
$$ \nu_{[z],[x]}= 
\sin\left(
\frac12\delta_\bwaga+
+((x-1-(x-1)\mod  2)/2)\delta_{\waga\bwaga}+((x-1)\mod 2)\delta_\waga
\right)
$$
where $\delta_{\waga\bwaga}=\delta_{\waga}+\delta_{\bwaga}$ 
subject to $z\pi=\frac12 n \delta_{\waga\bwaga}$ and $-\bwaga\sin(\delta_\bwaga)
+\waga\sin(\delta_\waga)=0$ 
Here $z$ is ranging from $1$ to $n$. 

Note that 
$$-\bwaga\sin(\delta_\bwaga)
+\waga\sin(\delta_\waga)=0$$
means that 
$$-\bwaga\sin(\delta_{\waga\bwaga}-\delta_\waga)
+\waga\sin(\delta_\waga)=0$$
$$-\bwaga\sin(\delta_{\waga\bwaga})\cos( \delta_\waga)
+\bwaga\cos(\delta_{\waga\bwaga})\sin( \delta_\waga)
+\waga\sin(\delta_\waga)=0$$
$$-\bwaga\sin(\delta_{\waga\bwaga})\cos( \delta_\waga)
+\sin(\delta_\waga)\left(\bwaga\cos(\delta_{\waga\bwaga}) 
+\waga\right)=0$$
$$
+\sin(\delta_\waga)\left(\bwaga\cos(\delta_{\waga\bwaga}) 
+\waga\right)=\bwaga\sin(\delta_{\waga\bwaga})\cos( \delta_\waga)$$
$$
\tan(\delta_\waga)=\frac{\bwaga\sin(\delta_{\waga\bwaga})}{\bwaga\cos(\delta_{\waga\bwaga}) 
+\waga }$$
$$
\delta_\waga=\arctan \frac{\bwaga\sin(\delta_{\waga\bwaga})}{\bwaga\cos(\delta_{\waga\bwaga}) 
+\waga }$$

From computational point of view we compute 
first 
$$ \delta_{\waga\bwaga}=\frac{z\pi}{\frac12 n}$$
Then 
$$
\delta_\waga=\arctan \frac{\bwaga\sin(\delta_{\waga\bwaga})}{\bwaga\cos(\delta_{\waga\bwaga}) 
+\waga }$$
Then 
$$\delta_\bwaga=\delta_{\waga\bwaga}-\delta_\waga$$
Then we substitute for eigenvalue and eigenvectors.
We need to distinuish the special cases: (1) when $n$ is even and 
$z=n/2$, then 
$ \delta_{\waga\bwaga}=\pi, \delta_{\waga}=0, \delta_{\bwaga}=\pi$, 
(1) when $z=n$   
 then 
$ \delta_{\waga\bwaga}=2\pi, \delta_{\waga}=\pi, \delta_{\bwaga}=\pi$, 
(3) otherwise if we obtain $\delta_{\waga}<0$, we need to update it to
$\delta_{\waga}:=\delta_{\waga}+\pi$. 

For multidimensional grids the eigenvalues are sums of component eigenvalues and the eigenvector components are products of component eigenvector components, like in case of weighted eigenvectors and eigenvalues. 

Note that contrary to  unweighted and weighted grids, the eigenvector componets for biweighted grids depend on the weights.

\subsection{Eigensolutions of Unoriented Laplacians of Biweighted Grid Graphs}\label{secW:CLdoubleweighted}
These are easily derived from the combinatorial Laplacian eigenvalues (identical) and eigenvectors (with alternating signs of components), just like in the unweighted and singly weighted case. 
\subsection{A Note on Eigensolutions of Normalized Laplacians of Biweighted Grid Graphs}\label{secW:CLdoubleweighted}

Let us consider first biweighted grid path. 
The biweighted grid graph treatment, unlike in the case of combinatorial Laplacian, does nmot generalize to the multidimensional case, but nonetheless provides with some insights. 

By analogy to the weighted and unweighted case, we do not consider the Laplacian, but rather its transform as previously.

Our working hypothesis is that the eigenvectors are of the form 
$(-1)^{nodeid}\cos(\alpha)$, and the angles $\alpha$ differ between neighbouring  nodes by either $\delta_\waga$ or $\delta_\bwaga$. 
Hence upon multiplication of the modified Laplacian matrix with the modified eigenvector the result for  an non-border odd node 
would be 
$$\bwaga \left( \cos(\alpha)+\cos(\alpha-\delta_\bwaga)\right)
+\waga \left( \cos(\alpha)+\cos(\alpha+\delta_\waga)\right)
$$
and for the even non-border nodes of the form 
$$\waga \left( \cos(\alpha)+\cos(\alpha-\delta_\waga)\right)
+\bwaga \left( \cos(\alpha)+\cos(\alpha+\delta_\bwaga)\right)
$$

Let us consider the odd nodes 
$$(\waga+\bwaga)\lambda = \bwaga \left( \cos(\alpha)+\cos(\alpha-\delta_\bwaga)\right)
+\waga \left( \cos(\alpha)+\cos(\alpha+\delta_\waga)\right)
$$
$$=\bwaga \left( \cos(\alpha)+\cos(\alpha)\cos(\delta_\bwaga)+\sin(\alpha)\sin(\delta_\bwaga)\right)
+\waga \left( \cos(\alpha)+\cos(\alpha)\cos(\delta_\waga)-\sin(\alpha)\sin(\delta_\waga)\right)
$$
$$=\cos(\alpha) 
\left( 
	\bwaga \left(1+ \cos(\delta_\bwaga)\right)
	+\waga \left(1+ \cos(\delta_\waga)\right)
\right)
+\sin(\alpha)\left(\bwaga\sin(\delta_\bwaga)
-\waga\sin(\delta_\waga)
\right)
$$
Obviously, the property of eigenvector requires that 
$\bwaga\sin(\delta_\bwaga)
-\waga\sin(\delta_\waga=0$, as for combinatorial Laplacian, which will imply an eigenvalue of
$ 
	\left(\bwaga \left(1+ \cos(\delta_\bwaga)\right)
	+\waga \left(1+ \cos(\delta_\waga)\right)\right)/(\waga+\bwaga)
$.

We need to compensate at the ends of the path for 
one of the weights is missing.
So consider the first element
$$ \waga  \lambda =   \waga \left( \cos(\alpha)+\cos(\alpha+\delta_\waga)\right)
$$
$$ \left(\bwaga \left(1+ \cos(\delta_\bwaga)\right)
	+\waga \left(1+ \cos(\delta_\waga)\right)\right)/(\waga+\bwaga) =     \left( \cos(\alpha)(1+\cos(\delta_\waga))-\sin(\alpha)\sin(\delta_\waga)\right)
$$

$$  =   \waga \left( \cos(\alpha)+\cos(\alpha)\cos(\delta_\waga)-\sin(\alpha)\sin(\delta_\waga)\right)
$$ 

As we see, the issue gets quite complex even for the path graph, though it is solvable. 
Therefore, derivation of analytical formulas in this case does not make much sense for general biweighted  grid graphs.

\section{Conclusions}\label{sec:conclusions}

In this paper we have presented a (closed-form or nearly-closed form) method of computation of all eigenvalues and eigenvectors 
of a multi-dimensional grid graph for  unnormalised, unoriented, normalized and random walk Laplacians.%
We considered both unweighted and weighted grid graphs, using simple (regular) weighting scheme that is one weight for one direction.

Their properties may be of interest as generalisations of results of other authors discussed in the Introduction. 
Furthermore, note that the multidimensional grid graphs are bipartite graphs
so that they may be exploited in the investigations of properties of Laplacians of bipartite graphs.

In particular, for unweighted grid graphs, one sees that the principal eigenvalue of a $d$-dimensional grid graph is limited from above by 
$4d$ for unnormalized Laplacians and the biggest eigenvalue for normalized and random walk Laplacians is equal 2. 
The Fiedler eigenvalue on the other hand approaches zero with the increase of the number of nodes in such a graph. 

The closed-form or nearly closed-form formulas for eigenvalues and eigenvectors for multidimensional unweighted grid graphs may be of high interest for researchers dealing with cluster analysis of graphs, in particular with spectral cluster analysis, especially compressive spectral clustering (CSC), which are essentially exploiting implicitly or explicitly the eigenvalues and eigenvectors.   
First of all because unweighted grid graphs can be considered as types of graphs that have no intrinsic cluster structure. Hence the spectral clustering algorithms should be checked against such structures getting advantage of the fact that the eigenvectors and eigenvalues are quite easy to obtain even for large graphs. 
But more important is the fact that \emph{eigenvalues of grid graphs are not uniformly distributed}. 
\textbf{
This implies that the theoretical foundations of CSC need to be thoroughly verified.  
}. This can be considered as the most important insight gained by this investigation.

The closed-form or nearly closed-form formulas for eigenvalues and eigenvectors for multidimensional weighted grid graphs 
may be of high interest in this context.
The weighted grid graphs can be considered as types of graphs that have either no intrinsic cluster structure (when the weights are equal)  or the structure of which can be twisted in various ways. 
 The weights permit to simulate node clusters not perfectly separated from each other, with various shades of this imperfection. 
This fact opens new possibilities for exploitation of closed-form or nearly closed form solutions 
eigenvectors and eigenvalues of graphs while testing and/or developing such algorithms and exploring their theoretical properties.

It shall be noted that the eigenvectors of combinatorial and unoriented Laplacians of weighted graphs 
are identical with those for unweighted graphs. 
In case of normalized and random walk Laplacians, the eigenvectors for weighted graphs are superficially identical with those for unweighted ones, but they differ nevertheless because the shifts  are influenced by weights. 

The study of differences between the weighted and unweighted case allow for new insights into the nature of normalized and unnormalized Laplacians.
Edge weights have no impact on the eigenvectors of combinatorial and undirected Laplacians. Only the presence or absence of an edge impacts them. This is not the case with normalized and random walk Laplacians. Here the relative edge weights influence the shifts  in the vector formulas. 
The weighting scheme opens up the possibility of manipulating of the magnitude of eigenvalues of combinatorial and unordered Laplacians, related to various grid dimensions. 
This has an interesting impacty for example on the concept of the Fiedler vector, associated with the second lowest eigenvalue. with the weight changes we can modify our preferences over which eigenvector to choose as Fiedler vector (among those with lowest components). We can change  the order of magnitude of eigenvectors associated with some direction and observe the impact on $k$-means clustering in spectral graph analysis.
We have also the possibility to study the impact of relative weights of various dimensions in a grid graph on the normalized and random walk Laplacians, while for example the connection between various grid layers is fading. 

As increasing interest in weighted graph Laplacians exists, it would be an  interesting research topic to find also closed form solutions to Laplacians of weighted graphs with other weighting schemas than those assumed in this work.

\subsection*{Software}

Please feel free to experiment with an R package (source code)   demonstrating computation of the eigenvalues and the eigenvectors for the mentioned Laplaacians from closed-form or nearly closed-form formulas for unweighted grid graphs. It is available at 
\url{install.packages("http://www.ipipan.waw.pl/staff/m.klopotek/ipi_archiv/GridGraph_1.0.tar.gz", repos = NULL, type ="source")}

\subsection*{Acknowledgement}
This research was funded by Polish goverment budget for scientific research.

 \bibliographystyle{plain}
\bibliography{UnAndWeightedGridGraphLaplacianCUNR_bib}
\end{document}